\documentclass{article}

\RequirePackage{amsthm,amsmath,amsfonts,amssymb,mathtools}
\usepackage[bb=dsserif]{mathalpha} 
\usepackage{bm} 

\RequirePackage[numbers]{natbib}
\RequirePackage[colorlinks,citecolor=blue,urlcolor=blue]{hyperref}
\RequirePackage{graphicx}
\usepackage{verbatim}

\newcommand{\ZZ}{\mathbb{Z}}

\newcommand{\mgf}{\phi}
\newcommand{\bigmgf}{\Phi}
\newcommand\Bbbbone{%
  \ifdefined\mathbbb%
    \mathbbb{1}%
  \else%
    \boldsymbol{\mathbb{1}}%
  \fi}

\newcommand{\indicator}[1]{\Bbbbone_{ #1 }}

\DeclarePairedDelimiter\set{ \{ }{ \} }

\newtheorem{proposition}{Proposition}

\newtheorem{lemma}{Lemma}
\newtheorem{theorem}{Theorem}

\newtheorem{corollary}{Corollary}
\newtheorem{remark}{Remark}

\newenvironment{hyp}[1]
	{\begin{equation} \tag{#1} \begin{minipage}{0.9\linewidth}}
	{\end{minipage} \end{equation} \ignorespacesafterend}

\begin{document}

\title{Rare events in a polling system:\\ Rays \& Spirals}


\author{Robert D. Foley \\
ISYE Georgia Institute of Technology\\
rfoley@isye.gatech.edu\\
\\
and\\
\\
David R. McDonald\\
Mathematics and Statistics University of Ottawa\\
david.r.mcdonald@gmail.com
}

\date{\today}    
\maketitle


\noindent AMS Subject Classification: 60J10 (primary), 60J50 (secondary).


\begin{abstract}
It's a situation everyone dreads.  A road is down to one lane for repairs. Traffic is let through one way until the backlog clears and then
traffic is let through the other way to clear that backlog and so on.  When stuck in a very long queue it is inevitable to wonder {\em how did I get into this mess}?

We  study a polling model with a server having exponential service time with mean $1/\mu$ alternating between two queues, emptying one queue before switching to the other. Customers arrive at queue one according to a Poisson process with rate $\lambda_1$ and at queue two with rate $\lambda_2$. We discuss how we get at a rare event with a large number of
customers in the system. In fact this can happen in two different ways depending on the parameters. In one case one queue simply explodes and
runs away without emptying. We call this the ray case. In the other spiral case the queues are successively emptied but in a losing battle as the system zigzags to the rare event.
This dichotomy extends to the steady state distribution and leads to quite different asymptotic behaviour in the two cases.
\end{abstract}



\section{Tandem polling model}

We are interested in the way rare events  occur and where possible in estimating the steady state probability of these rare events  for Markov chains of the type that typically arise in modelling polling networks.  The probabilities of such rare events are usually difficult to obtain---simulation can be very slow.
Consider two queues labeled 1 and 2.   A Poisson stream of jobs  arrives with intensity $\lambda_1$ to queue~1 where jobs are served in order of arrival.
A Poisson stream of jobs  arrives with intensity $\lambda_2$ to queue~2. Hence the customer arrival rate is $\lambda=\lambda_1+\lambda_2$.  A single server services both queues
with exponential service time having mean $1/\mu$.  The server remains at a queue serving jobs until there are no more jobs to serve before switching to the other queue. If there are jobs present, the server is working.
For convenience and without loss of generality, we assume that $\lambda + \mu= 1$.  To avoid degenerate situations, we assume that
$\lambda_1 > 0$,   $\lambda_2 > 0$ and $\mu > 0$.

There is a vast literature on polling systems. The earliest reference is \cite{Takacs}  and a good review is found in \cite{Takagi}).
The paper \cite{Boxma} considers our model with the additional twist that if queue~2 is being served but queue~1 reaches a threshhold then the server preemptively switches back to serving
queue~1. \cite{Borst} deals with our exhaustive service model as well as the threshhold model. Both papers are able to give the joint generating function of the steady state joint queue length distribution.
Both give results on mean queue sizes in steady state and other interesting relationships but don't say
much about large deviations of the queues.

The state of the system is denoted by $(x,y,s)$ where $(x,y)$ is the joint queue length of queues 1 and 2 and $s \in \{1,2\}$ is the queue being served.  When $(x,y) = (0,0)$, then $s = 1$.  Note that $(x,0,2)$ for $x\geq 0$ and $(0,y,1)$ for $y \geq 1$ are not in the state space.  Let $S$ denote the state space, and $Q$ the generator of, this Markov process $M$.  Under our assumptions, the Markov process is irreducible.  If the Markov process is positive recurrent, let $\pi$ denote the stationary distribution.
We may be interested in $\pi(x,y,1)$ or $\pi(x,y,2)$ where $x+y=\ell$ is a large integer and we are interested in the large deviation path of excursions from the origin to
points $(x,y,1)$ or $(x,y,2)$ such that $x+y=\ell$.  We shall see that depending on the parameters there are two distinct ways of temporarily overloading the queues. One way is for the  server to have an extraordinarily long busy period.  During this busy period  the server can't keep up with the arrivals to the queue it is serving so both queues get large together. This is the explosion or ray case.
  For other parameters the most likely path is a spiral
where the server alternates between emptying the two queues but in a losing battle.  Each time the server returns to a queue it is longer than before!
Such spirals have been observed to cause instability of multiclass reentrant networks even when the load on each server is less than one (see \cite{Bramson} and for an  review see \cite{Dai}.
The situations are of course different.  Our queues are stable but it is striking that the large deviations can occur because of the same spiral behaviour.

Large deviations of processes with boundaries have been studied extensively; see \cite{weiss} and \cite{Ellis-Dupuis}.
Large deviation results  for polling models  where the server has Markovian routing between queues are given in  \cite{DelcoigneI} and \cite{DelcoigneII}.
The local rate function is derived and a large deviation principle is established.  The technique can be used to estimate the (rough) stationary probability decay rate.
Large deviation results are also given in \cite{Asian} for polling systems where the server is routed between queues according to a Bernoulli service schedule.
This paper establishes a rate function and derives  upper and lower bounds on the probability that the
queue length of each queue exceeds a certain level (i.e., the buffer overflow probability).

A complementary (but far less general) approach using
harmonic functions was developed in \cite{PfH,FM,JSQ,FM:BridgesExact}. This technique enables one to obtain exact asymptotics of the stationary distribution of the queues.
In particular if one wishes calculate the steady state probability $\pi(F)$ for some rare event $F$ then one finds
 an $h$-transformation that converts
$M$ into a {\em twisted} chain $\mathcal{M}$ for which the rare event is no longer rare.
Equivalently one studies the (extreme) point on the Martin boundary of $M$ associated with the harmonic function $h$
whose twist drifts toward $F$.
Chang and Down \cite{Chang-Down-2007} used this approach to study polling models under limited service policies.
Our paper is complementary to theirs in that ours studies exhaustive service.

The  paper by  Ignatiouk-Robert \cite{Irina} explores the Martin boundary of random walks killed along a killing boundary.
Her paper has a
close connection to large deviation theory and to our techniques. Extending Ignatiouk-Robert's paper, \cite{Raschel} studies singular random walks like ours
but with absorption outside a cone. Positive harmonic functions are exhibited using the compensation approach \cite{AdanI,AdanII}.
Their cone of harmonic functions is much more complicated than our simple extension of the classical Ney and Spitzer results \cite{Ney-Spitzer}.
These papers and ours point to a circle of ideas (Martin boundary - large deviations - asymptotics of $\pi$)
with more open questions than answers.

One could quite rightly object that the service times of the cars released to drive on the one available lane are in no way exponential. The polling model studied is more appropriate for a server serving two queues of retail customer or computer jobs.  It is however the simplest model we know to illustrate the different ways such systems overload
so we will work it out in detail.
A more sensible model might be to take $\mu$ to be deterministic.  One might regard instead the state of the two queues $(x,y,1)$ or $(x,y,2)$
at service times. In this case there would be a random number of arrivals to each queue per service period with means $\lambda_1$ and $\lambda_2$ respectively.
To be even more realistic we include a switch over period when there are no services. This corresponds to the delay in getting traffic flowing in the reverse direction.
Indeed in \cite{Hofri} switch over times are included and  if one assigns holding costs it is shown that the exhaustive service policy is optimal for reducing long run discounted costs.
We will briefly study a more general model in Section \ref{trafficoverload}. The method and the main features are the same; i.e.
either the queue being served runs away to a large value or the server does succeed in successively emptying each queue only to find the other queue has built up to
a higher level than before.

\subsection{State space \& macro behavior:}
Let sheet $s$ be all states where the server is at queue $s$.  The projection of the two sheets onto the plane gives the joint queue length. When the server is at 2, the joint queue length is drifting southeasterly as can be seen from the possible customer flows and rates in the right sheet of Fig.~\ref{ratediagram}.  When the server is processing customers at 1, the  queue length process is drifting northwesterly  as can be seen in the left sheet of Fig.~\ref{ratediagram}.

   The two sheets look like  a coffee filter if the sides of the sheets' corresponding axes are glued together.
   When folded flat, the two creases line up along each axis, and the point of the filter is at the origin.  The left sheet in Fig.~\ref{ratediagram}, which includes the $x$-axis, corresponds to the server at queue~1; the right sheet in Fig.~\ref{ratediagram}, which includes the $y$-axis but not the origin, corresponds to the server at 2.
Whenever the Markov process crosses the crease going from one sheet to the other, the server  moves to the other queue.  Crossing a crease is like going through a turnstile since the process can only cross in one direction.  Along the $x$-axis, the server can move from queue~2 to queue~1.  Along the $y$-axis, the server can move from queue~1 to queue~2.

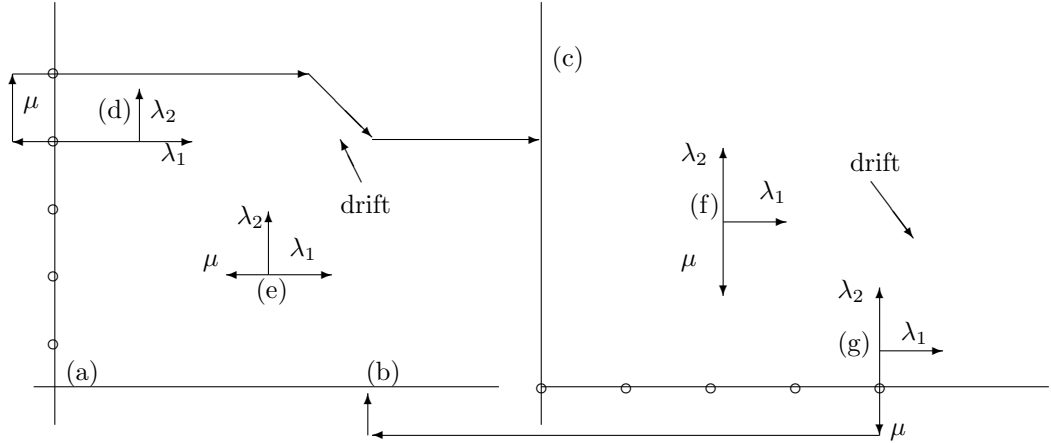
\begin{figure}
\hspace*{7cm} \setlength{\unitlength}{.8pt}
\begin{picture}(200,240)
\put(-240,3) {\line(1,0){220}}
\put(0,3) {\line(1,0){240}}
\put(-225,6){(a)}

\put(-230,-15) {\line(0,1){200}}
\put(0,-15) {\line(0,1){200}}

\multiput(-3,-1)(40,0){5}{$\circ$}
\multiput(-234,20)(0,32){5}{$\circ$}

\put(70,85){(f)}
 \put(86,81){\vector(1,0){30}}
 \put(103,92){$\lambda_1$}
 \put(86,81){\vector(0,-1){35}}
 \put(66,60){$\mu$}
\put(86,81){\vector(0,1){35}}
 \put(66,110){$\lambda_2$}

\put(-210,130){(d)}
\put(-250,119) {\vector(0,1){32}}
\put(-250,151) {\vector(1,0){140}}
\put(-110,151) {\vector(1,-1){30}}
\put(-80,120) {\vector(1,0){79}}

\put(-190,119) {\vector(-1,0){60}}
 \put(-245,135){$\mu$}
 \put(-180,110){$\lambda_1$}
\put(-190,119) {\vector(1,0){25}}
 \put(-185,130){$\lambda_2$}
\put(-190,119) {\vector(0,1){25}}

 \put(5,155){(c)}

\put(-135,45){(e)}
 \put(-129,56){\vector(1,0){30} }
 \put(-119,65){$\lambda_1$}
 \put(-129,56){\vector(-1,0){20} }
 \put(-160,60){$\mu$}
 \put(-129,56){\vector(0,1){30} }
 \put(-144,80){$\lambda_2$}

 \put(-95,85){\mbox{drift}}
  \put(-85,100){\vector(-1,2){10} }

  \put(146,105){\mbox{drift}}
  \put(156,100){\vector(3,-4){20} }

\put(-83,6){(b)}

\put(140,20){(g)}
\put(160,20){\vector(1,0){30}}
 \put(170,25){$\lambda_1$}
 \put(160,20){\vector(0,1){30}}
 \put(140,45){$\lambda_2$}
 \put(160,20){\vector(0,-1){40}}
\put(160,-20){\vector(-1,0){240}}
\put(-82,-20){\vector(0,1){20}}
 \put(165,-20){$\mu$}

\end{picture}
\setlength{\unitlength}{1pt} \caption{Jump rates when queue $1$ is served are on the left sheet and on the right sheet when queue $2$ is served. Open circles denote points {\em not} in the state space
} \label{ratediagram}
\end{figure}

\subsection{Stability  of the tandem polling model}

Each job brings an amount of work $1/\mu$.
Since the system is non-idling, in order that the load on the single server is less than its capacity we
assume that
\begin{eqnarray}\label{stability}
\lambda_1+\lambda_2 &<& \mu,
\end{eqnarray}
which is a necessary and sufficient conditions for stability and is equivalent to having a unique stationary distribution $\pi$.
In fact the total queue length $Z$ is an $M|M|1$ queue  with arrival rate $\lambda=\lambda_1+\lambda_2$ and service rate $\mu$.
This chain is stable if $\rho<1$ and
the steady state is $\pi(\ell)=(1-\rho)\rho^{\ell}$ for $\ell=0,1,\ldots$ where $\rho=\lambda/\mu$.
Clearly the total queue length is a stable chain if and only if $M$ is stable.
Moreover $\pi(0,0,1)=\pi(0)=1-\rho$ by definition (there is no state $(0,0,2)$).

\subsection{Uniformization}

It is convenient to study the Markov chains whose uniformization is $Z$ or $M$. Denote the associated
transition kernel by $K$ in both cases. The homogenization of the combined chain $Z$ has transition kernel $K$ on $S=\{0,1,2,\ldots\}$ where $K(u,u+1)=\lambda$ and
$K(u,u-1)=\mu$ for $u>0$ and $K(0,1)=\lambda$ and $K(0,0)=\mu$. $\pi(z)=(1-\rho)\rho^z$ is the associated stationary probability.
 The homogenization of $M$ has transition kernel $K$
\begin{equation} \label{K-defn}
{K}(u,v) =
\begin{cases}
\lambda_1 &\text{ if $u=(x,y,s)$ and $v=(x+1,y,s)$ }
\\
\lambda_2 &\text{ if $u=(x,y,s)$ and $v=(x,y+1,s)$ }
\\
\mu &\text{ if $u=(x,y,1)$, $x > 1$, $v=(x-1,y,1)$}
\\
\mu &\text{ if $u=(1,y,1)$,  $v=(0,y,2)$}
\\
\mu &\text{ if $u=(x,y,2)$, $y > 1$, $v=(x,y-1,2)$}
\\
\mu &\text{ if $u=(x,1,2)$,  $v=(x,0,1)$}
\\
\mu &\text{ if $u=(0,0,1)$ and $u = v$}
\\
0 &\text{ otherwise.}
\\
\end{cases}
\end{equation}
for all pairs of states $(u,v)$.


\subsection{Twisted $Z$}

Extend $K$ to $\ZZ$  giving the   free kernel $\overline{K}$:
 $\overline{K}(u,u+1)=\lambda$ and
$\overline{K}(u,u-1)=\mu$ for all $u$.
 Let
$G=\sum_{n=0}^{\infty}\overline{K}^d$  be the associated potential. Define $\blacktriangle$ to be $\{z:z\leq 0\}$ in $\ZZ$ and $\Delta=S\setminus\blacktriangle=\{0\}$. Let $\overline{K}_{\blacktriangle}$ be the taboo kernel $\overline{K}$ killed on $\blacktriangle$ and let
 $G_{\blacktriangle}$ be the associated potential.

$h(u)=\rho^{-u}$ is harmonic for $\overline{K}$. The associated $h$-transformed kernel is
$\mathcal{K}(u,v)=\overline{K}(u,v)h(v)/h(u)$; i.e.  $\mathcal{K}(u,u+1)=\lambda$ and
$\mathcal{K}(u,u-1)=\mu$ for all $u$.
 Let
$$\mathcal{G}(u,v)=\sum_{n=0}^{\infty}\mathcal{K}^n(u,v)=G(u,v)\frac{h(v)}{h(u)}$$  be the associated potential.
 Let $\mathcal{K}_{\blacktriangle}$ be the taboo kernel $\mathcal{K}$ killed on $\blacktriangle$ and let
 $\mathcal{G}_{\blacktriangle}$ be the associated taboo potential.

We now use the representation
\begin{eqnarray}
\pi(z)&=&\sum_{u\in \Delta}\pi(u) G_{\blacktriangle}(u;z))
=\sum_{u\in \Delta}\frac{h(u)}{h(z)}\pi(u) \mathcal{G}_{\blacktriangle}(u;z))\nonumber\\
 &=&(1-\rho)\rho^{-z} \mathcal{G}_{\blacktriangle}(0;z)).
 \end{eqnarray}
 This just means ${\cal G}_{\blacktriangle}(0;z)=1$.
 The probability $\mathcal{Z}$, starting at $x>0$, escapes to infinity without returning to $0$ is
 $1-\rho^{x}$ so the probability of escaping from $0$ is $\mu(1-\rho^{1})=\mu-\lambda$. Taking $z=\ell$ this means
 means the expected number of visits to $\ell$ is $1/(\mu-\lambda)$ given the twisted chain does escape from $0$.

 \subsection{Twisted $M$}
For the uniformized {\em free} kernels $\overline{K}_1$ and $\overline{K}_2$
define the transforms $\bigmgf_s(\alpha,\beta)=\sum_{x,y}\overline{K}_s((0,0),(x,y)\alpha^x\beta^y$ for $s=1,2$.
Using convexity we see the two "eggs", $\bigmgf_1(\alpha,\beta)=1$ and $\bigmgf_2(\alpha,\beta)=1$, have two points of intersection;
i.e. $(1,1)$ and $(\rho^{-1},\rho^{-1})$.
Hence
$h(x,y,s)=\rho^{-(x+y)}$ is harmonic for $K$ away from $(0,0,1)$.
We may therefore calculate the $h$-transform. The transition kernel
of the resulting {\em twisted} chain $\mathcal{M}$ is:

\begin{eqnarray}\label{free}
\mathcal{K}(a,b) &=
\begin{cases}
\tilde{\lambda}_1=\frac{\lambda_1}{\lambda}\mu &\text{ if $a=(x,y,s)$ and $b=(x+1,y,s)$ } \\
\tilde{\lambda}_2=\frac{\lambda_2}{\lambda}\mu&\text{ if $a=(x,y,s)$ and $b=(x,y+1,s)$ } \\
\tilde{\mu}=\lambda &\text{ if $a=(x,y,1)$, $x > 1$, $b=(x-1,y,1)$} \\
\tilde{\mu}=\lambda&\text{ if $a=(1,y,1)$,  $b=(0,y,2)$} \\
\tilde{\mu}=\lambda&\text{ if $a=(x,y,2)$, $y > 1$, $b=(x,y-1,2)$} \\
\tilde{\mu}=\lambda &\text{ if $a=(x,1,2)$,  $b=(x,0,1)$} \\
\mu &\text{ if $u=(0,0,1)$ and $u = v$} \\
0 &\text{ otherwise.}
\end{cases}
\end{eqnarray}
Note that $\mathcal{K}$ is super-stochastic at $(0,0,1)$.

We note the above could be generalized. Suppose  the service rate is different for the two queues; i.e. $\mu_1$ on sheet~1 and $\mu_2$ on sheet $2$.
The points of intersection of the two "eggs": $\lambda_1\alpha+\lambda_2\beta+\mu_1 \alpha^{-1}=1$ and
 and $\lambda_1\alpha+\lambda_2\beta+\mu_2 \beta^{-1}=1$ gives a point $(\gamma_1,\gamma_2)$
 where
 $$\gamma_2=\frac{\mu_2}{\mu_1}\gamma_1\mbox{ and } \gamma_1=\frac{1-\sqrt{1-4(\lambda_1\mu_1+\lambda_2\mu_2)}}{\frac{2}{\mu_1}(\lambda_1\mu_1+\lambda_2\mu_2)}.$$
  Hence the $h(x,y)=\gamma_1^x \gamma_2^y$ is harmonic on both sheets away from $(0,0,1)$. We may use this harmonic function to obtain a transient twisted chain whose
  path gives the large deviation path of the original chain as above.
 Again the $h$-transformed chain can be a ray or a spiral.
 Then, by the above arguments, the large deviation paths are rays as discussed in Section \ref{rayraylab} or spirals discussed in Section \ref{zigzag}.

\subsection{Path of the rare event of hitting a high level}\label{pathrare}
We will now describe how $M$ hits the level $\ell$ at some point $(x,y,s)$; i.e. when $x+y=\ell$.
 Let $F=\{(x,y,s):x+y\geq\ell\}$ and $\alpha=\{(0,0,1)\}$.
 Let $\tau_F=\min\{n\geq 0 :M(n)\in F\}$ and $\tau_{\alpha}=\min\{n\geq 1 :M(n)=\alpha\}$
 Let $h_{\alpha}(x,y,s)=P_{(x,y,s)}[\tau_F<\tau_{\alpha}]$ with $h_{\alpha}(0,0,1)=0$ and $h_{\alpha}(p)=1$ if $p\in F$ where $p=(x,y,s)$.   $h_{\alpha}$ is harmonic
except on $F$.
Simple calculation shows $(\rho^{-(x+y)}-1)/(\rho^{-\ell}-1)$ also satisfies the same conditions and since two harmonic functions agreeing on their boundaries are equal
we know $h_{\alpha}(x,y,s)=(\rho^{-(x+y)}-1)/(\rho^{-\ell}-1)$.

As above define the twisted Markov chain $\mathcal{M}_{\alpha}$ with twisted kernel $$\mathcal{K}_{\alpha}(p,q)=K(p,q)h_{\alpha}(q)/h_{\alpha}(p)$$
for $p\notin \{\alpha\}\cup F$. Define $$\mathcal{K}_{\alpha}(\alpha,p)=K(\alpha,p)h_{\alpha}(p)/\sum_{q\neq (0,0,1)}K(\alpha,q)h_{\alpha}(q).$$ Let $\mathcal{P}^{h_{\alpha}}$
denote the probability measure constructed from $\mathcal{K}_{\alpha}$.
Since $h_{\alpha}(\alpha)=0$ the twisted chain cannot return to $\alpha$.  If follows that $\mathcal{M}_{\alpha}$ hits $F$ before $\alpha$ with probability one.


Pick a path $(a_0,a_1,\ldots a_{T-1},a_{T})$ starting at $a_0=\alpha=(0,0,1)$ that enters $F$  for the first
time at $a_{T}$ before returning to $\alpha$. We see
\begin{eqnarray}
\lefteqn{P_{\alpha}[M(n)=a_n, 0\leq n\leq T|\tau_F<\tau_{\alpha}]}\label{good}\\
&=&P_{\alpha}[M(n)=a_n, 0\leq n\leq T]/P_{\alpha}[\tau_F<\tau_{\alpha}]\nonumber\\
&=&\mathcal{P}^{h_{\alpha}}_{\alpha}[\mathcal{M}(n)=a_n, 0\leq n\leq T]h_{\alpha}^{-1}(a_{T})/\mathcal{E}_{\alpha}^{h_{\alpha}}[\chi\{\tau_F<\tau_{\alpha}\}h_{\alpha}^{-1}(M(\tau_F))]\nonumber\\
&=&\mathcal{P}^{h_{\alpha}}_{\alpha}[\mathcal{M}(n)=a_n, 0\leq n\leq T]\label{gooder}
\end{eqnarray}
since $h_{\alpha}$ equals $1$ on $F$ and since $\mathcal{M}$ can't hit $\alpha$
but must hit $F$ under measure $\mathcal{P}^{h_{\alpha}}$.

In other words the  probability of the path conditioned on  the rare event of leaving $\alpha$ and going directly to $F$
is the same as that of a path for the $h_{\alpha}$-transformed chain. Moreover, suppose $H_{\ell}$ is a set of large deviation paths that leave $\alpha$ that hit $F$ before returning to $\alpha$
such that $P_{\alpha}[H_{\ell}]/P_{\alpha}[\tau_F<\tau_{\alpha}]\sim 1$ as $\ell\to\infty$. Typically $H_{\ell}$ is a tube of radius $\epsilon \ell$
around a fluid limit as in Theorem 6.15 in \cite{weiss}. Then
$$1\sim P_{\alpha}[H_{\ell}|\tau_F<\tau_{\alpha}]=\mathcal{P}^{h_{\alpha}}_{\alpha}[H_{\ell}];$$
i.e. the large deviation paths of $P_{\alpha}$ are almost surely the paths of $h_{\alpha}$-transformed chain.
We note in passing that $F$ could be any set far from $\alpha$ and all the above calculations hold.

We can use $\mathcal{K}_{\alpha}$ to simulate large deviations to $F$
and in particular estimate the hitting distribution on $F$.  $h_{\alpha}$ does not depend on the transition probabilities of kernel $K$ for transitions inside $F$; i.e.
we can change $K(f,\cdot)$ at any point $f\in F$ without changing $h_{\alpha}$ or $G_F$.
Consequently above shows that the large deviation path from $\alpha$ to $F$ does not depend on $K(f,\cdot)$ at any point $f\in F$.

Note that
$h_{\alpha}(x,y,s)\approx h(x,y,s)=\rho^{\ell-(x+y)}$.  In other words the exact rare event kernel $\mathcal{K}_{\alpha}$
is approximately the same as the $h$-transformed kernel. In fact all the calculations above work equally well with
$\mathcal{K}(p,q)=K(p,q)h(q)/h(p)$ (even at $p=\alpha$.
\begin{eqnarray}
\lefteqn{P_{\alpha}[M(n)=a_n, 1\leq n\leq \ell|\tau_F<\tau_{\alpha}]}\label{wantit}\\
&=&P_{\alpha}[M(n)=a_n, 1\leq n\leq \ell]/P_{\alpha}[\tau_F<\tau_{\alpha}]\nonumber\\
&=&\mathcal{P}^h_{\alpha}[\mathcal{M}(n)=a_n, 1\leq n\leq \ell]h^{-1}(a_{\ell})/\mathcal{E}^h[\chi\{\tau_F<\tau_{\alpha}\}h^{-1}(M(\tau_F))]\nonumber\\
&=&\mathcal{P}^h_{\alpha}[\mathcal{M}(n)=a_n, 1\leq n\leq \ell]/\mathcal{P}_{\alpha}^h[\tau_F<\tau_{\alpha}]\nonumber\\
&=&\mathcal{P}^h_{\alpha}[\mathcal{M}(n)=a_n, 1\leq n\leq \ell|\tau_F<\tau_{\alpha}]\label{gotit}.
\end{eqnarray}
For simulations of trajectories of the measure $\mathcal{P}^h_{\alpha}$  this just means we reject trajectories  which return to $0$ before hitting $F$
and these are just a fixed proportion since trajectories under $\mathcal{P}^h_{\alpha}$ drift toward $F$.

The drift of the twisted chain on sheet~1 is $m^1:=(\tilde{\lambda}_1-\tilde{\mu},\tilde{\lambda}_2)$ and
the drift of the twisted chain on sheet~2 is $m^2:=(\tilde{\lambda}_1,\tilde{\lambda}_2-\tilde{\mu},)$.
There are two cases: rays or spirals. There is a ray on sheet~1 if $\tilde{\lambda}_1-\tilde{\mu}>0$. There is a ray on sheet~2 if
$\tilde{\lambda}_2-\tilde{\mu}>0$. If both sheets are rays we call it the ray-ray case.  On a sheet with a ray the rare event occurs when the queue  just explodes due to fast arrivals and slow service.
In the spiral case $\tilde{\lambda}_1-\tilde{\mu}<0$ and $\tilde{\lambda}_2-\tilde{\mu}<0$ and in this case
we show that the server alternates between emptying the two queues but in a losing battle.
We explore the two cases in following sections.

The twisted chain provide a means of estimating $\pi$. Let $G_{\blacktriangle}((0,0,1);(x,y,s)$ denote the mean number of hits at $(x,y,s)$ where $x+y=\ell$
before returning to $\blacktriangle=\{(0,0,1)\}$.
We can use the representation
\begin{eqnarray}\label{firstrep}
 \pi(x,y,1)&=&\pi(0,0,1) G_{\blacktriangle}((0,0,1);(x,y,s))\nonumber\\
 &=&\pi(0,0,1) {\cal G}_{\blacktriangle}((0,0,1);(x,y,s))\frac{h(0,0)}{h(x,y)}\nonumber\\
 &=&(1-\rho)\rho^{\ell}{\cal G}_{\blacktriangle}((0,0,1);(x,y,s)).
 \end{eqnarray}
 Note that $\sum_{s\in \{1,2\}}\sum_{x+y=\ell}{\cal G}_{\blacktriangle}((0,0,1);(x,y,s))=1$ as above.

The above representation (\ref{firstrep}) shows we can obtain the distribution of $\pi$ on $\{(x,y,s)\}$ by simulating ${\cal G}_{\blacktriangle}((0,0,1);(x,y,1))$.
However, for the polling model, this  representation is not practical for giving the asymptotics of $\pi$ analytically because it involves paths crossing from one sheet to the other
where the transition kernel changes discontinuously. This also constitutes a complication for the large deviation approach.  In the next subsection
we show how to avoid this complication by restricting the representation to one sheet.

We summarize our definitions and assumptions:
\begin{hyp}{K}
	\textit{The transition kernel $K$ given in \eqref{K-defn}
		has
		\[
		\lambda_1 > 0, \quad \lambda_2 > 0, \quad
		\lambda \coloneqq \lambda_1 + \lambda_2 < \mu,
		\]
		and w.l.o.g.\ that
		\begin{equation}
		\lambda + \mu = 1.
		\end{equation}
		Under these conditions, $K$ is irreducible over the state space
		$S$ and has a stationary distribution denoted by
		$\pi$.
	}
\end{hyp}

\begin{hyp}{R1}
	\textit{The ray condition on sheet 1; i.e. $\tilde{\lambda}_1>\tilde{\mu}$ or $\rho^{-1}\lambda_1>\lambda$.}
\end{hyp}

\begin{hyp}{R2}
	\textit{The ray condition on sheet 2; i.e. $\tilde{\lambda}_2>\tilde{\mu}$ or $\rho^{-1}\lambda_2>\lambda$.}
\end{hyp}

\section{The ray case}\label{rayraylab}
In this section we assume there is a ray on  sheet~1; i.e. $\tilde{\lambda}_1-\tilde{\mu}>0$.
We wish to give the asymptotics of $\pi(x,y,1)$ where $x+y=\ell$ as $\ell\to\infty$.
There is a problem with periodicity since our chain has period $2$.
Since $\pi$ is the invariant probability of $K$ and of $(K+I)/2$
we could replace $K$ by $(K+I)/2$  in this section. Instead we will just ignore periodicity for now
and fix things up in Theorems \ref{one} and \ref{two}.

To evaluate $\pi(x,y,1)$ we will extend the transition kernel $K$ to the {\em free} kernel $\overline{K}$ in the interior of sheet~1 to the whole plane $\ZZ^2$.
This is in fact a random walk increments $J$ where
\begin{eqnarray*}
\overline{K}((0,0);(0,1))&=&P[J=(0,1)]=\lambda_2,\overline{K}((0,0);(1,0))= P[J=(1,0)]=\lambda_1 \\
\mbox{ and }& &\overline{K}((0,0);(-1,0))=P[J=(-1,0)]=\mu.
\end{eqnarray*}
We can drop the notation indicating sheet $1$ for the free process; i.e. $(x,y)$ replaces $(x,y,1)$.

We will apply the results in Ney and Spitzer  \cite{Ney-Spitzer}.
For $a,b\in \ZZ^2$
let $\overline{G}(a,b)=\sum_{n=0}^{\infty}\overline{K}^n(a,b)$ be the  potential associated with a random walk on $\ZZ^d$ with kernel $\overline{K}$ and let
 $$\mgf(u)=E[\exp(u\cdot J)]=\sum_v e^{v\cdot u}\overline{K}(0,v)$$
  be the associated transform.
Define $ D^{\mgf}=\{u|\mgf(u)\leq 1\}$ which is a convex set with surface
$\partial D^{\mgf}=\{u|\mgf(u)=1\}$.

We summarize our assumptions:
\begin{hyp}{(1.4) in \cite{Ney-Spitzer}:}
	\textit{each point in $\partial D^{\mgf}$ has a neighbourhood in which $\mgf$ is finite.}
\end{hyp}
\begin{hyp}{(1.3) in \cite{Ney-Spitzer}:}
	\textit{$m=\sum_v v\overline{K}(0,v)\neq 0$.}
\end{hyp}
\begin{hyp}{Nonsingular:}
	\textit{$\{v: \overline{K}(0,v)>0)\}$ does not lie in a subspace of $\ZZ^d$.}
\end{hyp}
All these assumptions clearly hold in our nearest neighbour example.
Let $|\cdot|$ be the $\ell^1$ norm; i.e. $|(x,y)|=|x|+|y|$ and let $||\cdot||$ be the $\ell^2$ norm.

By \cite{doob} and \cite{hennequin}, $h(z)=\exp(z\cdot u)$ is an extremal harmonic function in the Martin boundary for the Martin kernel $k(a;b)=G(a;b)/G(0;b)$
where $a,b,0\in \ZZ^d$.
We introduce the transition kernel
$\overline{K}^u(a,b)=\overline{K}(a,b)e^{u\cdot(b-a)}$  where  $ \mgf(u)=1$ and the associated random walk with increments $J^u$. The mean vector of
$\overline{K}^u$ is
$$m^u=\sum_{v\in \ZZ^d}v\overline{K}^u(0,v)=\nabla \mgf(u).$$
Let $d^u=m^u/|m^u|$.
Denote the potential
$\overline{G}^u(a,b)$. Note $\overline{G}^u(a,b)=\overline{G}(a,b)\exp(u\cdot (b-a))$.
If we pick $\overline{u}=(\ln(\rho^{-1}),\ln(\rho^{-1}))'$ then $\overline{u}\in \partial D^{\mgf}$
and $\overline{K}^{\overline{u}}=\overline{\mathcal{K}}$ and  $\overline{G}^{\overline{u}}=\overline{\mathcal{G}}$
where $\overline{\mathcal{K}}$ is the free twisted  kernel obtained by $h$-transformation using  the harmonic function $h(x,y)=\rho^{-(x+y)}$
and $\overline{\mathcal{G}}$ is the associated potential.

Theorem 1.2 and its Corollary 1.3 in \cite{Ney-Spitzer} shows the Martin compactification is equivalent to the the compactification $\overline{\ZZ}^d$ of $\ZZ^d$
with respect to the metric $\rho(a,b)=|\frac{a}{1+|a|}-\frac{b}{1+|b|}|$.
Corollary 1.3 shows any sequence $b_n$ such that $b_n/||b_n||\to d^{u}/||d^{u}||$; i.e. converging to $p$ in the $\rho$-compactification
also converges to a boundary point (we also call $p$) in the Martin boundary;
 i.e. $k(a;b_n)\to \exp(a\cdot u)$. Thus  Corollary 1.3 in \cite{Ney-Spitzer} gives an equivalence between the topology of the geometric boundary
and the topology of the Martin boundary.

Theorems 1.2 and 2.2 in Ney and Spitzer  \cite{Ney-Spitzer} require (1.1), (1.2), (1.3) and (1.4) given there.
(1.1) just requires that the kernel $\overline{K}$ be a random walk kernel which is our situation. (1.3) and (1.4) are [2] and [1] above.
However (1.2) requires irreducibility; i.e. for all $x$, $\overline{K}^n(0,x)>0$ for some $n$ and this not true in our example since there are no southern jumps on sheet~1.
Thus the Martin kernel $\overline{G}(a,b)/\overline{G}(0,b)$ is not defined for $b$  pointed in a southern direction. This is not a worry since
we are interested only in directions inside the cone $\mathcal{C}$ generated by the support of $\overline{K}(0,\cdot)$; i.e. the cone generated  $\{v:\overline{K}(0,v)>0\}$.
In the appendix we prove that we can choose $u$ where  $ \mgf(u)=1$ such that the mean direction $d^u$ lies anywhere in the interior of
$\mathcal{C}$. The appendix also shows how to point $d^u$ along the boundary of $\mathcal{C}$.

The asymptotics of ${\cal G}(a;z)$ are given by Theorem 2.2 in \cite{Ney-Spitzer}.
Following \cite{Ney-Spitzer} define the quadratic form
$$Q^{u}[\theta]=(\theta\cdot Q^u\theta)=\sum_{v\in \ZZ^d}|(v-m^u)\cdot \theta|^2\overline{K}^u(0,v)$$
which is the variance of $J^u\cdot \theta$ so $Q^u$ is the covariance matrix of $J^u$.
These quadratic forms are positive definite since if $Q^u(\theta)=0$ for some $\theta\neq 0$ then
$(v-m^u)\cdot \theta=0$ on the support of $\overline{K}^u(0,\cdot)$; i.e. on $v$ such that $\overline{K}(0,v)>0$.
This means $(v-m^u)\cdot \theta$ is a constant for $v$ such that $\overline{K}(0,v)>0$; i.e.
$\{v-m: \overline{K}(0,v)>0)\}$ lies in a subspace of $\ZZ^d$
which violates our Assumption [3]. Hence $Q^{u}[\theta]>0$ for all $\theta$; i.e. $Q^u$ is positive definite.
The inverse of $Q^u$ is denoted by $\Sigma^u$ and the corresponding determinants are denoted by $|Q^u|$ and $|\Sigma^u|=|Q^u|^{-1}$.

The proof of Corollary 1.3 in \cite{Ney-Spitzer} without assuming (1.2) follows  the proof given on page 121 just after the statement of Theorem 2.2 in \cite{Ney-Spitzer}.
Using the notation in \cite{Ney-Spitzer} the key point is that
$$\lim_{n\to\infty}\frac{G(x,x_n)}{G(0,x_n)}=\lim_{n\to\infty}\frac{G^{u_n}(x,<t_n\mu^{u_n}>)}{G^{u_n}(0,<t_n\mu^{u_n}>)}e^{u_n\cdot x}.$$
Using the uniformity in $u$ in Theorem 2.2 in \cite{Ney-Spitzer} the  fraction on the right hand side of the above expression tends to $1$ so
$G(x,x_n)/G(0,x_n)\to e^{u\cdot x}$ giving Corollary 1.3 in \cite{Ney-Spitzer}.
The proof of Theorem 2.2  requires Lemmas 2.4 through 2.11. These lemmas require (1.2) but only to show $Q^u$ is positive definite
but this follows from our Condition [3] so (1.2) is not needed (see the comment at the top of page 127 in  \cite{Ney-Spitzer}).
We conclude, under the Conditions [1], [2] and [3], that Theorems 1.2 and 2.2 in \cite{Ney-Spitzer} both hold provided only
that Theorem 2.2 be modified to a homeomorphism between $\partial D^{\mgf}$ and $\overline{\ZZ}^d\cap \mathcal{C}$.

We also define $\blacktriangle$ to be the complement of $\{(x,y,1):x\geq 1,y\geq 1\}$ in $\ZZ^2$. This defines $\Delta=S\cap\blacktriangle=\{(x,0,1):x\geq 0\}$;
i.e. the boundary points of $S$ in the plane. Let $\mathcal{K}_{\blacktriangle}$ be the taboo kernel $\mathcal{K}$ killed on $\blacktriangle$ and let
 $\mathcal{G}_{\blacktriangle}$ be the associated potential. These kernels are identical to the corresponding  free kernels killed on $\blacktriangle$ and we
 denote the free kernels as $\overline{\mathcal{K}}_{\blacktriangle}$, $\overline{\mathcal{K}}$ and $\overline{\mathcal{K}}_{\blacktriangle}$.
We recall the representation (\ref{firstrep})
\begin{eqnarray}\label{firstrep2}
 \pi(x,y,1)&=&h(x,y)^{-1}\sum_{(z_1,z_2,1)\in \Delta}h(z_1,z_2)\pi(z_1,z_2,1) {\cal G}_{\blacktriangle}((z_1,z_2);(x,y))\nonumber\\
 &=&\rho^{\ell}\sum_{z}\rho^{-z}\pi(z,0,1) \overline{{\cal G}}_{\blacktriangle}((z,0);(x,y))\mbox{ where }x+y=\ell.
 \end{eqnarray}

We can now give the asymptotics of (\ref{firstrep2}). First define
\begin{eqnarray}
 \Pi^{\overline{{\cal K}}}_{\blacktriangle}(a,b)&=&\mathcal{P}_a[\overline{\mathcal{J}}(\tau_{\blacktriangle})=b,\tau_{\blacktriangle}<\infty]\mbox{ for } a\in S,b\in\blacktriangle .\nonumber
 \end{eqnarray}
 where $\overline{\mathcal{J}}$ is the random walk with kernel $\overline{\mathcal{K}}$. Note that $\Pi^{\overline{{\cal K}}}_{\blacktriangle}(a,\cdot)$ has support on the $x$ and $y$ axes.
For $a\in \Delta$ and $z\in S\setminus \blacktriangle$, $\overline{{\cal G}}(a;z)$ can be decomposed as the expected number of visits to $z$ before hitting $\blacktriangle$;
i.e. $\overline{{\cal G}}_{\blacktriangle}(a,z)$ and the expected number of visits to $z$ by trajectories hitting $z$ after hitting blacktriangle $\blacktriangle$; i.e.
$\sum_{w\in\blacktriangle}\Pi^{\overline{{\cal K}}}_{\blacktriangle}(a,w))\overline{{\cal G}}(w;z)$.
Therefore
\begin{eqnarray}\label{Spitz}
\overline{{\cal G}}_{\blacktriangle}(a,z)&=&\overline{{\cal G}}(a;z)-\sum_{w\in\blacktriangle}\Pi^{{\overline{\cal K}}}_{\blacktriangle}(a,w))\overline{{\cal G}}(w;z)
\end{eqnarray}
for $a=(x,0)$ and $z\in \ZZ^2\setminus\blacktriangle$.

Let  $d^u(\ell)=(x,y)=<\ell\cdot d^{u}>$  where $<\cdot >$ denotes the nearest lattice point in $\ZZ^2$ on the line $x+y=\ell$.
Then $d^u(\ell)$ converges to a boundary point $p$ in the $\rho$-compactification $\overline{\ZZ}^d$
associated with the direction  $d^{u}$. Corollary 1.3 in \cite{Ney-Spitzer} establishes that $d^{u}$ converges to a point  in the Martin boundary;
 i.e. $k(a;d^{u}(\ell))\to \exp(a\cdot u)$.
Hence, for any $(x,y)$,
$$\frac{\overline{{\cal G}}((x,y);d^{\overline{u}}(\ell))}{\overline{{\cal G}}((0,0);d^{\overline{u}}(\ell))}=
\frac{\overline{G}^{\overline{u}}((x,y);d^{\overline{u}}(\ell))}{\overline{G}^{\overline{u}}((0,0);d^{\overline{u}}(\ell))}\to 1$$
as $\ell\to\infty$.
Note that the above result shows
$$\frac{\overline{{\cal G}}((0,0);d^{\overline{u}}(\ell)-(x,y))}{\overline{G}^{\overline{u}}((0,0);d^{\overline{u}}(\ell))}\to 1;$$
i.e. the exact rounding off by $<>$ doesn't change the result.

Consequently, from representation (\ref{firstrep2})
\begin{eqnarray}
\lefteqn{\pi(d^{\overline{u}}(\ell))=\rho^{\ell}\sum_{z\in \Delta}h(z)\pi(z)
 {\cal G}_{\blacktriangle}(z;d^{\overline{u}}(\ell))}\nonumber\\
 &=&\rho^{\ell}\sum_{z\in \Delta}h(z)\pi(z)
(\overline{{\cal G}}(z;w)-\sum_{w\in\blacktriangle}\Pi^{\overline{{\cal K}}}_{\blacktriangle}(z,w)\overline{{\cal G}}(w;z))\label{therep}\\
&=& \rho^{\ell}\overline{{\cal G}}((0,0);d^{\overline{u}}(\ell))\nonumber\\
& &\cdot
\left(\sum_{z\in \Delta}h(z)\pi(z)\frac{\overline{{\cal G}}(z;d^{\overline{u}}(\ell))}{\overline{{\cal G}}((0,0);d^{\overline{u}}(\ell)}-\sum_{z\in \Delta}h(z)\pi(z)\sum_{w\in\blacktriangle}\Pi^{\overline{{\cal K}}}_{\blacktriangle}(z;w)
\frac{\overline{{\cal G}}(w;d^{\overline{u}}(\ell))}{\overline{{\cal G}}((0,0);d^{\overline{u}}(\ell))}\right).\nonumber\\
\end{eqnarray}

To go further we need  $\sum_{z\in \Delta}h(z)\pi(z)<\infty$ but this is shown in Appendix \ref{N-T} when there is a ray on sheet~1. Then
using Proposition \ref{uniformbound} below we can apply dominated convergence to get
\begin{eqnarray*}
\pi(d^{\overline{u}}(\ell))&\sim&\rho^{\ell}\overline{{\cal G}}((0,0);d^{\overline{u}}(\ell))\left(\sum_{z\in \Delta}h(z)\pi(z)
-\sum_{z\in \Delta}h(z)\pi(z)\sum_{w\in\blacktriangle}\Pi^{\overline{{\cal K}}}_{\blacktriangle}(z,;w))\right)\\
&=&\rho^{\ell}\overline{{\cal G}}((0,0);d^{\overline{u}}(\ell))\sum_{z\in \Delta}h(z)\pi(z)\mathcal{P}^h_z[\mathcal{M}\mbox{ never hits }\blacktriangle]
\end{eqnarray*}

We remark that is tempting to consider the Martin kernel
$$k_{\gamma}(a;d^{\overline{u}}(\ell))=\frac{\overline{{\cal G}}(a;d^{\overline{u}}(\ell)}
{\sum_{z\in \Delta}h(z)\pi(z)\overline{{\cal G}}(z;d^{\overline{u}}(\ell))}.$$
Then $\liminf_{\ell\to\infty}k_{\gamma}(a;d^{\overline{u}}(\ell))\leq \frac{1}{\sum_{z\in \Delta}h(z)\pi(z)}$. Hence any  subsequential limit is bounded
and since $1$ is extremal for ${\cal K}$ we conclude these harmonic limits are constant functions. We could therefore avoid Proposition \ref{uniformbound}. The problem is the limit might be zero even if
$\sum_{z\in \Delta}h(z)\pi(z)<\infty$ so a result like Proposition \ref{uniformbound} seems to be necessary.

Let $b^u(\ell)=<\ell\cdot m^u>$ where $<>$ means the nearest lattice point. Note that
$$d^u(\ell)=<\ell\cdot \frac{m^u}{|m^u|}>=<\frac{\ell}{|m^u|}\cdot m^u>=b^u(\frac{\ell}{|m^u|}).$$
 Theorem 2.2 in \cite{Ney-Spitzer} gives
$$\overline{{\cal G}}((0,0,1);b^{\overline{u}}(\ell))\sim (2\pi \ell)^{-(d-1)/2}[|Q^{\overline{u}}|(m^{\overline{u}}\cdot \Sigma^{\overline{u}} m^{\overline{u}})]^{-1/2}$$
where the dimension $d$ is $2$ and
where $\Sigma^{\overline{u}}$ is the inverse of $Q^{\overline{u}}$, the covariance matrix of the twisted random walk increments; i.e.
$Q^{\overline{u}}= E[(\mathcal{J}-m^{\overline{u}})\cdot (\mathcal{J}-m^{\overline{u}} )')]$ and $|Q^{\overline{u}}|$ is the determinant of $Q^{\overline{u}}$.
Consequently,
$$\overline{{\cal G}}((0,0,1);d^{\overline{u}}(\ell))\sim \frac{1}{\sqrt{2\pi \ell/|m^{\overline{u}}|}}[|Q^{\overline{u}}|(m^{\overline{u}}\cdot \Sigma^{\overline{u}} m^{\overline{u}})]^{-1/2}$$

We conclude
\begin{theorem}\label{one}
If there is a ray on sheet~1; i.e. $\tilde{\lambda}_1>\tilde{\mu}$ then
\begin{eqnarray*}
\pi(d^{\overline{u}}(\ell))&\sim&B^{\overline{u}}\rho^{\ell}\frac{1}{\sqrt{2\pi \ell/|m^{\overline{u}}|}}
\end{eqnarray*}
where $$B^{\overline{u}}=([|Q^{\overline{u}}|(m^{\overline{u}}\cdot \Sigma^{\overline{u}} m^{\overline{u}})]^{-1/2})(\sum_{z\in \Delta}h(z)\pi(z)P_z[\mathcal{M}\mbox{ never hits }\blacktriangle]).$$
\end{theorem}

\begin{proof}
The only issue is periodicity. If we had replaced $K$ by $(K+I)/2$ then that issue disappears so the above theorem holds but
$Q^{\overline{u}}$, $\Sigma^{\overline{u}}$ and $m^{\overline{u}}$ above are  the covariance, the inverse covariance and the mean of $(\overline{K}^u+I)/2$
respectively.
Let $Q_K^{\overline{u}}$, $\Sigma_K^{\overline{u}}$ and $m_K^{\overline{u}}$ above are  the covariance, the inverse covariance and the mean of $\overline{K}^u$.
The value of $\pi(d^{\overline{u}}(\ell))$ is the same for both kernels. The value
$(\sum_{z\in \Delta}h(z)\pi(z)P_z[\mathcal{M}\mbox{ never hits }\blacktriangle])$ for the kernel $(\overline{K}^u+I)/2$ is half that for the kernel  $\overline{K}^u$
because staying at $z$ means hitting $\blacktriangle$.

Let $L$ be an independent random variable such that $P(L=1)=1/2$. Then the increment of the random walk with kernel $(\overline{K}^u+I)/2$
can be written $L\cdot J$ where as above $J$ is the increment of the walk with kernel $\overline{K}^u$.
Consequently $m^{\overline{u}}=E[L\cdot J]=E[J]/2$. Next,
\begin{eqnarray*}
Q^{\overline{u}}(\theta)&=&E[(LJ-E[LJ])\cdot \theta]^2\\
&=&E[(LJ\cdot\theta)^2]-\frac{(E[J]\cdot\theta)^2}{4}=\frac{E[(J\cdot\theta)^2]}{4}-\frac{(E[J]\cdot\theta)^2}{4}\\
&=&\frac{1}{4}E[(J-E[J])\cdot \theta]^2=\frac{1}{4}Q_K^{\overline{u}}(\theta);
\end{eqnarray*}
i.e. $Q^{\overline{u}}=Q_K^{\overline{u}}/2$ and $\Sigma^{\overline{u}}=2\Sigma_K^{\overline{u}}$.
Hence
$(m^{\overline{u}}\cdot \Sigma^{\overline{u}} m^{\overline{u}})=(m_K^{\overline{u}}\cdot \Sigma_K^{\overline{u}} m_K^{\overline{u}})/2$.

The value of the determinant $|Q^{\overline{u}}|=(\frac{1}{2})^d|Q_K^{\overline{u}}|/4$  so the value of $B^{\overline{u}}=((\frac{1}{2})^d\cdot\frac{1}{2})^{-1/2}\cdot \frac{1}{2})B_K^{\overline{u}}$
where $B_K^{\overline{u}}$ is the corresponding value for the kernel $K$.
Finally the value of
$$(2\pi \frac{\ell}{|m^{\overline{u}}|})^{-(d-1)/2}=(\frac{1}{2})^{-(d-1)/2}(2\pi \frac{\ell}{|m_K^{\overline{u}}|})^{-(d-1)/2}.$$
The product of all these extra factors is
$((\frac{1}{2})^d\cdot\frac{1}{2})^{-1/2}\cdot (\frac{1}{2})^{-(d-1)/2}=1$
 so we see we were justified in ignoring periodicity (even in dimension $d$).
\end{proof}

If we consider another direction $d^u$
then again
$$\overline{G}^u((0,0);b^u(\ell))\sim \frac{1}{\sqrt{2\pi \ell}}[|Q^u|(m^u\cdot \Sigma^u m^u)]^{-1/2}$$
where $\Sigma^u$ is the inverse of $Q^u$, the covariance matrix of the  random walk with kernel $\overline{K}^u$; i.e.
$Q^u= E[(J^u-m^u)\cdot (J^u-m^u )')]$ and $|Q^u|$ is the determinant of $Q^u$.
However
\begin{eqnarray*}
\overline{{\cal G}}((0,0);b^u(\ell))&=&\overline{G}((0,0);b^u(\ell))\exp(\overline{u}\cdot b^u(\ell))\\
&=&\exp((\overline{u}-u)\cdot b^u(\ell))\overline{G}^u((0,0);d(\ell))\\
&\sim&\exp(\ell (\overline{u}-u)\cdot d^u)\frac{1}{\sqrt{2\pi \ell}}[|Q^u|(m^u\cdot \Sigma^u m^u)]^{-1/2}.
\end{eqnarray*}
However $d^u=\nabla[\mgf(u)]/|\nabla[\mgf(u)]|$ and $\nabla[\mgf(u)]$ is orthogonal to the supporting hyperplane (in fact a line)
at the extremal point $u$ of the convex set $\partial D^{\mgf}$. Consequently for $\overline{u}$ or any other point in  $\partial D^{\mgf}$,
$(\overline{u}-u)\cdot \nabla[\mgf(u)]<0$.
Consequently
\begin{theorem}\label{two}
Along any northern direction $d^u$
\begin{eqnarray*}
\pi(d^u(\ell))&=&O(\frac{\rho^{\ell}}{\sqrt{\ell}})\exp(-\alpha_{d^u}\ell )
\end{eqnarray*}
where $\alpha_{d^u}=(\overline{u}-u)\cdot d^u>0$.
\end{theorem}

\begin{proposition}\label{uniformbound}
$\frac{\overline{{\cal G}}(z;b^{\overline{u}}(\ell))}{\overline{{\cal G}}((0,0);b^{\overline{u}}(\ell))}$ is uniformly bounded in $\ell$ and $z$ on the $x$ or $y$ axes of sheet~1.
\end{proposition}
\begin{proof}
From Theorem 2.2 in \cite{Ney-Spitzer},
$$\overline{G}^u((0,0);b^u(\ell))\sim \frac{1}{\sqrt{2\pi \ell}}[|Q^u|(m^u\cdot \Sigma^u m^u)]^{-1/2}$$ uniformly in $u$; i.e.
we can take $L$ sufficiently large so that
$$(1-\epsilon)<\overline{{\cal G}}((0,0);b^u(\ell))\frac{\sqrt{2\pi \ell}}{ [|Q^u|(m^u\cdot \Sigma^u m^u)]^{-1/2}}<(1+\epsilon)$$
for $\ell>L$ uniformly in $u$.
Now consider  $z\in \blacktriangle$ of the form $z=<a\ell>$ where $a=(a_1,a_2)$. Since $\overline{K}$ and  $\overline{\mathcal{K}}$ have no southern jumps,
$\overline{{\cal G}}(z;b^{\overline{u}}(\ell))=0$ if $a_1>1$.  If $a_1\leq 1$ then $p(\ell)=b^{\overline{u}}(\ell)-<a\ell>=<c\ell>$ is in the cone formed from the support of $\overline{K}$ or $\mathcal{K}$.

Pick $u^*$ such that $\mgf(u^*)=1$  and
 $d^{u^*}=\nabla \mgf(u^*)/\nabla \mgf(u^*)|$ is in direction $p(\ell)$ so $p(\ell)=<\ell |c| d^{u^*}>$. Moreover
$$\overline{\mathcal{G}}^{u^*}((0,0);p(\ell))\sim \frac{1}{\sqrt{2\pi |c|\ell}}[|\mathcal{C}^{u^*}|(m^{u^*}\cdot \Sigma^{u^*} m^{u^*})]^{-1/2}$$
where $\Sigma^{u^*}$ is the inverse of $\mathcal{C}^{u^*}$, the covariance matrix of the  random walk with kernel $\mathcal{K}^{u^*}$ and $|\mathcal{C}^{u^*}|$ is the determinant of $\mathcal{C}^{u^*}$.

Hence,
\begin{eqnarray*}
\overline{\mathcal{G}}((0,0);p(\ell))&=&\overline{G}^{\overline{u}}((0,0);p(\ell))=\overline{G}((0,0);p(\ell))e^{\overline{u}\cdot p(\ell)}\\
&=&e^{(\overline{u}-u^*)\cdot p(\ell)}\overline{G}^{u^*}((0,0);p(\ell))\\
&\leq& e^{(\overline{u}-u^*)\cdot p(\ell)}(1+\epsilon)\frac{1}{\sqrt{2\pi |c|\ell}}[|Q^{u^*}|(m^{u^*}\cdot \Sigma^{u^*} m^{u^*})]^{-1/2}
\end{eqnarray*}
as long as $|c|\ell>L$.
Again by the convexity of $D^{\mgf}$, $(\overline{u}-u^*)\cdot p(\ell)\leq 0$.
Moreover, since the support of $\overline{K}$ is finite we use Theorem \ref{ontheedge} to see probabilities $\overline{K}(0,v)e^{u\cdot v}$ where $u\in \partial D^{\mgf}$ form
a compact set. Hence $|Q^u|$ and $(m^u\cdot \Sigma^u m^u)$ are bounded away from zero and infinity (see Lemma 2.4 in \cite{Ney-Spitzer}).

Hence, for $z=<a\ell>\in \blacktriangle$,
\begin{eqnarray*}
\lefteqn{\frac{\overline{{\cal G}}(z;b^{\overline{u}}(\ell))}{\overline{{\cal G}}((0,0);b^{\overline{u}}(\ell))}=
\frac{\overline{\mathcal{G}}((0,0);p(\ell))}{\overline{{\cal G}}((0,0);b^{\overline{u}}(\ell))}}\\
&\leq&(1+\epsilon)\frac{1}{\sqrt{2\pi |c|\ell}}[|Q^{u^*}|(m^{u^*}\cdot \Sigma^{u^*} m^{u^*})]^{-1/2}(\frac{1}{\sqrt{2\pi \ell}}[|Q^{\overline{u}}(m^{\overline{u}}\cdot \Sigma^{\overline{u}} m^{\overline{u}})]^{-1/2})^{-1}\\
&\leq&(1+\epsilon)\frac{1}{\sqrt{|c|}}C
\end{eqnarray*}
where $C$ is a constant independent of $u^*$ or $\overline{u}$. Moreover,  for $z=<a\ell>\in \blacktriangle$ there is a minimum value of $|c|$
taken on the $x$ or $y$ axes. Hence the above bound is uniform in $z=<a\ell>\in \blacktriangle$.
\end{proof}

This gives a precise picture of a run-away queue on sheet~$1$. The paths to the boundary $x+y=\ell$ drift in direction $d^{\overline{u}}$.
Moreover $\pi(d^{\overline{u}}(\ell))$ is of order $\rho^{\ell}/\sqrt{\ell}$ while points on the boundary a distance $\epsilon \ell$
from $b^{\ell}$ have probability $\pi$ exponentially smaller in $\ell$. By symmetry there are analogues to Theorems \ref{one} and \ref{two} on sheet~2.
This partially explains the histogram \ref{histhits} of hits on level $\ell=1000$ of the ray-spiral case $\mu=.65$, $\lambda_1=.3$ and $\lambda_2=.05$ based on simulating $10^6$ busy
periods of the twisted chain.  There were no hits on sheet~2 at level $1000$. The mode is $x=(691,309)$ while the theory predicts $(690.48,309.52)$.
The histogram is not normal because all the mass is within one standard deviation. We have no theoretical prediction for the shape of the histogram.

\begin{figure}
\begin{center}
\includegraphics[width=0.9\textwidth]{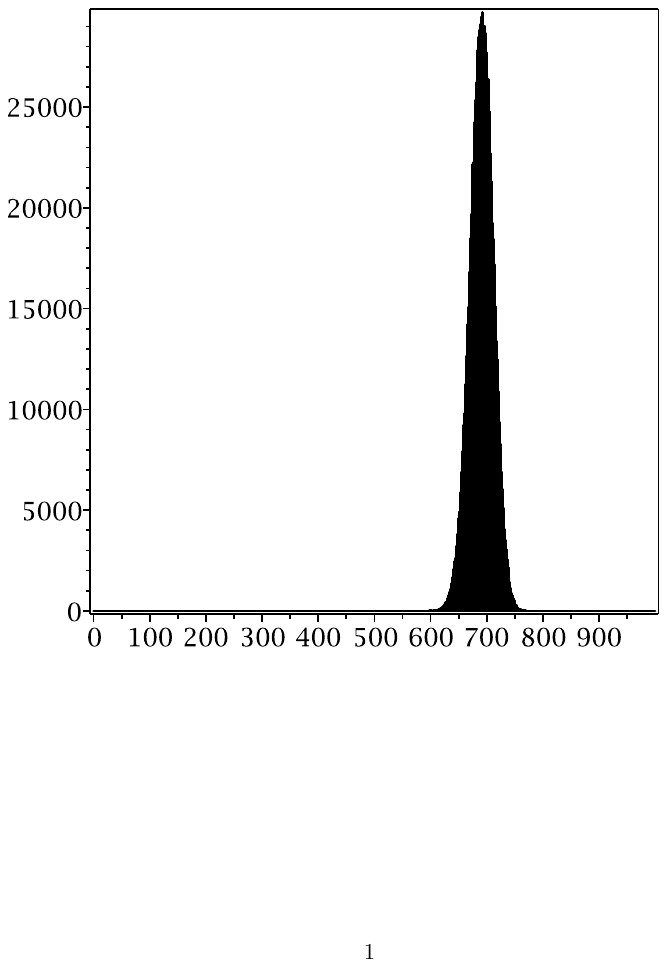}
\caption{Ray-Spiral case: $\lambda_1=0.3$, $\lambda_2=0.05$, $\mu=0.65$\label{histhits}}
\end{center}
\end{figure}

\section{The ray case: all directions}\label{ray-spiral}

We note that the methodology used in the ray case above can be generalized to any direction.
To find  $\pi(b^{u}(\ell))$ we must define $\blacktriangle$ to cut out the boundary and then check that
$\sum_{z\in\Delta}e^{u\cdot z}\pi(z)<\infty$ where $\Delta=S\cap \blacktriangle$. Also take $\alpha=\vec{0}=(0,0,\cdots ,0)$ and note that the function $h_{\alpha}(z)=(e^{u\cdot z}-1)/(e^{u\cdot b^{u}(\ell)}-1)$ is harmonic and
satisfies $h_{\alpha}(\vec{0})=0$ and $h_{\alpha}(b^{u}(\ell))=1$. Take $F=\{b^{u}(\ell)\}$  and apply the reasoning in Subsection \ref{pathrare}.
We get that for any path $(a_0,a_1,\ldots a_{T-1},a_{T})$ starting at $a_0=\alpha=(0,0,1)$ and entering $F$ for the first
time at $a_{T}$:
\begin{eqnarray*}
P_{\alpha}[M(n)=a_n, 1\leq n\leq T|\tau_F<\tau_{\alpha}]
=\mathcal{P}^{h_{\alpha}}_{\alpha}[\mathcal{M}(n)=a_n, 1\leq n\leq T];
\end{eqnarray*}
i.e. if the large deviation paths of the chain twisted by $h_{\alpha}$ lie in some set $H_{\ell}$ then
the same is true of the original chain conditioned that a large deviation occurs.

Consider the free kernel of our polling model on sheet~1 $\overline{K}$ on $\ZZ^2$. Since the support of $\overline{K}$ is $S_{\rho}=\{(1,0),(-1,0),(0,1)\}$ it follows
from Theorem \ref{goanyway} that for any direction $\beta=(\beta_1,\beta_2)$ such that $\beta_2\geq 0$ there exists a twist $u$ such that
$\overline{K}^u$ has mean $\mu^u$ in direction $\beta$.
In this section we evaluate the decay of $\pi(b^u(\ell))$ in all north-easterly directions.
We just assume there is a ray one sheet~1; i.e.  $\tilde{\lambda}_1-\tilde{\mu}>0$; i.e. $\sqrt{\mu\lambda_1}>\lambda$.
By twisting the free kernel $\overline{K}$ into the direction $\beta=(0,1)$ we will obtain what we call a bridge path.
Let $\exp(u_1x+u_2y)=\alpha_T^x\beta_T^y$. Find $\alpha_T$ and $\beta_T$ such that
$$\alpha_T\lambda_1+\beta_T\lambda_2+\alpha_T^{-1}\mu=1\mbox{ and }\alpha_T^{-1}\mu=\lambda_1\alpha_T.$$
Clearly $\alpha_T=\sqrt{\mu/\lambda_1}$ and
$\beta_T=(1-2\sqrt{\mu\lambda_1})/\lambda_2$. Note that $\beta_T$ is positive since $\lambda\mu\leq 1/4$ so $\mu\lambda_1<1/4$ and so $2\sqrt{\mu\lambda_1}<1$.

\begin{proposition}\label{determine}
$\alpha_T<\rho^{-1}$ if and only if sheet~1 is a ray.
\end{proposition}
\begin{proof}
$\alpha_T=\sqrt{\mu/\lambda_1}<\rho^{-1}$ if and only if $\lambda<\sqrt{\mu\lambda_1}$; i.e. if and only there is a ray on sheet~1.
\end{proof}

\begin{figure}[t!]
\begin{center}
\includegraphics{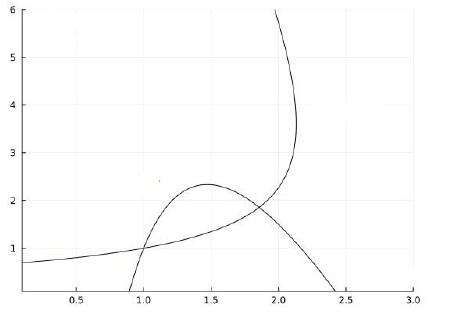}
\caption{Ray-Spiral case: $\lambda_1=6/20$, $\lambda_2=1/20$, $\mu=13/20$}
\end{center}
\begin{picture}(100,100)(-122,-6)
\put(80,215){\circle*{2}}
\put(85,215){$(\alpha_E,\beta_E)$}
\put(61,180){\circle*{2}}
\put(70,177){$(\rho^{-1},\rho^{-1})$}
\put(37,191){\circle*{2}}
\put(32,195){$(\alpha_T,\beta_T)$}
\put(72,191){\circle*{2}}
\put(78,189){$(\gamma_T,\beta_T)$}
\put(90,157){Curve~1}
\put(85,245){Curve~2}
\end{picture}
\end{figure}
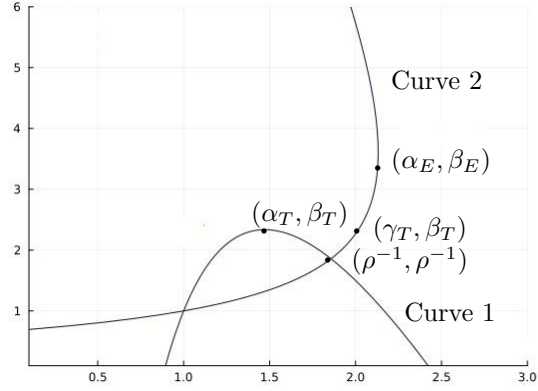

Curve~1: $\lambda_1\alpha+\lambda_2\beta+\mu\alpha^{-1}=1$ is associated with sheet~1 and
Curve~2: $\lambda_1\alpha+\lambda_2\beta+\mu\beta^{-1}=1$ is associated with sheet~2. The points $(1,1)$ and  $(\rho^{-1},\rho^{-1})$ are the only intersection points on both curves.
 $(\alpha_T,\beta_T)$  is the top of the curve associated with sheet~1
 so $\beta_T>\rho^{-1}$.
 Moreover by Proposition \ref{determine},  $\alpha_T<\rho^{-1}$.

 Denote the right most point on Curve~2 by $(\alpha_E,\beta_E)$
 where
 \begin{eqnarray}\label{rightmost}
 \alpha_{E}=(1-2\sqrt{\mu\lambda_2})/(2\lambda_1)\mbox{ and }\beta_E=\sqrt{\mu/\lambda_2}
 \end{eqnarray}
  by the above argument. Since $(\rho^{-1},\rho^{-1})$ lies on Curve~2 it follows that
 $\alpha_E>\rho^{-1}>\alpha_T$. Since there are no other intersection points it follows that there is a point $(\gamma_T,\beta_T)$
 on Curve~2 where $\gamma_T>\rho^{-1}$. We use these points later.

\subsection{Computing $\pi(1,y,1)$ and $\pi(0,y,2)$}
In order to determine if a large deviation is a ray on the first sheet we need precise estimates of $\pi(x,0,1)$. We proceed in a roundabout manner.
 \subsubsection{A bridge to $(1,y,1)$}\label{upbridge}
We will now obtain the asymptotics of $\pi(1,y,1)$ as $y$ gets large. Intuitively the large deviation path to the point $(1,y,1)$ is a bridge which
skims above the $y$-axis. We take $\Delta=\{(x,0,1):x\geq 1\}$
and $\blacktriangle= \{(x,y):x\geq 1,y\geq 1\}^c$ inside $\ZZ^2$. We consider the free kernel on sheet~1 $\overline{K}$ to have a killing $K((1,y,1),(0,y,2)$ at points $(1,y)$
where the chain leaves sheet~1.
To twist the free kernel $\overline{K}$ to make a bridge path consider a {\em rough} $h_T$-transform where $h_T(x,y)=\alpha_T^x\beta_T^y$.

Let $\mathcal{K}_T$ denote the $h_T$-transformed kernel.  The Markovian part (the $x$-component) of $\mathcal{K}_T$ is $s=\hat{\mathcal{K}}_T(x,x)=\lambda_2\beta_T $ for $x\geq 1$ and
$$p=\hat{\mathcal{K}}_T(x,x+1)=\hat{\mathcal{K}}_T(x,x-1)=\sqrt{\mu\lambda_1}\mbox{ for } x\geq 2. $$
Hence $\hat{\mathcal{K}}_T(x,x-1)=p$ for $x\geq 2$. At $x=1$
$\hat{\mathcal{K}}_T(1,2)=\sqrt{\mu\lambda_1}$ but $\hat{\mathcal{K}}_T(1,0)=0$ since this is a transition off sheet~1.
Hence $\hat{\mathcal{K}}_T$ has a killing $\kappa=\sqrt{\mu\lambda_1}$ at points $(1,y)$.
The mean drift in the $y$-direction at $(x,y)$ for $x\geq 2$ is
$d_{+}=\lambda_2\beta_T$.

Note that $\hat{h}_0(x)=x$ is harmonic for $\hat{\mathcal{K}}_T$ and the associated $\hat{h}_0$-transform gives a Markovian kernel
satisfying the conditions of Theorem 1 in \cite{kesten}. Consequently $\hat{h}_0$ is unique harmonic function for $\hat{\mathcal{K}}_T$ and consequently $h(x,y)=\hat{h}_0(x)h_T(x,y)=x\alpha_T^x\beta_T^y$
is harmonic for $\overline{K}$ killed on the $y$-axis.

 Again the $h$-transform of $\overline{K}$ is denoted by $\overline{\mathcal{K}}$ which is associated with the twisted chain $\mathcal{M}$ and the associated potential
is $\overline{\mathcal{G}}$.
$\overline{\mathcal{K}}_{\blacktriangle}$ is the kernel  killed on hitting $\blacktriangle$.
Also note that $\sum_{x}x\alpha_T^{x}\pi(x,0,1)<\infty$ since $\alpha_T<\rho^{-1}$ and $\pi(x,0,1)\leq (1-\rho)\rho^x$.
Again recall the representation (\ref{firstrep}),
\begin{eqnarray}\label{firstrep3}
 \pi(1,y,1)&=&h(1,y)^{-1}\sum_{x\geq 0}h(x,0)\pi(x,0,1) {\cal G}_{\blacktriangle}((x,0);(1,y))\nonumber\\
 &=&\alpha_T\beta_T^{-y}\sum_{x}x\alpha_T^{x}\pi(x,0,1) \overline{{\cal G}}_{\blacktriangle}((x,0);(1,y)).
 \end{eqnarray}

Let $\gamma(x,0)=h_0(x)\alpha_T^{x}\pi(x,0,1)\chi_{\Delta}((x,0))$ and $\gamma_{\blacktriangle}(w)=\gamma \Pi^{\overline{{\cal K}}}_{\blacktriangle}(w)$
where $w\in\blacktriangle$ and $\Pi^{\overline{{\cal K}}}_{\blacktriangle}((x,0);w))$ is the probability of hitting $w$ if  the twisted chain started at $(x,0)$ enters $\blacktriangle$.
As in (\ref{therep}),
\begin{eqnarray*}
& &\gamma\overline{{\cal G}}_{\blacktriangle}(1,y)\\&=&\sum_x\gamma(x,0)\overline{{\cal G}}((x,0);(1,y))-\sum_x\gamma(x,0)\sum_{w\in\blacktriangle}\Pi^{\overline{{\cal K}}}_{\blacktriangle}((x,0);(w,0)))\overline{{\cal G}}((w,0);(1,y))\\
&=&\gamma  \overline{{\cal G}}(1,y)-  \gamma_{\blacktriangle}\overline{{\cal G}}(1,y) .
\end{eqnarray*}

We will  check below that, $y\to\infty$,
$\frac{\overline{{\cal G}}((x,0);(1,y))}{\overline{\mathcal{G}}((1,0);(1,y))}$ is uniformly bounded in $x$ and tends to $1$ as $y\to \infty$. If this is true then
$$\frac{\gamma  \overline{{\cal G}}(1,y)}{\overline{\mathcal{G}}((1,0);(1,y)}= \sum_{x}x\alpha_T^{x}\pi(x,0,1)\frac{\overline{\mathcal{G}}((x,0);(1,y))}{\overline{\mathcal{G}}((1,0);(1,y))}
\to \sum_{x}x\alpha_T^{x}\pi(x,0,1)$$ and
\begin{eqnarray*}
& &\frac{\gamma_{\blacktriangle} \overline{{\cal G}}(1,y) }{\overline{\mathcal{G}}((1,0);(1,y))}= \sum_{x}x\alpha_T^{x}\pi(x,0,1)\sum_{w\in\blacktriangle}\Pi^{\overline{{\cal K}}}_{\blacktriangle}((x,0);w))\frac{\overline{\mathcal{G}}(w;(1,y))}{\overline{\mathcal{G}}((1,0);(1,y))}\\
&\to& \sum_{x}x\alpha_T^{x}\pi(x,0,1)\mathcal{P}_{(x,0)}[\mathcal{M}\mbox{ hits }\blacktriangle].
\end{eqnarray*}
In which case,
\begin{proposition}\label{three}
If there is a ray on sheet~1 then
$$\pi(1,y,1)\sim\alpha_T^{-1}
\beta_T^{-y}\overline{\mathcal{G}}((1,0);(1,y)\sum_{x\in \Delta}x\alpha_T^{x}\pi(x,0,1) \mathcal{P}_{(x,0)}[\mathcal{M}\mbox{ never hits }\blacktriangle].$$
\end{proposition}
We have therefore stripped out the exponential decay $\beta_T^{-y}$ out of $\pi(1,y,1)$ leaving the polynomial decay $\overline{\mathcal{G}}((1,0);(1,y)$.

Next recall $$\overline{\mathcal{G}}((1,0);(1,y))=\mathcal{G}_T((1,0);(1,y)\frac{\hat{h}_0(1)}{\hat{h}_0(1)}=\mathcal{G}_T((1,0);(1,y).$$
$\mathcal{G}_T((1,0);(1,y)$ was analyzed analytically in \cite{FM:BridgesExact}. To see that requires relabeling or shifting the first quadrant so $(1,0)\to (0,0)$ and $((1,y)\to (0,y)$.
Denote the relabelled or shifted kernel by $\mathcal{K}_s$ and the associated potential by $\mathcal{G}_s$.
Hence $\mathcal{K}_s((x,y);((u,v))=\mathcal{K}_T((x+1,y);((u+1,v))$ and
$\mathcal{G}_T((1,0);(1,y)=\mathcal{G}_s((0,0);(0,y)$.
In these new coordinates the probability of killing $\kappa$ at any point $(0,y)$ is $\mathcal{K}_T((1,y);(0,y)=\mu \alpha_T^{-1}$
and $p_0:=\mathcal{K}_s((0,y);(1,y))=p$.
By Proposition 1 in \cite{FM:BridgesExact},
${\cal G}_s((0,0);(y,0)) \sim C_+ \,  y^{-3/2}$
where
$C_+ = \frac{p}{\kappa^2}\sqrt{\frac{d_+}{2 \pi (1 - s)}}$
\begin{theorem}\label{threeplus}
If $\mu\cdot(\lambda_1/\lambda)-\lambda>0$ then
$$\pi(1,y,1)\sim\alpha_T\beta_T^{-y}C_+   y^{-3/2}\sum_{x\in \Delta}x\alpha_T^{x}\pi(x,0,1) \mathcal{P}_{(x,0)}[\mathcal{M}\mbox{ never hits }\blacktriangle].$$
\end{theorem}

\begin{proof}
We still need to check that, $y\to\infty$, for $x\in \Delta$,
$\frac{\overline{{\cal G}}((x,0);(1,y))}{\overline{\mathcal{G}}((1,0);(1,y)}$ is uniformly bounded in $x$ and tends to $1$ as $y\to \infty$.
First remark that $\overline{\mathcal{K}}$ has period $2$. We could proceed by taking $x$ odd and $y$ even but
it is easier to recall that we could just as well have redefined  $\overline{\mathcal{K}}$ as $\frac{1}{2}\overline{\mathcal{K}}+\frac{1}{2}I$.
This eliminates the periodicity problem so $\overline{\mathcal{K}}$ satisfies  Kesten's uniform aperiodicity property (1.5) in \cite{kesten}. Also that $1$ is the unique harmonic function for $\hat{\mathcal{K}}$
so Condition A5.5 holds. We may now apply Proposition 4 in \cite{FM:BridgesExact} to conclude
$\frac{\overline{{\cal G}}((x,0);(1,y))}{\overline{\mathcal{G}}((1,0);(1,y)}\to 1$ as $y\to \infty$.

To show $\frac{\overline{{\cal G}}((x,0);(1,y))}{\overline{\mathcal{G}}((1,0);(1,y)}$ is bounded uniformly bounded $x$ we use the argument in Section 7.3 in \cite{FM:BridgesExact}.
For paths $\omega'$ of $\mathcal{M}$ starting from
$(1,0)$ define $N_{(1,y)}(\omega')$ to be the number of visits
by $\mathcal{M}$ to $(1,y)$ following the trajectory $\omega'$.
 Similarly
define for paths $\omega$ of $\mathcal{M}$ starting from $(x,0)$
define $N_{(1,y)}(\omega)$ to be the number of hits at
$(\ell,0)$.

Now consider the product space of all paths $\omega$ of $\mathcal{M}$
which start from $(x,0)$ times paths $\omega'$ which start from
$(1,0)$. On this product space we can define a coupled path
starting from $(1,0)$ which follows $\omega'$ until $\omega'$ hits
the path $\omega$ and then follows $\omega$. Given a path $\omega$
which hits $(1,y)$, we note that all paths $\omega'$ must
hit the path $\omega$ because the path $\omega'$ is trapped
between the $y$-axis and the path $\omega$ (recall $\mathcal{K}((1,t),(0,t))=0$ for all $t$).
Define $N_{(1,y)}(\omega',\omega)$ to be the number of hits at
$(1,y)$ by the coupled path. Then, for any $\omega'$ and any
 $\omega$,
 $N_{(1,y)}(\omega)\leq N_{(1,y)}(\omega',\omega)$.
But $E[N_{(1,y)}(\omega',\omega)]=E[N_{(1,y)}(\omega')]$. Hence
$E[N_{(1,y)}(\omega)] \leq E[N_{(1,y)}(\omega')]$; i.e.
$$\overline{{\cal G}}((x,0);(1,y))\leq \overline{{\cal G}}((1,0);(1,y)).$$

\end{proof}

In the above we have followed the argument in \cite{FM:BridgesExact}.
 Instead we could have used the much more general argument in \cite{Irina}
to check that, as $y\to\infty$,
$\frac{\overline{{\cal G}}((x,0);(1,y))}{\overline{\mathcal{G}}((1,0);(1,y))}$  tends to $1$ as $y\to \infty$.
In \cite{Irina} the Martin boundary of is obtained for a killed random walk on a half-space. This is exactly what we need in our special case since falling off sheet~1 is equivalent to killing.
Moreover in our special nearest-neighbour case, the function $h_{a,+}(z)=h_{a,+}(x,y)$ defined before Theorem 1 in \cite{Irina}, reduces to $x\exp(a\cdot z)=x\alpha_T^x\beta_T^y$.
when $a\in \partial_0D=(\ln(\alpha_T),\ln(\beta_T))$; i.e. it reduces to  the harmonic function we used in the bridge case.

\subsubsection{Extending to $\pi(0,y,2)$}\label{labcascade}
Henceforth in this section we assume there is a ray on sheet~1; i.e. $\mu\cdot(\lambda_1/\lambda)-\lambda>0$.
We can bootstrap the asymptotics of $\pi(1,y,1)$ to obtain the asymptotics of $\pi(0,y,2)$.
It suffices to consider
\begin{eqnarray}
\pi(0,y,2)&=&\mu \pi(1,y,1)+\lambda_2 \pi(0,y-1,2)+\mu \pi(0,y+1,2)\mbox{ for }y\geq1 \label{equilibrium1}
\end{eqnarray}
where we recall $\pi(0,0,2)=0$ by definition so $\pi(0,0,1)=(1-\rho)$.
This immediately yields the lower bound $\pi(0,y,2)\geq \mu \pi(1,y,1)$.

On the other hand, we can define $\Pi_2(z)=\sum_{y=1}^{\infty}z^y\pi(0,y,2)$ and $\Pi_1(z)=\sum_{y=1}^{\infty}z^y\pi(1,y,1)$.
Multiplying (\ref{equilibrium1}) by $z^y$ and summing from $y=1$ we get
\begin{eqnarray*}
\Pi_2(z)&=&\mu \Pi_1(z)+\lambda_2 z \sum_{y=1}^{\infty}z^{y-1}\pi(0,y-1,2)+\frac{\mu}{z}\sum_{y=1}^{\infty}z^{y+1}\pi(0,y+1,2)\\
&=&\mu \Pi_1(z)+\lambda_2 z \Pi_2(z)+\frac{\mu}{z}(\Pi_2(z)-z\pi(0,1,2))
\end{eqnarray*}
so
\begin{eqnarray}\label{swap}
\Pi_2(z)&=&-\frac{\mu z}{\lambda_2 z^2-z+\mu}\left(\Pi_1(z)-\pi(0,1,2)\right).
\end{eqnarray}

We wish to show (\ref{swap}) is analytic on a disk of radius greater than $\beta_T$ except for a singularity at $\beta_T$.
Now the left hand side above is analytic at least in the ball of radius $\rho^{-1}$ since $\pi(0,y,2)\leq (1-\rho)\rho^y$
so the right hand side is also.
$\lambda_2 z^2-z+\mu=0$ has real roots since $\lambda_2\mu<\lambda\mu<1/4$ using $\lambda+\mu=1$ and stability. The root
$(1-\sqrt{1-4\mu\lambda_2}/(2\lambda_2)$  is less than $1$ if and only if
$1-2\lambda_2<\sqrt{1-4\mu\lambda_2}$; i.e. if and only if $\lambda_2+\mu<1$ and this is always true.
Hence the pole at $(1-\sqrt{1-4\mu\lambda_2}/(2\lambda_2)$ must then cancel a factor in the numerator.

On the other hand the root $(1+\sqrt{1-4\mu\lambda_2}/(2\lambda_2)$ is greater than $\beta_T$ if and only if
\begin{eqnarray}\label{equiv}
\sqrt{1-4\mu\lambda_2}+4\sqrt{\mu\lambda_1}-1&>&0.
\end{eqnarray}
For fixed $\lambda_1\in [0,1/2]$. the above function is decreasing as we see by setting $\mu=1-\lambda_1-\lambda_2$ and by taking the partial derivative in $\lambda_2$.
Since we have a ray on sheet~1, $\lambda_1\rho^{-1}>\lambda$; i.e. $f(\lambda_2)>1$ where
\begin{eqnarray}\label{dependson}
f(\lambda_2)&=&\lambda_1(1-\lambda_1-\lambda_2)/(\lambda_1+\lambda_2)^2.
\end{eqnarray}
Moreover, for fixed $\lambda_1$, $f(\lambda_2)$ is decreasing in $\lambda_2$ again by taking partial derivatives.
Since $f(0)=(1-\lambda_1)/\lambda_1\geq 1$ for $0\leq \lambda_1\leq 1/2$ it follows we can define
$\lambda_2^*=g(\lambda_1)$ be the unique value such that $f(\lambda_2^*)=1$.
Hence $\lambda_1(1-\lambda_1-\lambda_2^*)=(\lambda_1+\lambda^*_2)^2$ that is, finding the positive root,
$\lambda_2^*=(-3\lambda_1+\sqrt{\lambda_1^2+4\lambda_1})/2$.
We conclude that (\ref{equiv}) holds if
it holds on the curve $(\lambda_1,\lambda_2^*)$.

Now (\ref{equiv}) certainly holds if $4\sqrt{\mu\lambda_1}-1\geq 0$ since $\sqrt{1-4\mu\lambda_2}>0$ if $\lambda_1>0$.
Alternatively assume  $4\sqrt{\mu\lambda_1}-1< 0$. (\ref{equiv}) holds if and only if
\begin{eqnarray*}
1-4\mu\lambda_2^*&>&(1-4\sqrt{\mu\lambda_1})^2\mbox{ or equivalently if and only if}\\
2\sqrt{\mu\lambda_1}&>&\lambda_2^*\mu+4\mu\lambda_1\mbox{ or equivalently if and only if}\\
2\sqrt{\mu\lambda_1}&>&\frac{(-3\lambda_1+\sqrt{\lambda_1^2+4\lambda_1})}{2}\mu+4\mu\lambda_1\mbox{ or equivalently if and only if}\\
4&>&5\sqrt{\mu\lambda_1}+\sqrt{\lambda_1+4}.
\end{eqnarray*}
The latter inequality is certainly true because $\sqrt{\mu\lambda_1}\leq 1/4$ and $\sqrt{\lambda_1+4}\leq \sqrt{4.5}$
and $5/4+\sqrt{4.5}<5$. We conclude that (\ref{equiv}) holds.

By Theorem \ref{threeplus}, $\Pi_1(z)$ has radius of convergence $\beta_T$.  We shift the pole in $\Pi_1(z)$ to $1$
by defining $w=z/\beta_T$ to give
\begin{eqnarray}\label{swap2}
\Pi_2(\beta_T w)&=&-\frac{\mu \beta_T w}{\lambda_2 (\beta_T w)^2-\beta_T w+\mu}\left(\Pi_1(\beta_T w)-\pi(0,1,2)\right).
\end{eqnarray}
The right hand side of the above is analytic in $w$ in the disk of radius greater than $1$ except for the singularity at $1$.
To calculate the asymptotics of $\pi(0,y,2)$ we use the results in \cite{Odd}. Hence, as $w\to 1$,
$$\Pi_2(\beta_T w)\sim -\frac{\mu \beta_T }{\lambda_2 \beta_T^2-\beta_T +\mu}\left(\Pi_1(\beta_T w)-\pi(0,1,2)\right).$$
 Note that since $\beta_T$ is between the roots of $\lambda_2 z^2-z+\mu=0$ so
$\lambda_2 \beta_T^2-\beta_T+\mu<0$. Consequently the coefficient of (\ref{swap2}) is positive and modulo this positive constant the asymptotics of
$\pi(0,y,2)$ are the same as $\pi(1,y,1)$.

\begin{corollary}\label{goodenough}
If $\mu\cdot(\lambda_1/\lambda)-\lambda>0$  then as $y\to\infty$,
$\pi(0,y,2)\sim C_1\beta_T^{-y}   y^{-3/2}$ where $C_1$ is a fixed constant.
\end{corollary}

\subsection{The Cascade case}
There are two distinct ways to reach $(x,1,2)$. When $\beta_T\leq \beta_E$ we have cascade paths. When  $\beta_T> \beta_E$ we have cascade paths.
 We now estimate the asymptotics of
$\pi(x,0,1)$ in the cascade case. Therefore, in addition to our  assumption of a ray on sheet~1 in this subsection, we  assume $\beta_T\leq \beta_E$. In the cascade case
the large deviation path to $(x,1,2)$ first climbs the $y$-axis and then cascades across to $(x,1,2)$.

\subsubsection{Extending to $\pi(x,1,2)$  in the Cascade case}

We use the technique developed in \cite{Ivo} and \cite{Jesse}.
First we find a   measure $\psi$ which is invariant on the second sheet  of the form $\psi(x,y)=\psi_0(y)\gamma_T^{-x}\beta_T^{-y}$
where $\psi(x,0)=0$. $\psi(0,y)$ has roughly the same asymptotics as $\pi(0,y,2)$.
The associated time reversal $\overleftarrow{M}$ with respect to $\psi$ has kernel $\overleftarrow{K}$
and we note $\overleftarrow{K}((x,1),(x,0))=0$ because $\psi(x,0)=0$. Moreover $\overleftarrow{M}$ has negative drift; i.e.
$\overleftarrow{K}((x,y),(x-1,y))=\gamma_T^{-1}\lambda_1$.

If $\psi$ is invariant on sheet~2 at $(x,y)$ then
\begin{eqnarray*}
\lefteqn{\sum_{s,t}\overleftarrow{K}((x,y),(s,t))\frac{\pi(s,t)}{\psi((s,t)}=\sum_{s,t}\frac{\psi((s,t)}{\psi(x,y)} \overline{K}((s,t),(x,y))\frac{\pi(s,t)}{\psi((s,t)}}\\
&=&\frac{1}{\psi(x,y)}\sum_{s,t}\pi(s,t) \overline{K}((s,t),(x,y))=\frac{\pi(x,y)}{\psi(x,y)}
\end{eqnarray*}
so $\pi/\psi$ is harmonic for $\overleftarrow{K}$ and  $\pi(\overleftarrow{M}(n))/\psi(\overleftarrow{M}(n))$ is a martingale.

Let $\Delta=\{(0,y,2):y\geq 1\}$.
Recall
\begin{eqnarray*}
& &\pi(x,1,2)\\
&=&\sum_{y}\pi(0,y,2)E_{(x,1,2)}[\mbox{\# visits to $(x,1,2)$ before $M$ returns to $\Delta$} ].
\end{eqnarray*}
If we redefine $M$ as having a killing at $(x,1,2)$ if there is a transition to $(x,0,1)$, this will not change the above representation.
Now do the time reversal with respect to $\psi$ starting at $(x,1,2))$ or use the martingale property of $\pi/\psi$ to get
\begin{eqnarray}\label{beautiful}
\frac{\pi(x,1,2)}{\psi(x,1,2)}&=&E_{(x,1,2)}[\frac{\pi(\overleftarrow{M}(\tau))}{\psi(\overleftarrow{M}(\tau))}]
\end{eqnarray}
where $\tau$ is the first time $\overleftarrow{M}$ hits $\Delta$.

We now show $\psi$ actually exists and that
$\frac{\pi(0,y,2)}{\psi(0,y,2)}=\frac{\pi(0,y,2)}{\psi_0(y)\beta_T^{-y}}$ is bounded.
We pick $\gamma_T$ so
\begin{eqnarray}
\gamma_T\lambda_1+\beta_T^{-1}\mu+\beta_T\lambda_2=1.\label{sumstoone}
\end{eqnarray}
In order that $\psi$ be invariant it must satisfy
$$\psi_0(y)\gamma_T^{-x}\beta_T^{-y}=\psi_0(y)\gamma_T^{-(x-1)}\beta_T^{-y}\lambda_1+\psi_0(y+1)\gamma_T^{-x}\beta_T^{-(y+1)}\mu+\psi_0(y-1)\gamma_T^{-x}\beta_T^{-(y-1)}\lambda_2;$$
that is
$$\psi_0(y)=\gamma_T\lambda_1\psi_0(y)+\beta_T^{-1}\mu\psi_0(y+1)+\beta_T\lambda_2\psi_0(y-1).$$
Because of (\ref{sumstoone}), $\psi_0$ is harmonic for the random walk with probability transition kernel
$$\hat{\mathcal{K}}^T(y,y)=s=\gamma_T\lambda_1, \hat{\mathcal{K}}^T(y,y+1)=u=\beta_T^{-1}\mu,\hat{\mathcal{K}}^T(y,y-1)=d=\beta_T\lambda_2.$$
Up to constants there is a unique positive solution with $\psi_0(0)=0$. This follows as a consequence of the nearest neighbour character of $\hat{\mathcal{K}}^T$. Knowing $\psi_0(0)=0$ and the value of $\psi_0(1)$ allows us to iteratively
determine all of $\psi_0$.

By inspection the unique positive solution with $\psi_0(0)=0$ is $\psi_0(y)=1-(d/u)^y$ if $d<u$.
But  $d<u$ means
$\beta_T\lambda_2<\beta_T^{-1}\mu$; i.e.
$$\beta_T<\sqrt{\frac{\mu}{\lambda_2}} =\beta_E$$
and this is true by hypothesis.
 Using Corollary \ref{goodenough},
$$\frac{\pi(0,y,2)}{\psi(0,y,2)}=\frac{\pi(0,y,2)}{\psi_0(y)\beta_T^y}\sim C_1  y^{-3/2}.$$
If $d=u$ then the unique solution up to constants is $\psi_0(y)=y$. Hence,
$$\frac{\pi(0,y,2)}{\psi(0,y,2)}=\frac{\pi(0,y,2)}{\psi_0(y)\beta_T^y}\sim C_1  y^{-5/2}.$$

 $\overleftarrow{M}$ drifts north-west if $d>u$ or west if $d=u$. Consequently $\overleftarrow{M}(\tau)$ is distributed higher up the $y$-axis as the starting point
$(x,1,2)$ tends to infinity. Define $C(x)=E_{(x,1,2)}[C_1  (\overleftarrow{M}(\tau))^{-3/2}]$ if $d<u$ or
$C(x)=E_{(x,1,2)}[C_1  (\overleftarrow{M}(\tau))^{-5/2}]$ if $d=u$. Either way $C(x)$ tends to zero at a polynomial rate as $x\to\infty$.
Consequently from (\ref{beautiful}) we conclude
\begin{theorem}\label{goodenough2}
If $\mu\cdot(\lambda_1/\lambda)-\lambda>0$  and $\beta_T\leq\beta_E$ then as $x\to\infty$,
$\pi(x,1,2)\sim C(x)\beta_T^{-1} \gamma_T^{-x}  $ where $C(x)\to 0$ as $x\to\infty$.
\end{theorem}

\subsubsection{Extending to $\pi(x,0,1)$}\label{extension-xonone}
By the method used in Corollary \ref{goodenough} we remark that using the equilibrium equation
$$\pi(x,0,1)=\lambda_1\pi(x-1,0,1)+\mu\pi(x+1,0,1)+\mu\pi(x,1,2)$$ and multiplying by $z^x$ and summing from $x=1$ to infinity we get
$$\Pi_1(z)=\lambda_1 z(\Pi_1(z)+\pi(0,0,1))+\frac{\mu}{z}(\Pi_1(z)-z\pi(1,0,1))+\mu\Pi_2(z)$$
where $\Pi_1(z)=\sum_{x=1}^{\infty}z^x\pi(x,0,1)$ and $\Pi_2(z)=\sum_{x=1}^{\infty}z^x\pi(x,1,2)$.
Hence
$$\Pi_1(z)(1-\lambda_1 z-\mu/z)=\mu\Pi_2(z)+\lambda_1 z\pi(0,0,1)-\mu \pi(1,0,1) \mbox{ or }$$
\begin{eqnarray}\label{doitagain}
\Pi_1(z)&=&-\frac{z}{\lambda_1 z^2-z+\mu}(\mu\Pi_2(z)+\lambda_1 z\pi(0,0,1)-\mu \pi(1,0,1)).
\end{eqnarray}

The smaller root of $\lambda_1 z^2-z+\mu=0$ is at
$r_{-}=(1-\sqrt{1-4\mu\lambda_1}/(2\lambda_1)$ which is less than $1$ if and only if
$1-2\lambda_1<\sqrt{1-4\mu\lambda_1}$ and this is always true.
Hence the pole at $(1-\sqrt{1-4\mu\lambda_1}/(2\lambda_1)$ must then cancel a factor in the numerator.

On the other hand the root $r_{+}=(1+\sqrt{1-4\mu\lambda_1}/(2\lambda_1)$ is greater than $\alpha_E$ if and only if
\begin{eqnarray*}
 \frac{(1+\sqrt{1-4\mu\lambda_1}}{2\lambda_1}>\frac{1-2\sqrt{\lambda_2\mu}}{2\lambda_1}
\end{eqnarray*}
and this is always true - just square both sides and simplify. Hence
\begin{eqnarray}\label{orderofroots}
 r_{-}<1<\rho^{-1}<\gamma_T<\alpha_E<r_{+}.
\end{eqnarray}

 By Theorem \ref{goodenough2},
$\Pi_2(z)$ has radius of convergence $\gamma_T$ in the cascade case and we shall see later in Theorem \ref{four}, $\Pi_2(z)$ has radius of convergence $\alpha_E$ in the bridge case.
Hence the right hand side of (\ref{doitagain}) is analytic in a disk of
radius $\gamma_T$ in the cascade case and $\alpha_E$ in the bridge case. As in Subsection \ref{labcascade}, we shift the singularity by taking either
$w=z/\gamma_T$ or $w=z/\alpha_E$  and apply the results in \cite{Odd} to conclude $\pi(x,0,1)$ has the same asymptotics as
$\pi(x,1,2)$ multiplied by the constant
$-\gamma_T/(\lambda_1 \gamma_T^2-\gamma_T+\mu)$ in the cascade case
and by the constant
$-\alpha_E/(\lambda_1 \alpha_E^2-\alpha_E+\mu)$ in the bridge case. We remark that these constants are indeed positive because
$\alpha_E$ and $\gamma_T$ lie between the roots of $\lambda_1 z^2-z+\mu=0$  because  of (\ref{orderofroots}).
We have established
\begin{theorem}\label{goodenough3}
If $\tilde{\lambda}_1-\tilde{\mu}>0$  then as $x\to\infty$,
$\pi(x,0,1)\sim C_2\pi(x,1,2)$. where $C_2$ is a constant.
\end{theorem}
 \subsubsection{Asymptotics of $\pi(b^u(\ell))$ for $\alpha_T<e^{u_1}\leq \gamma_T$ in the cascade case}
Using the results in by Appendix \ref{N-T}, pick a twist such that $\alpha_T<e^{u_1}\leq \gamma_T$. The mean of the twisted kernel points  north-east. We can now repeat the calculations in Theorem \ref{one} with $h(x,y)=\exp(u_1 x+u_2 y)$. We take $\Delta=\{(x,0,1):x\geq 1\}$
and $\blacktriangle= \{(x,y):x\geq 1,y\geq 1\}^c$ inside $\ZZ^2$. The key fact to check is that $\sum_{x\geq 0} h(x,0) \pi(x,0)<\infty$. But
\begin{eqnarray*}
\sum_{x\geq 0} h(x,0) \pi(x,0)&=&\sum_{x\geq 0} e^{u_1 x} \pi(x,0)\\
&\leq& \sum_{x\geq 0}e^{u_1 x} C\gamma_T^{-x}
\end{eqnarray*}
by Theorems \ref{goodenough3} and \ref{goodenough2} where $C$ is some constant.
This sum is finite since $e^{u_1}\leq \gamma_T$
and we conclude
\begin{theorem}\label{upray}
If  $\tilde{\lambda}_1>\tilde{\mu}$ and $\beta_T\leq \beta_E$ and $\alpha_T<e^{u_1}\leq \gamma_T$ then
\begin{eqnarray*}
\pi(d^{u}(\ell))&\sim&B^{u}e^{-u\cdot d^{u}(\ell)}\frac{1}{\sqrt{2\pi \ell/|m^{}|}}
\end{eqnarray*}
where $$B^{u}=([|Q^{u}|(m^{u}\cdot \Sigma^{u} m^{})]^{-1/2})(\sum_{z\in \Delta}h(z)\pi(z)P_z[\mathcal{M}\mbox{ never hits }\blacktriangle]).$$
\end{theorem}
We conclude that large deviations on sheet~1 in north-north-east directions is a ray. However for north-east-east directions on sheet~1 the large deviation path
is not a ray and we investigate these paths next.

 \subsubsection{Asymptotics of $\pi(b^u(\ell))$ for $\gamma_T<e^{u_1}$ in the cascade case}
 Consider a point $w$ such on sheet~1 that $\mgf(w)=1$  and $e^{w_1}> \gamma_T$. Let $(x^w,y^w)=<\ell\cdot d^w>$.
 Construct an invariant measure on the  sheet~1 of the form $\psi(x,y)=e^{-(u_1 x+u_2y)}$ where $u$ belongs to the egg for sheet~1 such that $e^{u_1}=\alpha_E$.
 This is possible because $\gamma_T<\alpha_E$ with the argument above.  Consequently a line dropped from $(\gamma_T,\beta_T)$ must hit the egg on sheet~1 at some point $(\gamma_T,\delta_T)$ where $\delta_T=e^{u_2}$.
  Construct an invariant measure on the  sheet~1 of the form $\psi(x,y)e^{-(u_1 x+u_2y)}=\gamma_T^{-x}\delta_T^{-y}$. Again using time reversal
  \begin{eqnarray*}
\frac{\pi(x,y,1)}{\psi(x,y)}&=&E_{(x,y,1)}[\frac{\pi(\overleftarrow{M}(\tau))}{\psi(\overleftarrow{M}(\tau))}]\\
&\sim&C_2\cdot \beta_T^{-1} E_{(x,y,1)}[C(\overleftarrow{M}(\tau))]\beta_T^{-1} =D(x,y)
\end{eqnarray*}
by Theorems \ref{goodenough2} and \ref{goodenough3}. $D(x,y)$ tends to zero as $x\to\infty$ at a polynomial rate.
This gives
 $$\pi(x^w,y^w,1)\sim D(x^w,y^w)\gamma_T^{-x^w}\delta_T^{-y^w}$$
 as $x^w\to\infty$.
 In the cascade case the large deviation path from $(0,0,1)$ to $x^w,y^w,1$ is a bridge up the $y$-axis on sheet~1 followed by a jump to sheet~2 followed by a cascade across sheet~2 to some point $(x,1,2)$
 followed by a transition to sheet~1 and then along the path twisted by $(e^{u_1},e^{u_2})=(\gamma_T,\delta_T)$ to  $(x^w,y^w,1)$.

\subsection{The bridge case}
 Here we study
the bridge case where  $\beta_T>\beta_E$. In the bridge case
the large deviation path skims along the $x$-axis on sheet~2 to reach $(x,1,2)$.

\subsubsection{Extending to $\pi(x,1,2)$ in the Bridge case}
 We repeat the argument in Subsection \ref{upbridge}.
We define $\Delta=\{(0,y,2):y\geq 1\}$ and $\blacktriangle=\{(x,y,2):x\geq 0,y>0\}^c$.
We consider the \textit{free} kernel $\overline{K}$ with killing probability $K((x,1,2),(x,0,1))$ on points $(x,1)$ where there are transitions from sheet~2 to sheet~1.
To twist the free kernel $\overline{K}$ to make a bridge path consider a {\em rough} $h_E$-transform where $h_E(x,y)=\alpha_E^x\beta_E^y$
and where we already obtained  $\beta_E=\sqrt{\mu/\lambda_2}$ and
 $\alpha_{E}=(1-2\sqrt{\mu\lambda_2})/(2\lambda_1)$. We again remark that $(\rho^{-1},\rho^{-1})$ lies on Curve~2 so it follows that
 $\alpha_E>\rho^{-1}>\alpha_T$.

Let $\mathcal{K}_E$ denote the $h_E$-transformed kernel.  The Markovian part (the $y$-component) of $\mathcal{K}_E$ is $s=\hat{\mathcal{K}}_E(y,y)=\lambda_1\alpha_E $ for $y\geq 1$ and
$$\hat{\mathcal{K}}_E(y,y+1)=\hat{\mathcal{K}}_E(y,y-1)=\sqrt{\mu\lambda_2}=p\mbox{ for } y\geq 2. $$
Hence $\hat{\mathcal{K}}_E(y,y-1)=p$ for $y\geq 2$. At $y=1$
$\hat{\mathcal{K}}_E(1,2)=\sqrt{\mu\lambda_2}$ but we define a transition
$K((x,1,2),(x,0,1))$ to be a killing so
$\hat{\mathcal{K}}_E(1,0)=0$ since this is a transition off sheet~1.
Hence $\hat{\mathcal{K}}_E$ has a killing $\kappa=\sqrt{\mu\lambda_2}$ at point $1$.
The mean drift in the $x$-direction at $(x,y)$ for $y\geq 1$ is
$d_{+}=\lambda_1\alpha_E$.

Again note that $\hat{h}_0(y)=y$ is harmonic for $\hat{\mathcal{K}}_E$ and the associated $\hat{h}_0$-transform gives a Markovian kernel
satisfying the conditions of Theorem 1 in \cite{kesten}. Consequently $\hat{h}_0$ is unique harmonic function for $\hat{\mathcal{K}}_E$ and consequently $h(x,y)=\hat{h}_0(x)h_E(x,y)=y\alpha_E^x\beta_E^y$
is harmonic for $\overline{K}$.
Again the $h$-transform of $\overline{K}$ is denoted by $\overline{\mathcal{K}}$ which is associated with the twisted chain $\mathcal{M}$ and the associated potential
is $\overline{\mathcal{G}}$.
$\overline{\mathcal{K}}_{\blacktriangle}$ is the kernel  killed on hitting $\blacktriangle$.
We recall the representation (\ref{firstrep2}),
\begin{eqnarray}\label{firstrep3plus}
 \pi(x,1,2)&=&h(x,1)^{-1}\sum_{y\geq 0}h(0,y)\pi(0,y,2) \overline{{\cal G}}_{\blacktriangle}((0,y);(x,1))\nonumber\\
 &=&\beta_E\alpha_E^{-x}\sum_{y}y\beta_E^{y}\pi(0,y,2) \overline{{\cal G}}_{\blacktriangle}((0,y);(x,1)).
 \end{eqnarray}
By Corollary \ref{goodenough}, $\pi(1,y,2)=o(\beta_T^{-y})$ and by hypothesis $\beta_T>\beta_E$ so
 $$\sum_y y\beta_E^y \pi(0,y,2)<\infty.$$ 

Let $\gamma(0,y)=h_0(y)\alpha_E^{y}\pi(0,y,1)\chi_{\Delta}((0,y))$ and $\gamma_{\blacktriangle}(w)=\gamma \Pi^{\overline{{\cal K}}}_{\blacktriangle}(w)$
where $w\in\blacktriangle$ and $\Pi^{\overline{{\cal K}}}_{\blacktriangle}((0,y);w))$ is the probability of hitting $w$ if  the twisted chain started at $(0,y)$ enters $\blacktriangle$.
Again using Lemma \ref{Spitz}
\begin{eqnarray*}
& &\gamma\overline{{\cal G}}_{\blacktriangle}(x,1)\\
&=&\sum_x\gamma(0,y)\overline{{\cal G}}((0,y);(x,1))-\sum_y\gamma(0,y)\sum_{w\in\blacktriangle}\Pi^{\overline{{\cal K}}}_{\blacktriangle}((0,y);(w,0)))\overline{{\cal G}}((w,0);(x,1))\\
&=&\gamma  \overline{{\cal G}}(x,1)-  \gamma_{\blacktriangle}\overline{{\cal G}}(x,1) .
\end{eqnarray*}

We use the same method as in the proof of Theorem \ref{threeplus} to show that $x\to\infty$,
$\frac{\overline{{\cal G}}((0,y);(x,1))}{\overline{\mathcal{G}}((0,1);(1,y))}$ is uniformly bounded in $y$ and tends to $1$ as $x\to \infty$. If this is true then
$$\frac{\gamma  {\cal G}(1,y)}{\overline{\mathcal{G}}((1,0);(1,y)}= \sum_{x}x\alpha_E^{x}\pi(x,0,1)\frac{\overline{\mathcal{G}}((x,0);(1,y))}{\overline{\mathcal{G}}((1,0);(1,y))}
\to \sum_{y}y\alpha_E^{y}\pi(0,y,1)$$ and
\begin{eqnarray*}
& &\frac{\gamma_{\blacktriangle} \overline{{\cal G}}(x,1) }{\overline{\mathcal{G}}((0,1);(x,1))}= \sum_{y}x\alpha_E^{y}\pi(0,y,2)\sum_{w\in\blacktriangle}\Pi^{\overline{{\cal K}}}_{\blacktriangle}((0,y);w))\frac{{\mathcal{G}}(w;(x,1))}{\overline{\mathcal{G}}((0,1);(x,1))}\\
&\to& \sum_{y}y\alpha_E^{y}\pi(0,y,2)\mathcal{P}_{(0,y)}[\mathcal{M}\mbox{ hits }\blacktriangle].
\end{eqnarray*}
In which case,
$$\pi(x,1,2)\sim\beta_E^{-1}\alpha_E^{-x}\overline{\mathcal{G}}((1,0);(1,y)\sum_{y\in \Delta}y\beta_E^{y}\pi(0,y,2) \mathcal{P}_{(0,y)}[\mathcal{M}\mbox{ never hits }\blacktriangle].$$
We have therefore stripped out the exponential decay $\alpha_E^{-x}$ out of $\pi(x,1,2)$ leaving the polynomial decay $\overline{\mathcal{G}}((0,1);(x,1)$.

Next recall $\mathcal{G}((0,1);(x,1)=\mathcal{G}_E((0,1);(x,1)\frac{\hat{h}_0(1)}{\hat{h}_0(1)}=\mathcal{G}_E((0,1);(x,1)$. Again we note
$\mathcal{G}_E((0,1);(x,1)$ was analyzed analytically in \cite{FM:BridgesExact}. This again requires relabeling or shifting the first quadrant so $(0,1)\to (0,0)$ and $(x,1)\to (x,0)$.
Denote the relabelled or shifted kernel by $\mathcal{K}_s$ and the associated potential by $\mathcal{G}_s$.
Hence $\mathcal{K}_s((x,y);((u,v))=\mathcal{K}_E((x,y+1);((u,v+1))$ and
$\mathcal{G}_E((0,1);(x,1)=\mathcal{G}_s((0,0);(x,0)$.
In these new coordinates the probability of killing $\kappa$ at any point $(x,0)$ is $\mathcal{K}_E((x,1);(x,0))=\mu \beta_E^{-1}$
and $$p_0:=\hat{\mathcal{K}}_s(0,1)=\hat{\mathcal{K}}_s(y,y+1)=p=\hat{\mathcal{K}}_s(y,y-1).$$
By Proposition 1 in \cite{FM:BridgesExact},
${\cal G}_s((0,0);(0,x)) \sim C_+ \,  x^{-3/2}$
where
$C_+ = \frac{p_0}{\kappa^2}\sqrt{\frac{d_+}{2 \pi (1 - s)}}$
\begin{theorem}\label{four}
If $\tilde{\lambda}_1-\tilde{\mu}>0$  and  $\beta_T>\beta_E$ then
$$\pi(x,1,2)\sim\beta_E^{-1}\alpha_E^{-x}C_+   x^{-3/2}\sum_{y}y\beta_E^{y}\pi(0,y,2) \mathcal{P}_{(0,y)}[\mathcal{M}\mbox{ never hits }\blacktriangle].$$
\end{theorem}

\subsubsection{Extending to $\pi(x,0,1)$ in the bridge case}\label{extension-xonone2}

This extension was accomplished in Theorem \ref{goodenough3} for both the cascade and bridge cases.

 \subsubsection{Asymptotics of $\pi(b^u(\ell))$ for $\alpha_T<e^{u_1}\leq \alpha_E$ in the bridge case}
Again, using the results in by Appendix \ref{N-T}, pick a twist such that $\alpha_T<e^{u_1}\leq \alpha_E$. The mean of the twisted kernel points  north-east. We can now repeat the calculations in Theorem \ref{one} with $h(x,y)=\exp(u_1 x+u_2 y)$. We take $\Delta=\{(x,0,1):x\geq 1\}$
and $\blacktriangle= \{(x,y):x\geq 1,y\geq 1\}^c$ inside $\ZZ^2$. The key fact to check is that $\sum_{x\geq 0} h(x,0) \pi(x,0)<\infty$. But
\begin{eqnarray*}
\sum_{x\geq 0} h(x,0) \pi(x,0)&=&\sum_{x\geq 0} e^{u_1 x} \pi(x,0)\\
&\leq& \sum_{x\geq 0} \alpha_E^x \pi(x,0)<\sum_{x\geq 0}  \alpha_E^x  C x^{-3/2}\alpha_E^{-x}
\end{eqnarray*}
by Theorems \ref{goodenough3} and \ref{four} where $C$ is some constant.
Since this sum is finite we conclude
\begin{theorem}\label{upray2}
If  $\tilde{\lambda}_1>\tilde{\mu}$ and $\alpha_T<e^{u_1}\leq \alpha_E$ then
\begin{eqnarray*}
\pi(d^{u}(\ell))&\sim&B^{u}e^{-u\cdot d^{u}(\ell)}\frac{1}{\sqrt{2\pi \ell/|m^{}|}}
\end{eqnarray*}
where $$B^{u}=([|Q^{u}|(m^{u}\cdot \Sigma^{u} m^{})]^{-1/2})(\sum_{z\in \Delta}h(z)\pi(z)P_z[\mathcal{M}\mbox{ never hits }\blacktriangle]).$$
\end{theorem}
We conclude that large deviations on sheet~1 in north-north-east directions is a ray. However for north-east-east directions on sheet~1 the large deviation path
is not a ray and we investigate these paths next.

 \subsubsection{Asymptotics of $\pi(b^u(\ell))$ for $\alpha_E<e^{u_1}$ in the bridge case}

The large deviation path to these points is not a ray but rather a path the does a large deviation on sheet~2 before returning to sheet~1.
In the Bridge case consider a point $w$ such on sheet~1 that $\mgf(w)=1$  and $e^{w_1}> \alpha_E$. Let $(x^w,y^w)=<\ell\cdot d^w>$.
 Construct an invariant measure on the  sheet~1 of the form $\psi(x,y)=e^{-(u_1 x+u_2y)}$ where $u$ belongs to the egg for sheet~1 such that $e^{u_1}=\alpha_E$.
 This is possible because the asymptote of the egg $\lambda_1\alpha+\lambda_2\beta+\mu \alpha^{-1}=1$ as $\beta\to 0$ is given by
 $\lambda_1\alpha+\mu \alpha^{-1}=1$; i.e. $\alpha=(1+\sqrt{1-4\lambda_1\mu})/(2\lambda_1)$ which is greater than $\alpha_E=(1-2\sqrt{\mu\lambda_2})/(2\lambda_1)$.
 Consequently a line dropped from $(\alpha_E,\beta_E)$ must hit the egg on sheet~1 at some point $(\alpha_E,\gamma_E)$ where $\gamma_E=e^{u_2}$.
 Now construct the time reversal with $\psi$. Note that the mean drift of the time reversal is $-m^w$ in direction $-d^w$. Next recall (\ref{beautiful}) so
 \begin{eqnarray*}
\frac{\pi(x,y,1)}{\psi(x,y)}&=&E_{(x,y,1)}[\frac{\pi(\overleftarrow{M}(\tau))}{\psi(\overleftarrow{M}(\tau))}].
\end{eqnarray*}
 But $\pi(x,0,1)/\psi(x,0)\sim  C_4 x^{-3/2}$ by Theorems \ref{four} and \ref{goodenough3} where $C_4$ is some constant. The time reversal drifting in direction $-d^w$
 will hit the $x$-axis with probability one. The value of the function
 $$C(x,y)=E_{(x,y,1)}[\frac{\pi(\overleftarrow{M}(\tau))}{\psi(\overleftarrow{M}(\tau))}]\sim
 E_{(x,y,1)}[C_4(\overleftarrow{M}(\tau))^{-3/2}]$$
 clearly tends to zero as $x\to\infty$ at a polynomial rate. This gives
 $$\pi(x^w,y^w,1)\sim C(x^w,y^w)\gamma_E^{-x^w}\beta_E^{-y^w}.$$
 Clearly the large deviation path from $(0,0,1)$ to $x^w,y^w,1$ is an immediate leap to sheet~2 followed by a bridge along the $x$-axis for a certain distance.
 Next the path leaps back to the first sheet and then follows the path twisted by $(e^{u_1},e^{u_2})=(\alpha_E,\gamma_E)$.


\section{Worse and Worse-the spiral-spiral case}\label{zigzag}
We assume $\tilde{\lambda}_1-\tilde{\mu}<0$ and $\tilde{\lambda}_2-\tilde{\mu}<0$. By adding these conditions we see that $\rho>1/2$.
We can investigate $\pi(x,y,s)$ for $s=1,2$ and for $x+y=\ell$ for $\ell$ large. Let
$\blacktriangle=\Delta=\{(0,0,1)\}$.
$h(x,y,s)=\rho^{-(x+y)}$ is harmonic on $\blacktriangle^{c}$. Using the harmonic function $\rho^{-(x+y)}$ we can produce the free twisted chain with kernel $\mathcal{K}$.
 We will first describe how large deviation paths to a point $(x,y,s)$ where $x+y=\ell$ occur.

Starting on sheet~1 at $(u,0,1)$, the mean number of steps for the $h$-transformed chain to drift across to the $y$ axis is $u/(\lambda-\rho^{-1}\lambda_1)$.
During that time the mean rise in $y$ is $\rho^{-1}\lambda_2 u/(\lambda-\rho^{-1}\lambda_1)$. This means that on average
the twisted chain hits the $y$ axis at $v=u+f_1\cdot u$ where
$$f_1=\frac{\rho^{-1}\lambda_2} {\lambda-\rho^{-1}\lambda_1}-1=\frac{\mu-\lambda}{\lambda-\rho^{-1}\lambda_1}.$$
Note that $f_1$ is positive since
$$\frac{\mu-\lambda}{\lambda-\rho^{-1}\lambda_1}>0\mbox{ if and only if }\lambda-\rho^{-1}\lambda_1>0$$
where the later expression is positive by hypothesis in the spiral-spiral case.

Similarly starting on sheet~2 at $(0,v,2)$, the mean number of steps to drift across to the $x$ axis is $v/(\lambda-\rho^{-1}\lambda_2)$.
During that time the mean rise in $x$ is $\rho^{-1}\lambda_1 v/(\lambda-\rho^{-1}\lambda_2)$. This means that on average
the twisted chain hits the $x$ axis at $v+f_2\cdot v$ where
$$f_2=\frac{\rho^{-1}\lambda_1} {\lambda-\rho^{-1}\lambda_2}-1=\frac{\mu-\lambda}{\lambda-\rho^{-1}\lambda_2}.$$
Again $f_2$ is positive since $\lambda-\rho^{-1}\lambda_2>0$ by hypothesis in the spiral-spiral case.
It follows that the spiral-spiral path essentially follows a sequence of similar triangles which grow exponentially.

We can make conjecture about the form of $\pi$. Let
$\mathcal{G}=\sum_{n=0}^{\infty}\mathcal{K}^n$  be the associated potential.
Let $\mathcal{K}_{\blacktriangle}$ be the taboo kernel $\mathcal{K}$ killed on $\blacktriangle$ and let
 $\mathcal{G}_{\blacktriangle}$ be the associated potential. We have  the representation
\begin{eqnarray}\label{secondrep}
 \pi(x,y,s)&=&(1-\rho)\rho^{x+y} {\cal G}_{\blacktriangle}((0,0,1);(x,y,s)).
 \end{eqnarray}

Let $\gamma(x,y)=(1-\rho){\cal G}_{\blacktriangle}((0,0,1);(x,y,s))$ and $\kappa(x,y)=(1-\rho){\cal G}_{\blacktriangle}((0,0,1);(x,y,2))$ so
  $\pi(x,y,1)=\rho^{x+y}\gamma(x,y)$ and $\pi(x,y,2)=\rho^{x+y}\kappa(x,y)$.
 Also let
 \begin{eqnarray}\label{simple}
 a=\lambda_1\rho^{-1}, b=\mu\rho \mbox{ and }c=\lambda_2\rho^{-1}.
 \end{eqnarray}
 Note that $b>a$ and $b>c$ in spiral-spiral case. $\mathcal{M}$ has transition probabilities given by $a$, $b$ and $c$ and
 \begin{eqnarray*}
 \gamma(x,y)=a\gamma(x-1,y)+b\gamma(x+1,y)+c\gamma(x,y-1)\mbox{ if }x,y\geq 1\\
 \kappa(x,y)=a\kappa(x-1,y)+b\kappa(x,y+1)+c\kappa(x,y-1)\mbox{ if }x,y\geq 1\\
 \gamma(x,0)=a\gamma(x-1,0)+b\gamma(x+1,0)+b\kappa(x,1)\mbox{ if }x\geq 1\\
 \kappa(0,y)=b\gamma(1,y)+b\kappa(0,y+1)+c\kappa(0,y-1)\mbox{ if }y\geq 1.
 \end{eqnarray*}
 Note that
 $$(1-\rho)\rho^{\ell}=\sum_{x+y=\ell}(\pi(x,y,1)+\pi(x,y,2))=\rho^{\ell}\sum_{x+y=\ell}(\gamma(x,y)+\kappa(x,y,2))$$
 so $\sum_{x+y=\ell}(\gamma(x,y)+\kappa(x,y))=(1-\rho)$

 It is reasonable to assume that, for $\ell$ large and $x+y=\ell$,
 $$\gamma(x,y)\sim \alpha(x)/\ell\mbox{ and }\kappa(x,y)\sim \beta(y)/\ell$$
 so $\sum_{x=1}^{\ell}\alpha(x)+\sum_{y=1}^{\ell}\beta(y)\sim (1-\rho)\ell$.
 $\alpha$ and $\beta$ approximately satisfy
 \begin{eqnarray}
 \frac{1}{\ell}\alpha(x)&=&\frac{a}{\ell-1}\alpha(x-1)+\frac{b}{\ell+1}\alpha(x+1)+\frac{c}{\ell-1}\alpha(x)\label{eqn1}\\
 \frac{1}{\ell}\alpha(\ell)&=&\frac{a}{\ell-1}\alpha(\ell-1)+\frac{b}{\ell+1}\alpha(\ell+1)+\frac{b}{\ell+1}\beta(1)\label{eqn2}\\
 \frac{1}{\ell}\beta(y)&=&\frac{a}{\ell-1}\beta(y)+\frac{b}{\ell+1}\beta(y+1)+\frac{c}{\ell-1}\beta(y-1)\label{eqn3}\\
 \frac{1}{\ell}\beta(\ell)&=&\frac{c}{\ell-1}\beta(\ell-1)+\frac{b}{\ell+1}\beta(\ell+1)+\frac{b}{\ell+1}\alpha(1)\label{eqn4}.
 \end{eqnarray}
 It follows from (\ref{eqn2}) and (\ref{eqn4}) that
 $$\frac{b}{\ell+1}\beta(1)=\frac{c}{\ell-1}\alpha(\ell)\mbox{ and }\frac{b}{\ell+1}\alpha(1)=\frac{a}{\ell-1}\beta(\ell)$$
 so to  order zero
 $b\beta(1)=c\alpha(\ell)$ and $b\alpha(1)=a\beta(\ell)$ .

 The $y$-axis acts like a turnstile on sheet~1 and the $x$-axis acts like a turnstile on sheet~2; i.e. they act like absorbing boundaries. Consequently
 if is reasonable to conjecture that
 \begin{eqnarray*}
\alpha(x)&=&C_1(1-\left(\frac{a}{b}\right)^x)-f\left(\frac{x}{\ell}\right)\\
\beta(y)&=&C_2(1-\left(\frac{c}{b}\right)^y)-g\left(\frac{y}{\ell}\right)
\end{eqnarray*}
where $f$ and $g$ are differentiable and
where $f(0)=0$ and $g(0)=0$.
Substitute $\alpha$ of this form into (\ref{eqn1}) and retain the terms of order zero and of order $1/\ell$.
This gives
\begin{eqnarray}
\lefteqn{C_1(1-\left(\frac{a}{b}\right)^x)-f\left(\frac{x}{\ell}\right)}\label{harmonic}\\
& &=\frac{a}{\ell-1}(C_1(1-\left(\frac{a}{b}\right)^{x-1})-f\left(\frac{x-1}{\ell-1}\right))\nonumber\\
& &+\frac{b}{\ell+1}(C_1(1-\left(\frac{a}{b}\right)^{x+1})-f\left(\frac{x+1}{\ell+1}\right))\nonumber\\
& &+\frac{c}{\ell-1}(C_1(1-\left(\frac{a}{b}\right)^x)-f\left(\frac{x}{\ell-1}\right))\nonumber
\end{eqnarray}

\begin{figure}
\begin{center}
\includegraphics{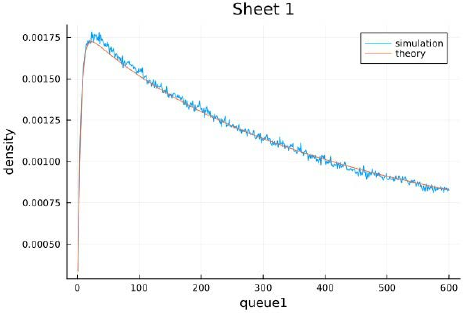}
\caption{Spiral-Spiral case: $\lambda_1=0.3$, $\lambda_2=0.15$, $\mu=0.55$\label{comparesim}}
\end{center}
\end{figure}
Now use the first order approximations
$$\frac{x-1}{\ell-1}-\frac{x}{\ell}\sim -\frac{1-x/\ell}{\ell},\frac{x+1}{\ell+1}-\frac{x}{\ell}\sim \frac{1-x/\ell}{\ell},\frac{x}{\ell-1}-\frac{x}{\ell}\sim -\frac{x/\ell}{\ell}$$
and
$$f\left(\frac{x-1}{\ell-1}\right)-f\left(\frac{x}{\ell}\right)\sim -f^{\prime}\left(\frac{x}{\ell}\right)\frac{(1-x/\ell)}{\ell},$$
and $f\left(\frac{x+1}{\ell+1}\right)-f\left(\frac{x}{\ell}\right)\sim f^{\prime}\left(\frac{x}{\ell}\right)\frac{(1-x/\ell)}{\ell}$ and
$$f\left(\frac{x}{\ell-1}\right)-f\left(\frac{x}{\ell}\right)\sim -f^{\prime}\left(\frac{x}{\ell}\right)\frac{x/\ell}{\ell}.$$
To first order (\ref{harmonic}) requires
$$f^{\prime}(s)((b-a)(1-s)+cs)+(a+c-b)f(s)=C_1(a+c-b)-C_1\left(\frac{a}{b}\right)^{\ell s}$$
where $s=x/\ell$. The function $(a/b)^{\ell s}$ goes to zero for $s>0$ as $\ell\to\infty$ so we discard it.
Hence we solve
$f^{\prime}(s)+r(s)f(s)=C_1r(s)$
where $r(s)=\frac{(a+c-b)}{((b-a)(1-s)+cs)}$.

Multiply the integrating factor $I(s)=((b-a)(1-s)+cs)$  and we get
$(I(s)f(s))^{\prime}= C_1(a+c-b)$. Integrating we get $I(s)f(s)=C_1(a+c-b)s+k$ where $k$ is a constant which must be $0$ if $f(0)=0$.
Hence $f(s)=C_1\frac{(a+c-b)s}{((b-a)(1-s)+cs)}$. Similarly
$g(s)=C_1\frac{(a+c-b)s}{((b-c)(1-s)+as)}$. Consequently
\begin{eqnarray}
\alpha(x)&=&C_1\left( (1-\left(\frac{a}{b}\right)^x)- \frac{(a+c-b)x/\ell}{((b-a)(1-x/\ell)+cx/\ell)}\right)\label{approx1}\\
\beta(y)&=&C_2\left((1-\left(\frac{c}{b}\right)^y)-\frac{(a+c-b)y/\ell}{((b-c)(1-y/\ell)+ay/\ell)}\right).\label{approx2}
\end{eqnarray}

Next the relations $b\beta(1)=c\alpha(\ell)$ and $b\alpha(1)=a\beta(\ell)$ to first order give
$$bC_2(1-\frac{c}{b})=c C_1(1-\frac{(a+c-b)}{c}\mbox{ and }bC_1(1-\frac{a}{b})=aC_2(1-\frac{a+c-b}{a}).$$
Both equations imply $C_2=C_1\frac{(b-a)}{(b-c)}.$

Finally we use the fact that
\begin{eqnarray}
\frac{1}{\ell}\sum_{x=1}^{\ell}\alpha(x)+\frac{1}{\ell}\sum_{y=1}^{\ell}\beta(y)\sim (1-\rho)\label{usesum}.
\end{eqnarray}
But
$$\frac{1}{\ell}\sum_{x=1}^{\ell}\alpha(x)\approx C_1-\frac{C_1}{\ell}\frac{a/b}{1-a/b}-C_1\int_0^1\frac{(a+c-b)s}{(b-a)(1-s)+cs}ds.$$
Since
$\int\frac{s}{p+qs}ds=\frac{1}{q^2}(p+qs-p\log(p+qs))$ we have $\int_0^1\frac{s}{(b-a)(1-s)+cs}ds=\frac{1}{(a+c-b)^2}((a+c-b)+(b-a)log((b-a)/c))$.
Hence
$$\frac{1}{\ell}\sum_{x=1}^{\ell}\alpha(x)\approx -\frac{C_1}{\ell}\frac{a/b}{1-a/b}-C_1\frac{(b-a)}{a+c-b}\log(\frac{b-a}{c})\mbox{ which is positive since }b<a+c.$$
Similarly $\frac{1}{\ell}\sum_{y=1}^{\ell}\beta(y)\approx -\frac{C_2}{\ell}\frac{c/b}{1-c/b}-C_2\frac{(b-c)}{a+c-b}\log(\frac{b-c}{a})$.
Using (\ref{usesum}) and $C_2=C_1\frac{(b-a)}{(b-c)}$ we get
\begin{eqnarray}
C_1&=&\frac{C(\ell)}{b-a} \mbox{ and }C_1=\frac{C(\ell)}{b-c}\mbox{ where }\label{foundC1}\\
C(\ell)&=&(1-\rho)(a+c-b)  (log(\frac{ac}{(b-a)(b-c)}) - \frac{1}{\ell}(a+c-b)(\frac{a}{(b-a)^2}+\frac{c}{(b-c)^2})    )^{-1}.\nonumber
\end{eqnarray}
(\ref{approx1}) and (\ref{approx2}) now give $\alpha$ and $\beta$.

To test our conjecture that
$\gamma(x,y)\sim \alpha(x)/\ell\mbox{ and }\kappa(x,y)\sim \beta(y)/\ell$
we simulated $n=10^6$ trajectories of $\mathcal{M}$  starting from $(0,0,1)$ using Julia \cite{Julia}. We count the number of visits $N(x,y,s)$
to $(x,y,s)$ before hitting $\Delta$ for $x+y=\ell$. Then
$$EN(x,y,1)={\cal G}_{\blacktriangle}((0,0,1);(x,y,1))\cdot n=\frac{\gamma(x,y)}{(1-\rho)}\cdot n$$
and
$$EN(x,y,2)={\cal G}_{\blacktriangle}((0,0,1);(x,y,2))\cdot n=\frac{\kappa(x,y)}{(1-\rho)}\cdot n.$$ Hence, if $N=\sum_{x+y=\ell}(N(x,y,1)+N(x,y,2))$, as $\ell\to\infty$,
$$\frac{N(x,y,1)}{N}\approx \frac{\gamma(x,y)}{\sum_{x+y=\ell}(\gamma(x,y)+\kappa(x,y))}= \frac{\gamma(x,y)}{(1-\rho)}\approx \frac{\alpha(x)}{\ell(1-\rho)}$$
and
$$\frac{N(x,y,2)}{N}\approx  \frac{\beta(y)}{\ell(1-\rho)}.$$

\begin{table*}
\caption{Values multiplied by $10^6$ and rounded}
\label{tablesmall}
\begin{tabular}{@{}lrrrrrrrrrr@{}}
\hline
 $x$  & \multicolumn{1}{c}{$1$}
& \multicolumn{1}{c}{$2$} & \multicolumn{1}{c}{$3$}
& \multicolumn{1}{c}{$4$} & \multicolumn{1}{c}{$3$}
& \multicolumn{1}{c}{$6$}
& \multicolumn{1}{c}{$7$} & \multicolumn{1}{c}{$8$}
& \multicolumn{1}{c}{$9$} & \multicolumn{1}{c}{$10$}\\
\hline
$N(x,y,1)/N$&  339  &594  &816 &1030 &1179 &1283  &1384  &1479  &1518  &1581\\
$\alpha(x)/(\ell(1-\rho))$&  334 & 605  &825 &1004 &1150  &1267 &1362 &1439  &1501  &1551  \\
\hline
\end{tabular}
\end{table*}

We took $\lambda_1  = 0.3$, $\lambda_2  = 0.15$ and $\mu= 0.55$. We took $\ell=600$. In Figure \ref{comparesim} the simulated curve  is $N(x,y,1)/N$ for $x+y=600$ while the theoretical
curve is $\alpha(x)/(\ell(1-\rho))$. The curves for $N(x,y,2)/N$ and $\beta(y)/(\ell(1-\rho))$ are similarly close.
The maximum relative error between $N(x,y,1)/N$ and $\alpha(x)/(\ell(1-\rho))$ for $x\in \{1,2,\ldots ,\ell\}$ is $5.2$ percent.
The fit for small values of $x$ is even better (see Table \ref{tablesmall}).
Our experimental results lead to the following conjecture.\\
 \noindent {\bf Conjecture:} In the spiral-spiral case, for  $x+y=\ell$,
$\pi(x,y,1))\sim \rho^{\ell}\frac{\alpha(x)}{\ell}$ and $\pi(x,y,2))\sim \rho^{\ell}\frac{\beta(y)}{\ell}$
as $\ell\to\infty$
where $\alpha$ is given at (\ref{approx1}) and  $\beta$ is given at (\ref{approx2})
and where $C_1$ and $C_2$ are given at (\ref{foundC1}).

\section{Traffic at a road closure}\label{trafficoverload}
A service period consists of the time it takes for a car to cross the section of road under repair. We will assume service time takes one unit of time. Traffic travels one way until the queue is emptied and all the cars complete the crossing. Then the cars in the other direction are released and traffic continues one way until that queue is emptied and all cars have completed the crossing. During a service period a random number of cars arrive at queue~1 having
density $f$ and a random number arrive at queue~2 having density $g$. We assume the number of arrivals in a period
are independent of the number of arrivals during other service periods and also independent of the arrivals at the other queue.
Moreover we assume $\lambda_1+\lambda_2<1$ where $\lambda_1=\sum_{n=0}^{\infty}nf(n)$ and $\lambda_1=\sum_{n=0}^{\infty}ng(n)$

We describe this system in the same way we did for the polling system; i.e.  state of the system is denoted by $(x,y,s)$ where $(x,y)$ is the joint queue length of queues 1 and 2 and $s \in \{1,2\}$ is the queue being served. When $(x,y) = (0,0)$, then $s = 1$.  Hence
\begin{eqnarray*}
K((x,y,1);(x+u-1,y+v,1))&=&f(u)g(v)\mbox{ for }x>1\\
K((1,y,1);(0,y+v,2))&=&f(0)g(v)\mbox{ for }x=1\\
K((x,y,2);(x+u,y+v-1,2))&=&f(u)g(v)\mbox{ for }y>1\\
K((x,1,2);(x+u,0,1))&=&f(u)g(0)\mbox{ for }y>1.
\end{eqnarray*}
We won't add a switch over time to avoid complications.

 Note that $(x,0,2)$ for $x\geq 0$ and $(0,y,1)$ for $y \geq 1$ are not in the state space $S$.
 Let $F(z)=\sum_{n=0}^{\infty}f(n)z^n$ and $G(w)=\sum_{n=0}^{\infty}g(n)w^n$.   We note that $\psi(\gamma)=\gamma^{-1}F(\gamma)G(\gamma)$ equals $1$ when $\gamma=1$.
 Moreover assuming both queues have nonzero probability of at least one arrival during a busy period
we have $\psi(\gamma)\to\infty$ as $\gamma\to\infty$. Next,
\begin{eqnarray*}
\frac{d}{d\gamma}\psi(\gamma)|_{\gamma=1}&=&-1+F^{\prime}(1)+G^{\prime}(1)\\
&=&-1+\lambda_1+\lambda_2<0.
\end{eqnarray*}
Finally $\psi(\gamma)$ is convex so there exists a unique point $\alpha>1$ such that $\psi(\alpha)=1$ unless either of $F(\alpha)$ or $G(\alpha)$ is infinite.
We will just assume finiteness.

 Define the function $h(x,y,s)=h(x,y,s,k)=\alpha^{x+y}$.  We first check $h$ is harmonic at points $(x,y,s)$ if $(x,y)\neq (0,0)$:
 \begin{eqnarray*}
 & &\sum_{u,v}h(x+u-1,y+v,1)f(u)g(v)=\frac{F(\alpha)G(\alpha)}{\alpha}h(x,y,1)=h(x,y,1) \mbox{ if }x>1\\
 & &\sum_{u>0,v}h(u,y+v,1)f(u)g(v)+h(0,y+v,2)f(0)g(v)=\frac{F(\alpha)G(\alpha)}{\alpha}h(x,y,1)\\
 & &=h(1,y,1) \mbox{ if }x=1\\
 & &\sum_{u,v}h(x+u,y+v-1,2)f(u)g(v)=\frac{F(\alpha)G(\alpha)}{\alpha}h(x,y,2)=h(x,y,2)\mbox{ if } y>1\\
& & \sum_{u,v>0}h(x+u,y+v-1,2)f(u)g(v)+h(x+u,0,1)f(u)g(0)\\
& &=\frac{F(\alpha)G(\alpha)}{\alpha}h(x,1,2)
=h(x,1,2) \mbox{ if }y=1.
 \end{eqnarray*}
 We can therefore twist the joint probability of $n$ customer arrivals during a service period to queue~1 and $m$ to queue~2 from
 $f(n)g(m)$ to $f(n)\alpha^ng(m)\alpha^m/\alpha$:
\begin{eqnarray*}
\lefteqn{\mathcal{K}((x,y,1);(x+m-1,y+n,1))}\\&=&K((x,y,1);(x+m-1,y+n,1))\frac{h(x+m-1,y+n,1)}{h(x,y,1)}\\
&=&f(n)\alpha^ng(m)\alpha^m/\alpha\mbox{ if }x>1
\end{eqnarray*}
\begin{eqnarray*}
\lefteqn{\mathcal{K}((1,y,1);(x+m-1,y+n,1,1))}\\
&=& K((1,y,1);(x+m-1,y+n,1,1))\frac{h(x+m-1,y+n,1,1)}{h(1,y,1)}\\
&=&        f(n)\alpha^ng(m)\alpha^m/\alpha.
\end{eqnarray*}
and
\begin{eqnarray*}
\lefteqn{\mathcal{K}((x,y,2);(x+m,y+n-1,2))}\\
&=&K((x,y,2);(x+m,y+n-1,2))\frac{h(x+m-1,y+n,2)}{h(x,y,2)}\\
&=&f(n)\alpha^ng(m)\alpha^m/\alpha
\end{eqnarray*}
\begin{eqnarray*}
\lefteqn{\mathcal{K}((x,1,2);(x+m,y+n-1,2,1))}\\
&=&K((x,1,2);(x+m,n,2))\frac{h(x+m,n,2,1)}{h(x,y,2)}\\
&=&f(n)\alpha^ng(m)\alpha^m/\alpha.
\end{eqnarray*}
These are all probability kernels by the definition of $\alpha$.

The joint density of the twisted increment on sheet~1 is
$$(f(n)\alpha^n/\alpha)(g(n)\alpha^m)=(\frac{f(n)\alpha^n/\alpha}{F(\alpha)/\alpha})(\frac{g(m)\alpha^m}{G(\alpha)}).$$
The marginal mean increment in $x$ is
$$\frac{1}{F(\alpha)/\alpha}\sum_{n=0}^{\infty}(n-1)f(n)\alpha^n/\alpha=\frac{\alpha}{F(\alpha)/\alpha}\frac{d}{d\gamma}(F(\gamma)/\gamma)|_{\gamma=\alpha}
=\alpha \frac{d}{d\gamma}(\log(F(\gamma)/\gamma))|_{\gamma=\alpha}.$$
The marginal mean increment in $y$ is
$$\frac{1}{G(\alpha)}\sum_{m=0}^{\infty}mg(m)\alpha^m
=\alpha \frac{d}{d\gamma}(\log(G(\gamma)))|_{\gamma=\alpha}.$$
The mean increment in the total number in the system is therefore
\begin{eqnarray*}
& &\alpha \frac{d}{d\gamma}(\log(F(\gamma)/\gamma))|_{\gamma=\alpha}+\alpha \frac{d}{d\gamma}(\log(G(\gamma)))|_{\gamma=\alpha}
=\alpha \frac{d}{d\gamma}(\log(\frac{F(\gamma)G(\gamma)}{\gamma}))\\
&=&\alpha \frac{d}{d\gamma}(\log(\psi(\gamma)))|_{\gamma=\alpha}.
\end{eqnarray*}
But this is positive by the construction of $\alpha$.  We conclude the twisted chain is transient on sheet~1. The same is true on sheet~2.

We again see there is a spiral or a ray on sheet~1 if
$$\alpha \frac{d}{d\gamma}(\log(F(\gamma)/\gamma))|_{\gamma=\alpha}<0,\mbox{ or }
\alpha \frac{d}{d\gamma}(\log(F(\gamma)/\gamma))|_{\gamma=\alpha}>0.$$
Similarly there is a spiral or a ray on sheet~2 if
$$\alpha \frac{d}{d\gamma}(\log(G(\gamma)/\gamma))|_{\gamma=\alpha}<0\mbox{ or }
\alpha \frac{d}{d\gamma}(\log(G(\gamma)/\gamma))|_{\gamma=\alpha}>0.$$
We won't go further but we do see the ray-spiral phenomenon is fairly common.

\begin{appendix}

\section{Using the Foster-Meyn-Tweedie criterion}\label{Condition}
Let
\(
h_\Delta \coloneqq h \indicator{\Delta}.
\)
In the proof of Prop.~\ref{uniformbound},
we need to know that
when the ray condition (R1) holds on sheet 1
\begin{equation}\label{pi-h-on-Delta}
	\pi h_\Delta
	=
	\sum_{z \in \Delta}
	\pi(z)
	h(z)
	=
\sum_{x = 0}^\infty
\rho^{-x}
\pi(x,0,1)
<
\infty.
\end{equation}
It is easy to see that
\(
\sum_{x = 0}^\infty
\alpha^x
\pi(x,0,1)
<
\infty
\)
for
\(
\alpha < \rho^{-1}
\)
since the stationary distribution for $x$ jobs in the system is
\(
(1 - \rho)\rho^x.
\)
It is reasonable to suspect that the tail of
\(
\pi(x,0,1)
\)
is even lighter when
(R1) holds.
We give two results in this section.
The first result Prop.~\ref{weneedit} shows that
\begin{equation}\label{pi-alpha-on-Delta}
\sum_{x = 0}^\infty
\alpha^x
\pi(x,0,1)
<
\infty
\end{equation}
for a value of $\alpha > \rho^{-1}$
provided that both (R1) and (R2) hold.
The second result
Prop.~\ref{weneedit2}
uses an approach inspired by Chang and Down \cite{Chang-Down-2002,Chang-Down-2007}
to show that
\eqref{pi-h-on-Delta} holds when only (R1) is assumed to hold.

 \begin{proposition}\label{weneedit}
	 If conditions (K), (R1), and (R2) hold,
	 then \eqref{pi-alpha-on-Delta} holds with
	 \[
	 \alpha = \alpha_E =
	 \frac{1 - 2\sqrt{\lambda_2 \mu}}{\lambda_1}
	 >
	 \rho^{-1},
	 \]
	 which means that \eqref{pi-h-on-Delta} holds.
 \end{proposition}

\begin{proof}
	Define the Lyapunov function
	\(
	V(x,y,s) \coloneqq \alpha_E^x \beta_E^y.
	\)
	where
	the point $(\alpha_E, \beta_E)$
	is the easternmost point on Curve~2
	as defined in \eqref{rightmost}.
	We will show that
	\begin{equation}
		\label{V3}
	KV - V \leq -f + g
	\end{equation}
	where
	$f$ and $g$ are nonnegative functions with
	\(
	r \alpha_E^x
	\leq
	f(x,0,1)
	\)
	for all $x$
	where
	\(
	r > 0
	\)
	and
	\(
	\pi g
	=
	\sum_{z \in S}
	\pi(z) g(z)
	<
	\infty.
	\)
	It would then follow from
	Theorem~14.3.7 of \cite{MeynTweedie} that
	\(
	\pi f \leq \pi g;
	\)
	hence,
	\[
	r
	\sum_{x = 0}^\infty \pi(x,0,1) \alpha_E^x
	\leq
	\pi f
	\leq
	\pi g
	<
	\infty.
	\]

	Let
	\(
	f(x,y,s) = r V(x,y,1) \delta_1(s)
	\)
	where
	\(
	r = \mu(\beta_E^{-1} - \alpha_E^{-1})
	\)
	and $\delta$ is the Kronecker delta function;
	i.e.,
	\[
	\delta_x(y)
	\coloneqq
	\delta_{x,y}
	\coloneqq
	\begin{cases}
	1,
	&\text{for $x = y$;}
	\\
	0,
	&\text{for $x \neq y$.}
	\end{cases}
	\]
	Note that
	\(
	r > 0.
	\)
	This
	follows from the following argument.
	Let
	$(\hat{\alpha}, \hat{\beta})$ be any other solution
	lying on Curve~2: $\lambda_1\alpha+\lambda_2\beta+\mu\beta^{-1}=1$.  Then
	\(
	\hat{\alpha} < \alpha_E.
	\)
	The second component of the gradient of
	\(
	\lambda_1 {\alpha} + \lambda_2 {\beta} + \mu_2 {\beta}^{-1}
	\)
	evaluated at
	$(\alpha_E, \beta_E)$
	is zero.
	If the second component
	evaluated at
	$(\hat{\alpha}, \hat{\beta})$
	is positive, then
	\(
	\beta_E < \hat{\beta}.
	\)
	Let
	$(\hat{\alpha}, \hat{\beta}) = (\rho^{-1},\rho^{-1})$,
	which lies on
	Curve~2.
	Under (R2), the second component of the gradient
	evaluated at
	$(\rho^{-1}, \rho^{-1})$
	is positive,
	so
	\[
	\beta_E < \rho^{-1} < \alpha_E,
	\]
	which means that $r > 0$.

	Since $V$ is harmonic on sheet 2 and
	\(
	f(x,y,2) = 0,
	\)
	we know that
	\eqref{V3}
	holds on sheet 2.
Let $g(z)$ be zero at all states except
	\[
	g(0,0,1) = KV(0,0,1) + f(0,0,1)
	=
	\lambda_1 \alpha_E + \lambda_2 \beta_E + \mu + r.
	\]
	Then \eqref{V3} also holds at state $(0,0,1)$.

	For the rest of sheet 1, notice that
	\begin{align*}
		KV(x,y,1)
		&=
		\lambda_1 V(x + 1, y, 1)
		+
		\lambda_2 V(x, y + 1, 1)
		+
		\mu V(x - 1, y, 1)
		\\
		&=
		(1 - \mu(\beta_E^{-1} - \alpha_E^{-1}))
		V(x,y,1)
		\\
		&=
		(1 - r)
		V(x,y,1)
	\end{align*}
	where we used the sheet 2 identitiy
	\(
	1 = \lambda_1 \alpha_E + \lambda_2 \beta_E + \mu \beta_E^{-1}.
	\)
	Thus, on sheet 1 away from the origin,
	\[
	KV(x,y,1) - V(x,y,1) \leq -r V(x,y,1),
	\]
	and we have shown that \eqref{V3} holds, which completes the proof.
\end{proof}

\begin{proposition}\label{weneedit2}
If conditions (K) and (R1) hold,
then
 \eqref{pi-h-on-Delta}
 holds.
 \end{proposition}

 We delay the proof of Prop.~\ref{weneedit2} until
 Subsection~\ref{Downs-approach}.

\subsection{Nonnegative matrices and the Chang-Down approach}\label{N-T}

Chang and Down have \cite{Chang-Down-2002,Chang-Down-2007} have developed
a novel approach to showing that
\(
\pi h_\Delta < \infty.
\)
We revisit the Chang-Down approach,
but first
we need to develop several results for
nonnegative matrices.
Let $\mathcal{J}$
(i.e., ``script J'') be a matrix.
Assume that the elements are indexed by pairs of
elements of $S$ where $S$ is countable.
We assume that $\mathcal{J}$ satisfies the following condition:
\begin{hyp}{$\mathcal{J}$}
	\textit{For all $(i,j) \in S^2$,
	\(
	\mathcal{J}(i,j) \geq 0,
	\)
	\(
	\mathcal{J}^n(i,j)
	\)
	is finite,
	and
	$\mathcal{J}$ is irreducible.  That is,
	for any $i$ and $j$, there exists
	\(
	n = n(i,j) \geq 0
	\)
	such
	\(
	\mathcal{J}^n(i,j) > 0.
	\)
}
\end{hyp}

Let $E$ and $\Delta$ be nonempty subsets of $S$ where $E$ contains a
finite number
of elements but $\Delta$ may contain an infinite number of elements.
Define
\begin{align}
	\mathcal{J}_E(i,j)
	&\coloneqq
	\begin{cases}
	\mathcal{J}(i,j),
	&\text{for $i \not\in E $;}
	\\
	0,
	&\text{for $i \in E$.}
	\end{cases}
	\label{dfn:JE}
	\\
	\mathcal{J}_E(i,\Delta)
	&\coloneqq
	\sum_{j \in \Delta}
	\mathcal{J}_E(i,j)
	\notag
\end{align}
Let
\(
\mathcal{J}_E^n
\)
be the $n$th power of $\mathcal{J}_E$ where
\(
\mathcal{J}_E^0(i,j) = \delta_{i,j}.
\)
Note that
\(
\mathcal{J}_E^n(k, \Delta) = 0
\)
if $n \geq 1$ and $k \in E$
and that
\(
\mathcal{J}_E^0(k, \Delta) = \indicator{\Delta}(k).
\)
Define the column vectors $G_E$ and $H_E$ as
\begin{align}
	G_E(i)
	&\coloneqq
	\sum_{n \geq 0}
	\mathcal{J}_E^n(i,\Delta)
	\label{dfn-G}
	\\
	H_E
	&\coloneqq
	\mathcal{J} G_E.
	\label{dfn-H}
\end{align}

If
\(
\mathcal{J}_E
\)
is substochastic,
then
$G_E(i)$
can be interpreted as the expected number of visits to
$\Delta$ until hitting $E$ or being killed.
The following expression for  $G_E$ will prove useful.
\begin{align}
	G_E(i)
	&=
	\sum_{n \geq 0}
	\mathcal{J}_E^n(i,\Delta)
	\notag
	\\
	&=
	\indicator{\Delta}(i)
	+
	\sum_k
	\mathcal{J}_E(i,k)
	\sum_{n \geq 1}
	\mathcal{J}_E^{n - 1}(k,\Delta)
	\notag
	\\
	&=
	\indicator{\Delta}(i)
	+
	\sum_k
	\mathcal{J}_E(i,k)
	G_E(k)
	\notag
	\\
	&=
	\begin{cases}
	\indicator{\Delta}(i)
	+
	\mathcal{J}G_E(i)
	=
	\indicator{\Delta}(i)
	+
	H_E(i),
	&\text{for $i \not\in E$;}
	\\
	\indicator{\Delta}(i),
	&\text{for $i \in E$.}
	\end{cases}
	\label{eqnforGE}
\end{align}

We will need to know when $H_E$ is finite.  Here are two sufficient conditions.
\begin{lemma}
	\label{lemma:script-J}
	Assume condition $(\mathcal{J})$.
	If there exists an $i \in S$ such that
	\begin{equation}
		\label{eqn:finite-sum}
	\sum_{n = 1}^\infty
	\mathcal{J}^n(i,\Delta)
	<
	\infty,
	\end{equation}
	then
	\(
	H_E(i) < \infty
	\)
	for all $i \in S$.
\end{lemma}
\begin{proof}
	If
	\(
	\sum_{n = 1}^\infty
	\mathcal{J}^n(i,\Delta)
	=
	\infty
	\)
	for some $i$, then
	it follows from irreducibility that
	\(
	\sum_{n = 1}^\infty
	\mathcal{J}^n(i,\Delta)
	=
	\infty
	\)
	for all $i$.
	Hence, from the hypothesis, we know that
	\eqref{eqn:finite-sum}
	holds for all $i \in S$.
	Since
	\(
	\mathcal{J}_E
	\leq
	\mathcal{J},
	\)
	\begin{align*}
		H_E(i)
		&=
		\mathcal{J}
		G_E(i)
		\\
		&=
		\sum_k
		\mathcal{J}(i,k)
		\sum_{n \geq 0}
		\mathcal{J}_E^n(k,\Delta)
		\\
		&\leq
		\sum_k
		\mathcal{J}(i,k)
		\sum_{n \geq 0}
		\mathcal{J}^n(k,\Delta)
		\\
		&=
		\sum_{n \geq 1}
		\mathcal{J}^n(i,\Delta)
		\\
		&<
		\infty.
	\end{align*}
	
\end{proof}

\begin{lemma}
	\label{lemma:extend-HE}
	Assume condition $(\mathcal{J})$.
	If
	\(
	H_E(i) < \infty
	\)
	for all $i \in E$, then
	\(
	H_E(i) < \infty
	\)
	for all $i \in S$.
\end{lemma}
\begin{proof}
	On the contrary, assume there exists $j \not\in E$
	with $H_E(j) = \infty$.
	Since
	\(
	G_E(j) = \indicator{\Delta}(j) + H_E(j),
	\)
	by \eqref{eqnforGE},
	we know that
	$G_E(j)$ is also infinite.
	By irreducibility,
	there exists a path
	\(
	(i_0, i_1, \dotsc, i_n)
	\)
	with
	$i_0 = i$,
	$i_n = j$,
	and
	\(
	\mathcal{J}(i_m, i_{m + 1}) > 0
	\)
	for $m = 0, \dotsc, n - 1$.
	Let $i_k$ be the last time that the path is in $E$; that is,
	\(
	k = \max\set{m < n | i_m \in E}.
	\)
	Hence,
	\[
	H_E(i_k) \geq \mathcal{J} \mathcal{J}_E^{n - k - 1}(i_k, i_n) G_E(i_n) = \infty,
	\]
	which is a contradiction.
\end{proof}
\begin{remark}
	Suppose that states not in $E$ communicate with respect to
	$\mathcal{J}_E$. Then if $G_E(i)$ is infinite for one state
	$i \not\in E$, then $G_E(i)$ and $H_E(i)$ are infinite for
	all states not in $E$.
\end{remark}

\begin{lemma}\label{notdumb}
If
\(
H_E(i)<\infty
\)
for all $i \in E$, then 
\begin{equation}\label{goodasitgets}
	\mathcal{J} G_E - G_E
	\leq
	-\indicator{\Delta} + b \indicator{E}
\end{equation}
for the finite, positive constant $b = \max_{i \in E} H_E(i)$.
\end{lemma}

\begin{proof}
	Since
	\(
	G_E \leq \indicator{\Delta} + H_E
	\)
	by \eqref{eqnforGE},
	it follows that $G_E(i)$ is also finite for all $i\in E$.
	Since
	\begin{align*}
		H_E(i)
		&=
		\mathcal{J}G_E(i)
		=
		\mathcal{J}(i,\Delta)
		+
		\sum_k
		\mathcal{J}(i,k)
		\sum_{n \geq 1}
		\mathcal{J}_E^n(k, \Delta),
	\end{align*}
	it also follows that $\mathcal{J}(i,\Delta)$ is finite for $i \in S$.

	To see that \eqref{goodasitgets} holds for $i \not\in E$, the l.h.s.\
	simplifies to $-\indicator{\Delta}(i)$ after using \eqref{eqnforGE}, so we have equality
	in this case.
	When
	$i \in E$, the l.h.s.\ simplfies to
	\(
	H_E(i) - \indicator{\Delta}(i).
	\)
	By choosing $b = \max_{i \in E} H_E(i)$, the inequality holds.
	The constant $b$ is finite since $E$ is a finite set.
\end{proof}

The Chang-Down approach
to establishing that $\pi h_\Delta < \infty$ is
to establish that
\eqref{goodasitgets}
holds.
The following proposition describes this connection.
\begin{proposition}
	\label{pi-h-prop}
	Let
	\[
	\mathcal{J}(i,j)
	=
	\frac{K(i,j)h(j)}{h(i)}
	\]
	where $0 < h < \infty$.
	Assume that condition
	$(\mathcal{J})$
	holds.
	If
	$H_E(i) < \infty$
	for all $i \in E$,
	then
	\eqref{pi-h-on-Delta}
	holds.
\end{proposition}
\begin{proof}
	Substituting for
	$\mathcal{J}$ in
	\eqref{goodasitgets} gives
	\begin{align*}
		\sum_{j \in S}
		\frac{K(i,j)h(j)}{h(i)} G_E(j) - G_E(i)
		&\leq
		-\indicator{\Delta}(i) + b \indicator{E}(i)
		\\
		\sum_{j \in S}
		K(i,j)h(j) G_E(j) - G_E(i){h(i)}
		&\leq
		-{h(i)}\indicator{\Delta}(i) + b {h(i)}\indicator{E}(i)
		\\
		KV(i) - V(i)
		&\leq
		-h_{\Delta}(i) + b h(i) \indicator{E}(i)
	\end{align*}
	where
	\(
	V(i) = h(i) G_E(i),
	\)
	which is finite for all $i\in S$ by \eqref{eqnforGE} since
	$H_E(i)$ is finite for all $i \in S$
	by
	Lemma~\ref{lemma:extend-HE}.
	Hence, from Theorem~14.3.7 of \cite{MeynTweedie},
	\eqref{pi-h-on-Delta}
	holds.
\end{proof}

There are similarities and differences between the above and pp.~140--142 of
\cite{Chang-Down-2007}.
For example, our $\mathcal{J}$ is not the same as their $K_{\mathcal{W}}$ since the latter
is stochastic,
and hence, our (45) is not exactly the same as
their equation in
Lemma~4.1.5(i).
The biggest difference might be that in
Lemma~4.1.5(ii),
they
assume (20), which would be equivalent to us assuming that
$G_E(i)$ is finite for all $i$.
However, we make the stronger assumption that $H_E(i)$ is finite.
The stronger assumption is necessary to guarantee that our $b$ and their
$B$ or $B'$ is finite.  In particular,
in the proof of
Lemma~4.1.5(ii),
even if their $V_0(y)$ is finite
for all $y$
there is no guarantee that
$B'$ (after taking the maximum over $z$)
is finite.
However, this does not cause a problem
in the application to the polling model,
since the support of $K_W(z, \cdot)$ is finite.

Another difference is our inclusion of Lemma~\ref{lemma:script-J}.
The Chang-Down approach seems like the best way of showing that
\(
\pi h_\Delta < \infty
\)
when only a finite number of rows of $\mathcal{J}$ have row sums greater than
1.
It certainly
works well
for the polling model.
We have included Lemma~\ref{lemma:script-J} in the hope that
it may prove useful in extending the
Chang-Down approach to problems where an infinite number of rows
of $\mathcal{J}$ have row sums greater than 1.

\subsection{Proof of Prop.~\ref{weneedit2}}
\label{Downs-approach}
We use Prop.~\ref{pi-h-prop}.
Let
\(
E = (0,0,1),
\)
\(
\Delta =
\)
$K$ given by \eqref{K-defn},
and $h(x,y,s)=\rho^{-(x+y)}$.
Recall that $\Delta$ is the $x$-axis.
Since $h$ is harmonic except at $E$,
\(
\mathcal{J}
\)
will be stochastic except for the row indexed by state $(0,0,1)$.
(At state $(0,0,1)$,
\(
\mathcal{J}
\)
is superstochastic.)

To make use of the above lemmas we recall Theorem 14.3.7 in \cite{MeynTweedie} which
applies to irreducible Markov chains with steady state $\pi$: if $V$, $f$ and $g$ are two non-negative, finite valued functions on $S$ and
\begin{eqnarray}\label{compare}
KV(i)-V(i)&\leq& -f(i)+s(i)
\end{eqnarray}
 for all $i\in S$ then $\sum_{i\in S}f(i)\pi(i)\leq \sum_{i\in S}s(i)\pi(i)$.
Consequently, to prove $\sum_{i\in S}h(i)\pi(i)<\infty$ it suffices to find a Lyapunov function $V$ such that
$$KV(i)-V(i)\leq -h(i)\chi_{\Delta}+B_0\delta_E \mbox{ for some finite set $E$}$$ and some constant $B$.
As in \cite{Chang-Down-2007}, divide by $h(i)$ and rewrite the above as
$$\sum_{j\in S}K(i,j)\frac{h(j)}{h(i)}\frac{V(j)}{h(j)}- \frac{V(i)}{h(i)}\leq -\chi_{\Delta}+B\delta_E$$
where $B$ is a constant. Now define the kernel $\mathcal{J}(i,j)=K(i,j)h(j)/h(i)$ for $i,j\in S$. By Lemma \ref{notdumb}, $V=G_E$ and $B$ will exist if
$H(i)<\infty$ for all $i$.

 In our ray case $h(x,y,s)=\rho^{-(x+y)}$ and $\Delta =\{(x,0,1):x\geq 0\}$. We pick $E=\{(0,0,1)\}$. Note that  $\mathcal{J}$ is exactly our twisted kernel $\mathcal{K}$
 (as defined at (\ref{free}) which is super-stochastic at $(0,0,1)$).
 Let $i_0=(0,0,1)$ then $H_E(i_0)=\mu+\tilde{\lambda}_1H_E(1,0,1)+\tilde{\lambda}_2 H_E(0,1,2)$. Away from $i_0$ $\mathcal{J}$ is stochastic
with jumps east-west-north on sheet~1 and east-south-north on sheet~2.
Notice that returns to $\Delta$ involve a geometric number of jumps east or west on $\Delta$ followed by a loop through sheet~2 followed by
a geometric number of jumps on $\Delta$ followed by another loop and so on.  The geometric number $N$ of jumps east and west has distribution $P(N=n)=(\tilde{\lambda}_1+\tilde{\mu})^{n}\cdot\tilde{\lambda}_2$
and thus has finite mean. Starting at any $(x,0,1)$ the chance $p$ of simply drifting away on a ray on sheet~1 and never doing another loop is greater than
$\tilde{\lambda}_1(1-\frac{\tilde{\mu}}{\tilde{\lambda}_1})$ since this is the probability the east-west component which has drift $\tilde{\lambda}_1-\tilde{\mu}>0$
drifts to $+\infty$ without returning to $x$. Consequently the number of loops $L$ has finite expectation and if $N_i$ is the number of visits on loop $i$
then the total number of visits to $\Delta$ is $\sum_{i=1}^L N_i$. This has a finite expected value by Wald's Lemma.
Therefore $H_E(0,0,1)<\infty$.
Therefore we have checked the condition for  Lemma \ref{lemma:script-J} and then using Lemma \ref{lemma:extend-HE} we can apply Proposition \ref{pi-h-prop}
to get our result.

 In \cite{Chang-Down-2007} customers arrive at rate $\lambda_i$ to queues $Q_{(i)}, i=1,\ldots ,N$ and are given limited service from a single server. A state can be represented as
 $s=(q_1,q_2,\ldots ,q_N,z,i)$ where the $q_i$ are the respective queue sizes and $z$ is the queue currently being served and $i$ is the number of service completions during the current visit to queue $z$
 ($0\leq i\leq k_z$ where $k_z$ is the service limit; i.e. the maximum number of customers served per visit to queue z).
If a server empties a queue it moves on to the next nonempty queue.
 On page 136 of \cite{Chang-Down-2007} there is a description of the free chain with kernel $\overline{K}$
 constructed from the original chain with kernel $K$ by always serving the customers in queues (1) through (N) up to the service limit {\em even if } the queue becomes zero or negative.
 The original chain would not enter these states because the server would have just moved to the next nonempty queue.
 The set $\Delta\subset S$ is the set of states where $\overline{K}\neq K$. The harmonic function used is
 $h(s)=\rho^{-(q_1+q_2+\ldots +q_{N})}$ where $\rho=\sum_{i=1}^N\lambda_i/\mu<1$. Note that $h$ is harmonic for $\overline{K}$. The associated {\em twisted} or $h$-transformed chain makes  the twisted sum
 of queues $\sum_{i=1}^N \mathcal{Q}_{(i)}$ drift to infinity.

  We note that $h$ is harmonic for $K$ at all points in $S$ including $\Delta$ except for $(0,0,\ldots,0,N,0)$. This is because $h(s)$ only depends on the sum of the queue sizes.
  Consequently a transition from a state $s$ representing a service at queue $i\in\{1,2,\ldots,N\}$ where $q_i=1$ results in a decrease of $(q_1+q_2+\ldots +q_{N})$ by one regardless if the server moves next to another queue as specified by $K$
  or to a state where the server continues serving queue $i$ even when nonempty as specified by $\overline{K}$. $h$ is harmonic at $s$ either way.
  At $e=(0,0,\ldots,0,N,0)$, $h$ is not harmonic for $K$ so $\mathcal{J}$ is not stochastic and is in fact super-harmonic:
  $$\sum_f \overline{\mathcal{K}}(e,f)=(\sum_{i=1}^N\lambda_i)\rho^{-1}+\mu\rho=1$$ while
 $$\sum_f \mathcal{J}(e,f)=(\sum_{i=1}^N \lambda_i)\rho^{-1}+\mu>1.$$
 Consequently take $E=\{e\}$. The proof that $H_E(e)<\infty$ follows because away from $\Delta$, $\mathcal{J}=\overline{\mathcal{K}}$ and
  $\sum_{i=1}^N \mathcal{Q}_{(i)}$ has a positive drift as in \cite{Chang-Down-2007}.

 Lemma 4.1.4 in \cite{Chang-Down-2007} follows from (\ref{compare}) with $f(i)=h(i)\chi_{\Delta}(i)$ and $g(i)=B_0\chi_E(i)$
 where $E$ is a finite set. The  proof of (18) given  in \cite{Chang-Down-2007} is doubtful because there is no guarantee $\sum_i\pi(i)V(i)<\infty$.
 Moreover the hypotheses of Lemma 4.1.5 don't say anything about $B'$ being finite in the proof of (ii) implies (i). In spite of these problems that are fixed with
 Lemma \ref{lemma:extend-HE} and Lemma \ref{notdumb}, the general idea is excellent and we wish we had thought of it ourselves.
   We anticipate this method will find other applications.

\section{Twisting toward a given direction}
\cite{doob} and \cite{hennequin} showed the extremal harmonic functions on $\ZZ^d$ are of the form
$h(x)=\exp(\alpha\cdot x)$. We want to use these harmonic functions
 to twist our random walk into the direction of a vector $\beta$.
 The question is which directions $\beta$ can be achieved?
  In our polling example  $S_{\rho}=\{(1,0),(-1,0),(0,1)\}$ and $\rho$ defines probabilities $\lambda_1$, $\mu$ and $\lambda_2$ to these points respectively.
 We want to twist in all directions $\beta=(\beta_1,\beta_2)$ with $\beta_2\geq 0$:

 Proposition 5.3 in \cite{hennequin}
 provides an answer but the proof is incomplete.  Theorems VII.4.3 and VII.4.4 in \cite{Ellis} provide a partial answer to our question. There it is shown that for any $\beta$ in the interior of
  the closed convex hull of the support $S_{\rho}$ there exists a $\theta$ such that $\beta=\nabla\Lambda(\theta)$; i.e.
 $$\beta=\sum_x x\frac{e^{\theta\cdot x}}{\mgf(\theta)}\rho(x).$$
 However this twist generates a factor $\mgf(\theta)$ which messes up the calculation of the probability of large deviation. We want to find
 a $\theta$ such that $\mgf(\theta)=1$ and such that the twist drifts in direction $\beta$.

Consider a distribution  on $\ZZ^d$ having distribution $\rho$ which may be substochastic.
The distribution $\rho$  has support $S_{\rho}$.
 Recall that
the moment generating function of $\rho$ is
$\mgf(\theta)=\sum_x e^{\theta\cdot x}\rho(x)$.
As in Condition 6.2.1~(a) in \cite{Ellis-Dupuis} we assume $\mgf(\theta)<\infty$ for all $\theta$.
By (a) in Lemma 6.2.3 in \cite{Ellis-Dupuis} $\mgf(\theta)$ is a convex function differentiable everywhere.
Define $\Lambda(\theta)=\log(\mgf(\theta))$.
$\Lambda(\theta)$ is convex and differentiable on its domain $dom(\Lambda)=R^d$.

We shall consider a random walk with increments  having distribution $\rho$ on the convex cone $\mathcal{C}$ with base $o$ generated by $S_{\rho}$
whre $o=(0,\ldots ,)\in R^d$.
Hence the random walk has kernel $\overline{K}(x,y)=\overline{K}(o,y-x)=\rho(y-x)$ on the state space $S=\ZZ^d\cap \mathcal{C}$. As in our polling example $S$ may be a strict subset of
$ \ZZ^d$.
We do suppose $o$ is standard in the sense of Dynkin \cite{Dynkin}; i.e. $0<\overline{G}(o,y)$ for all $y\in \mathcal{C}$ where
$\overline{G}(x,y)=\sum_{k=0}^{\infty}\overline{K}^k(x,y)$ is the Greens function. $\overline{G}(o,y)$ is finite when $\rho$ is substochastic or when the mean drift $m$ is nonzero.
Consider a new probability measure with support
	\(
	o \cup S_{\rho}
	\)
	where
	\(
	\rho
	\)
	is scaled down by a factor $q$
	and
	the remaining mass $1 - q$ is assigned to $o$.
	The new probability measure has
	moment generating function
	\(
	\mgf_q(\theta) = q \mgf(\theta) + 1 - q
	\)
	and
	cumulant generating function
	\(
	\Lambda_q(\theta) = \log(\mgf_q(\theta)).
	\)
	Notice that
	\(
	\nabla \mgf_q(\theta) = q \nabla \mgf(\theta).
	\)
	Also notice that
	\(
	\mgf_q(\theta) = 1
	\)
	iff
	\(
	\mgf(\theta) = 1
	\)
	iff
	\(
	\Lambda_q(\theta) = 0
	\)
	iff
	\(
	\Lambda(\theta) = 0.
	\)

Consequently if we can find $\theta^*$ and $\lambda^*$ such that
such that
	\(
	\mgf_q(\theta^*) = 1
	\)
	and
	\(
	\nabla \mgf_q(\theta^*) = \lambda^* q\beta.
	\)
	Then,
	\(
	\Lambda_q(\theta^*)
	=
	\Lambda(\theta^*)
	=
	0.
	\)
	Also
	\(
	\nabla \mgf(\theta^*)
	=
	\lambda^* \beta;
	\)
i.e. we have found our twist for the original kernel.

Consequently, without loss of generality we assume $o \in S_{\rho}$. We summarize our assumptions in this Appendix
\begin{itemize}
\item	\textit{$\mgf$ is finite everywhere.}
\item \textit{Either $\rho$ is substochastic (i.e. $\rho(\ZZ^d)<1$) or $\rho$ is stochastic (i.e. $\rho(\ZZ^d)=1$) and the mean $m=\nabla\Lambda(0)\neq 0$.}
\item \textit{$o \in S_{\rho}$.}
\end{itemize}

Define $U^d$ to be the unit sphere.  Our goal is to find a twist in any direction $\beta\in U^d\cap \mathcal{C}$.
The Fenchel-Legendre transformation of  $\Lambda(\theta)$ is
$\Lambda^*(\beta)=\sup_{\theta}(\theta\cdot \beta-\Lambda(\theta))$.
$\Lambda^*(\beta)$ is also convex and differentiable on its domain
$dom(\Lambda^*)=\{\beta:\Lambda^*(\beta)<\infty\}$. Let $conv (S_{\rho})$ denote the convex hull of the set $S_{\rho}$.
The relative interior of the convex hull generated by $S_{\rho}$  is denoted by $ri(conv (S_{\rho}))$
in \cite{Ellis-Dupuis}.  By Lemma 6.2.3~(d) in \cite{Ellis-Dupuis}  $ri(conv (S_{\rho}))$ is equal to the relative interior of $dom(\Lambda^*)$ which is denoted by
$ri(dom(\Lambda^*))$; i.e. $ri(conv (S_{\rho}))=ri(dom(\Lambda^*))$.
For any $\alpha \in ri(conv (S_{\rho}))=ri(dom(\Lambda^*))$, Lemma 6.2.3 in \cite{Ellis-Dupuis}
shows, $\Lambda^*(\alpha)$ is attained at a point $\theta\in \mathcal{D}_{\Lambda}$ such that $\alpha=\nabla\Lambda(\theta)$.

\begin{theorem}\label{goanyway}
 Pick a direction $\beta\in U^d$ in the relative interior of $\mathcal{C}$.
 Then there exists a unique $\lambda^*>0$, a unique $\theta^*$ such that $\mgf(\theta^*)=1$ and an exponential change of measure such that
\begin{eqnarray}\label{ourresult}
\beta=\frac{1}{\lambda^*}\nabla\mgf(\theta^*)=\frac{1}{\lambda^*}\sum xe^{\theta^*\cdot x}\rho(x).
\end{eqnarray}
Moreover the map $\beta\to \overline{K}^{\theta^*}$ is one-to-one and continuous on the interior of $U^d\cap \mathcal{C}$.
This means $\overline{K}^{\theta^*}$, the $h^{\theta^*}$-transform of the kernel $\overline{K}$  with respect to $h^{\theta^*}(x)=\exp(\theta^*\cdot x)$ has a mean drift in the direction $\beta$.
\end{theorem}

\begin{proof}

 Consider the function $g(\lambda)=\Lambda^*(\lambda\beta)/\lambda$.
 There must exist some $\gamma$ such that $\gamma\beta$ is the  interior of $conv(S_{\rho})$, the convex cone generated by $S_{\rho}$ so
 $g(\gamma)<\infty$.
 By the convexity of $\Lambda^*$ and the fact that $o \in S_{\rho}$, the domain of $g$; i.e. $\{\lambda : g(\lambda) <\infty\}$ is an interval
  $(0< \overline{\lambda})$ containing $\gamma$.

Suppose $\overline{\lambda}=+\infty$. Then $g(\lambda)\geq \theta\cdot\beta-\Lambda(\theta)/\lambda$ for all $\theta$. Take $\theta=t\beta$ where $t>0$ so
$g(\lambda)\geq t|\beta|^2-\Lambda(t\beta)/\lambda$. Let $\lambda\to\infty$ so
$\liminf_{\lambda\to\infty}g(\lambda)\geq t|\beta|^2$ where $t$ is arbitrarily large. It follows that $\liminf_{\lambda\to\infty}g(\lambda)=\infty$.
Consequently if $\overline{\lambda}=+\infty$ then $\liminf_{\lambda\to\infty}g(\lambda)=\infty$.

On the other hand if $\overline{\lambda}<+\infty$ then either $\Lambda^*(\overline{\lambda}\beta)=\infty$ in which case
$g(\overline{\lambda})=\infty$ or $\Lambda^*(\overline{\lambda}\beta)<\infty$. In this case
\begin{eqnarray*}
\lim_{\lambda\uparrow \overline{\lambda}}g^{\prime}(\lambda)
&=&\lim_{\lambda\uparrow \overline{\lambda}}\left(\frac{\lambda\frac{d\Lambda^*(\lambda\beta)}{d\lambda}-\Lambda^*(\lambda\beta)}{\lambda^2}\right)\\
&\geq&\frac{1}{\overline{\lambda}^2}\left(\overline{\lambda}\lim_{\lambda\uparrow \overline{\lambda}}\frac{d\Lambda^*(\lambda\beta)}{d\lambda}-\Lambda^*(\overline{\lambda}\beta)\right)\\
&=&\infty
\end{eqnarray*}
because
$\frac{d\Lambda^*(\lambda\beta)}{d\lambda}\to +\infty$ as $\lambda\uparrow \overline{\lambda}$
since $\Lambda^*$ is steep at the boundary (see Theorem 1 in \cite{rocky} which shows $\Lambda^*$ is steep or essentially smooth because $\Lambda$ is).
Hence $g^{\prime}(\lambda)\to \infty$ as $\lambda\uparrow \overline{\lambda}$ when
$\overline{\lambda}<\infty$.

The other boundary of the domain of $g(\lambda)$ is at $\lambda=0$.
As above $g(\lambda)\geq \theta\cdot\beta-\Lambda(\theta)/\lambda$. Pick $\theta$ such that $\Lambda(\theta)<0$.
This is certainly possible in the substochastic case because $\Lambda(0)<0$. In the stochastic case  when $\nabla \Lambda(0)=m\neq 0$,  $\Lambda(\theta)<0$ if we pick $\theta$ close to zero in the direction $-m$.
Hence $-\Lambda(\theta)/\lambda\to +\infty$ as $\lambda \downarrow 0$. Hence, $g(\lambda)\to \infty$ as $\lambda\to 0$.

We have proven $g(\lambda)\to\infty$ as   $\lambda\to \overline{\lambda}=\infty$ or $\lambda\to 0$ so $g$ can't have a minimum at
endpoints like this. If $\overline{\lambda}<\infty$ then either $\lim_{\lambda\uparrow \overline{\lambda}}g(\lambda)\to\infty$ or
$\lim_{\lambda\uparrow \overline{\lambda}}g^{\prime}(\lambda)=\infty$.
Either way it follows that there exist a  $\lambda^*$
such that $g^{\prime}(\lambda^*)=0$ where $0<\lambda^*<\overline{\lambda}$.

Consequently $inf_{\lambda>0}\Lambda^*(\lambda\beta)/\lambda$ is finite and is achieved at  $\lambda^*$. At this value
\begin{eqnarray*}
\frac{d}{d\lambda}\left(\frac{\Lambda^*(\lambda\beta)}{\lambda}\right)|_{\lambda=\lambda^*}
=\frac{\lambda^*\nabla \Lambda^*(\lambda^*\beta)\cdot \beta-\Lambda^*(\lambda^*\beta)}{\lambda^2}=0.
\end{eqnarray*}
So
\begin{eqnarray}\label{xx}
\lambda^*\nabla \Lambda^*(\lambda^*\beta)\cdot \beta-\Lambda^*(\lambda^*\beta)=0.
\end{eqnarray}

Next,
\begin{eqnarray}
\Lambda^*(\lambda^*\beta)=\theta^*\cdot \lambda^*\beta-\Lambda(\theta^*) \mbox{ where }\label{reverse}\\
\lambda^*\beta=\nabla \Lambda(\theta^*)=\sum \frac{xe^{\theta^*\cdot x}}{\mgf(\theta^*)}\rho(x).\label{partoftheway}
\end{eqnarray}
Since $\Lambda$ is convex and differentiable it follows that $\Lambda^{**}=\Lambda$
by D2.6 in \cite{Ellis} or \cite{rockafellar}. Hence
$\Lambda(\theta^*)=\sup_{v}(v\cdot\theta^*-\Lambda^*(v))$ where the sup is attained at $v^*$ where $\theta^*=\nabla \Lambda^*(v^*)$.
On the other hand, by (\ref{reverse}), $\Lambda(\theta^*)=\theta^*\cdot \lambda^*\beta-\Lambda^*(\lambda^*\beta)$
so the supremum is attained at $v^*=\lambda^*\beta$ and hence $\theta^*=\nabla \Lambda^*(\lambda^*\beta)$.

Next,
\begin{eqnarray*}
\Lambda(\theta^*)&=&\lambda^*\beta\cdot\nabla\Lambda^*(\lambda^*\beta)-\Lambda^*(\lambda^*\beta)\\
&=&0\mbox{ by (\ref{xx}),}
\end{eqnarray*}
 i.e. $\mgf(\theta^*)=1$.
Now
(\ref{ourresult}) follows from (\ref{partoftheway}).

The uniqueness of $\theta^*$ follows from the convexity of  the level curve $\Lambda(\theta)=0$ since there is only one point $\theta^*$
such that $\nabla\Lambda(\theta^*)$ points in direction $\beta$. This fixes $\lambda^*$ since
$\lambda^*\beta=\nabla\Lambda(\theta^*)$.

The map of $\beta\to \lambda^*$ defined by (\ref{xx}) is differentiable on  the interior of $U^d\cap \mathcal{C}$.
Hence the map of
$$\beta\to \lambda^*\beta\to \theta^*\to h^{\theta^*}\to \overline{K}^{\theta^*}$$
 given by (\ref{ourresult}) is differentiable
on  the interior of $U^d\cap \mathcal{C}$.
\end{proof}
\begin{remark}
$\overline{G}(x,y)=\overline{G}^{\theta^*}(x,d)e^{-\theta^*\cdot (y-x)}$ and $\overline{G}(0,y)=\overline{G}^{\theta^*}(0,y)e^{-\theta^*\cdot y}$
so
$$k(x,y)=\frac{\overline{G}(x,y)}{\overline{G}(0,y)}=\frac{\overline{G}^{\theta^*}(x,y)}{\overline{G}^{\theta^*}(0,y)}e^{\theta^*\cdot x}
=k^{\theta^*}(x,y_n)e^{\theta^*\cdot x}.$$
Pick a sequence of $y_{\ell}$ such that $|y_n|\to\infty$ in direction $\beta$. Corollary 1.3 in \cite{Ney-Spitzer} shows
$k^{\theta^*}(x,y_{\ell})\to 1$ so $k(x,y_{\ell})\to e^{\theta^*\cdot x}$ (This is a statement about the smoothness of $\overline{G}^{\theta^*}(0,y)$ and irreducibility on $\ZZ^d$ is unnecessary). Hence the direction $\beta=\lambda^*\nabla \Lambda(\theta^*)$ corresponds to a point on the Martin boundary associated with the
extremal harmonic function $e^{\theta^*\cdot x}$.
\end{remark}

It may be the case that  the support of $S_{\rho}$ before adding $o$ is contained in a hyperplane in $R^n$.  A simple example
would be  $S_{\rho}=\{(1,0),(0,1)\}$ with arbitrary weights. Nevertheless Theorem \ref{goanyway} applies. Moreover in such a case
$\lambda^*\beta$ must lie in this hyperplane because
\begin{equation}
	\label{twisted-drift}
\lambda^*\beta
=
\nabla \mgf(\theta^*)
=
\sum xe^{\theta^*\cdot x}\rho(x);
\end{equation}
i.e. $\lambda^*\beta$
 is a convex combination
	of the vectors in the support of $\rho$, which means that
	 $\lambda^*\beta$ lies in the hyperplane.

Now let $\overline{\rho}$ be a countable probability measure on a discrete lattice in $R^{n+m}$. Suppose
the convex cone $\overline{\mathcal{C}}$ with base $o\in R^{n+m}$ generated by $S_{\overline{\rho}}$.
is  an affine space in $R^{n+m}$ of the form
\begin{eqnarray*}
\overline{S}&=&\{(x_1,\ldots ,x_n,y_1,\ldots y_m):\mbox{ where }(x_1,\ldots ,x_n)\in R^n;\\
& &\mbox{ and }y_k=L_k(x_1,\ldots ,x_n), k=1,\ldots ,m\}
\end{eqnarray*}
where $L_k$ are linear functions of $(x_1,\ldots ,x_n)$ for $k=1,\ldots ,m$. Letting $x=(x_1,\ldots ,x_n)^T$ and
$$y=(L_1(x_1,\ldots ,x_n),\ldots ,L_m(x_1,\ldots ,x_n))^T=Lx$$ we can express the coordinates of points in the support of $\overline{\rho}$ by
$(x,Lx)$. The moment generating function can be written
$$\mgf_{\overline{\rho}}(\theta,\alpha)=\sum_{x}\exp(\theta\cdot x+\alpha\cdot Lx)\overline{\rho}(x,Lx).$$

Suppose $\overline{\beta}$ is a given direction
in the relative interior of $\overline{\mathcal{C}}$. Consequently $\overline{\beta}=(\beta,L\beta)$
where $\beta$ is in the relative interior of the cone   $\mathcal{C}$ with base $o\in R^n$ generated by the projection of the point masses from $R^{n+m}$ to $R^n$.
Denote the projected point mass by $\rho$; i.e. $\rho(x)=\overline{\rho}(x,Lx)$.
By Theorem \ref{ourresult} there exists a unique $\lambda^*>0$ in $R^n$ and a unique $\theta^*$ in $R^n$ such that
$\mgf_{\rho}(\theta^*)=1$
and an exponential change of measure such that
\begin{eqnarray}
\beta=\frac{1}{\lambda^*}\nabla\mgf_{\rho}(\theta^*)=\frac{1}{\lambda^*}\sum xe^{\theta^*\cdot x}\rho(x).
\end{eqnarray}
Moreover
$$\nabla\mgf_{\overline{\rho}}(\theta,\alpha)=\sum_{(x,y)}(x,Lx)\exp(\theta\cdot x+\alpha\cdot Lx)\overline{\rho}(x,Lx)$$
and evaluating at $\theta=\theta^*$ and $\alpha=0$ we get
\begin{eqnarray*}
\nabla\mgf_{\overline{\rho}}(\theta^*,0)&=&\sum_{(x,Lx)}(x,Lx)\exp(\theta^*\cdot x)\rho(x)\\
&=&(\lambda^*\beta,L(\lambda^*\beta))=\lambda^*(\beta,L\beta).
\end{eqnarray*}
We therefore have an exponential change of measure to point the mean in direction $(\beta,L\beta)$.

We now address the issue of a direction  $\beta\in U^d$ on the boundary of $\mathcal{C}$.
Consider a convex cone $\mathcal{C}^E$ with base  $o$ supporting $\mathcal{C}$
of the form $(x,Lx):x\in R^n,Lx\in R^m$ such that $n+m=d$.
We suppose $S^E_{\rho}\setminus\{o\}$ is nonempty where
$S^E_{\rho}=\mathcal{C}^E\cap S_{\rho}$. Further we assume $\mathcal{C}^E$ is contained in the affine hull of $S^E_{\rho}$;
i.e. $conv(S^E_{\rho})$ is contained in the relative interior of $\mathcal{C}^E$ and
 $\mathcal{C}^E$ is subspace (or face or edge) of $\mathcal{C}$.
 Let $\rho^E$ denote the restriction
of $\rho$ to $S^E_{\rho}$; i.e. zero out $\rho(z)$ for $z\notin S_{\rho}^E$. This makes $\rho^E$ substochastic but we have assumed $\rho^E(S^E_{\rho}\setminus\{o\})>0$.
As above $z\in \mathcal{C}^E$ can be written $(x,Lx)$.
Define $\mgf_E(\eta)=\sum_{ z} e^{\eta\cdot z}\rho^E(z)$
and $\Lambda_E(\theta)=\log(\mgf^E(\theta))$. Here where $\eta=(\theta,\alpha)$ as above so
 $$\mgf_E(\eta)=\sum_{ x} e^{\theta\cdot x+\alpha\cdot Lx}\rho^E(z).$$

Let $\Lambda_E^*$ denote the associated Fenchel-Legendre transformation.
The following is a corollary of Theorem \ref{goanyway}:
\begin{corollary}\label{dotheedge}
If $\beta\in U^d$ in the relative interior of $\mathcal{C}^E$ where  $\rho^E\neq 0$
then
  there exists unique $\lambda_E^*>0$ and  $\theta_E^*$ such that $\mgf_E(\theta_E^*,0)=1$ and an exponential change of measure such that
\begin{eqnarray}\label{ourresult2}
\beta=\frac{1}{\lambda_E^*}\sum_{ x\in S^E_{\rho}}x \nu_E(x)\mbox{ where } \nu_E(x)=e^{\theta_E^*\cdot x}\rho^E(x).
\end{eqnarray}
\end{corollary}
Define  $h^E(x)=\exp(\theta^*\cdot x)\chi_{\mathcal{C}^E}(x)$ where $\mathcal{C}^E$ is the convex cone  with base $0$ generated by $S^E_{\rho}$.
Then the $h^E$-transform of the kernel $\overline{K}$ restricted to $\mathcal{C}^E$ with respect to $h^E$ has a mean drift in the direction $\beta$.

In the polling example when the server serves the first queue, if we take $\beta =(1,0)$  then $\mathcal{C}^E$ is the $x$-axis and $\rho^E(1,0)=\lambda_1$ and $\rho^E(0,-1)=\mu_1$.
$\exp(\theta^*)$ satisfies $\lambda_1 e^{\theta^*}+\mu e^{-\theta^*}=1$; i.e.
$exp(\theta^*)=\frac{1+\sqrt{1-4\lambda_1 \mu}}{2\lambda_1}$.
$h^E(i,j)=\exp(\theta^*i)$ if $j=0$ and $h^E$ is $0$ otherwise.
Hence $\nu^*(1,0)=\frac{1+\sqrt{1-4\lambda_1 \mu}}{2}$, $\nu^*(-1,0)=\frac{1-\sqrt{1-4\lambda_1 \mu}}{2}$ and
$\nu^*(0,1)=0$. Therefore
$$\sum x \nu^*(x)=\sqrt{1-4\lambda_1 \mu}\cdot(1,0)$$ so $\lambda^*=\frac{1}{\sqrt{1-4\lambda_1 \mu}}$
and $g(1,0)=\frac{1+\sqrt{1-4\lambda_1 \mu}}{2\lambda_1}$, $g(-1,0)=\frac{1-\sqrt{1-4\lambda_1 \mu}}{2\mu}$
and $g(0,1)=0$.

Moreover, for $\theta_n=(\theta_n(1),\theta_n(2))$ satisfying $\mgf(\theta_n)=1$ we have
$\nabla \mgf(\theta_n)=(\lambda_1e^{\theta_1}-\mu e^{-\theta_1},\lambda_2 e^{\theta_2})$,
 Take the limit as $\theta_n(2)\to -\infty$
along the egg so $\nabla \mgf(\theta_n)$ points north-east.
Hence $e^{\theta_n(1)}\to c=\frac{1+\sqrt{1-4\lambda_1 \mu}}{2\lambda_1}$ and
\begin{eqnarray*}
\nabla \mgf(\theta_n)&\to& (\frac{1+\sqrt{1-4\lambda_1 \mu}}{2}-\mu \frac{2\lambda_1}{1+\sqrt{1-4\lambda_1 \mu}},0)\\
&=&\sqrt{1-4\lambda_1\mu}=\sum x \nu^*(x).
\end{eqnarray*}

If we take $\beta =(-1,0)$  then $exp(\theta^*)=\frac{1-\sqrt{1-4\lambda_1 \mu}}{2\lambda_1}$.
$h^E(i,j)=\exp(\theta^*i)$ if $j=0$ and $h^E$ is $0$ otherwise.
Hence $\nu^*(1,0)=\frac{1-\sqrt{1-4\lambda_1 \mu}}{2}$, $\nu^*(-1,0)=\frac{1+\sqrt{1-4\lambda_1 \mu}}{2}$ and
$\nu^*(0,1)=0$. Therefore
$$\sum x \nu^*(x)=-\sqrt{1-4\lambda_1 \mu}\cdot(1,0)$$ so $\lambda^*=\frac{1}{\sqrt{1-4\lambda_1 \mu}}$
and $g(1,0)=\frac{1+\sqrt{1-4\lambda_1 \mu}}{2\lambda_1}$, $g(-1,0)=\frac{1-\sqrt{1-4\lambda_1 \mu}}{2\mu}$
and $g(0,1)=0$.
 Take the limit as $\theta_n(2)\to -\infty$
along the egg so $\nabla \mgf(\theta_n)$ points north-west.
Hence $e^{\theta_n(1)}\to c=\frac{1-\sqrt{1-4\lambda_1 \mu}}{2\lambda_1}$ and
\begin{eqnarray*}
\nabla \mgf(\theta_n)&\to& (\frac{1-\sqrt{1-4\lambda_1 \mu}}{2}-\mu \frac{2\lambda_1}{1-\sqrt{1-4\lambda_1 \mu}},0)\\
&=&-\sqrt{1-4\lambda_1\mu}=\sum x \nu^*(x).
\end{eqnarray*}
We conclude $m^{u(\theta^*)}$ is continuous on the whole egg and hence uniformly bounded.
The same is true of $|Q^{u(\theta^*)}|$ and $(m^{u(\theta^*)}\cdot \Sigma^{u(\theta^*)} m^{u(\theta^*)})$ by direct computation.
The goal of the next results leading to Theorem \ref{ontheedge} is to show this continuity is a general property.

\begin{lemma}\label{useful}
 If $\alpha\in conv(S^E_{\rho})$  then $\Lambda_E^*(\alpha)=\Lambda^*(\alpha)<\infty$.
\end{lemma}

\begin{proof}
 Since $S_{\rho}$ is countable by Theorem 9.4 in \cite{Barndorff} we have
 $conv (S_{\rho})\subseteq dom(\Lambda^*)$
(with equality when $S_{\rho}$ is finite) so $\Lambda^*(\lambda\beta)<\infty$.

Decompose any $\theta\in R^n$ into $\theta=\theta_E+\theta_T$ where $\theta_T$ is orthogonal to $E$ and $\theta_E$ is in $E$.
\begin{eqnarray}
\Lambda^*(\alpha)&=&\sup_{\theta}\left(\theta\cdot\alpha-\Lambda(\theta)\right)\label{agoodstep}\\
&=&\sup_{\theta_E,\theta_T}\left(\theta_E\cdot\alpha-\log(\sum_{x\in E}e^{\theta_E\cdot x}\rho(x)+\sum_{x\in E^c}e^{\theta\cdot x}\rho(x))\right)\nonumber.
\end{eqnarray}
Now, for a fixed $\theta_E$,
$$\sum_{x\in E^c}e^{\theta\cdot x}\rho(x)=\sum_{x\in E^c}e^{\theta_E\cdot x+\theta_T\cdot x}\rho(x)$$
can be arbitrarily close to $0$ by choosing $\theta_T$ pointing out of $\mathcal{C}$ so $\theta_T\cdot x<0$ for $x\in \mathcal{C}$
and as long as we like.  Hence
\begin{eqnarray*}
\lefteqn{\sup_{\theta_T}\left(\theta_E\cdot\alpha-\log(\sum_{x\in E}e^{\theta_E\cdot x}\rho(x)+\sum_{x\in E^c}e^{\theta\cdot x}\rho(x))\right)}\\
&=&\left(\theta_E\cdot\alpha-\log(\sum_{x\in E}e^{\theta_E\cdot x}\rho(x))\right)
\end{eqnarray*}
and the supremum of the above over $\theta_E$  is precisely $\Lambda_E^*(\alpha)$.
The result follows from (\ref{agoodstep}).

\end{proof}

Finally we address the continuity of the map $\beta\to  \overline{K}^{\theta^*}$ at
 the boundary of $U^d\cap \mathcal{C}$. Pick a sequence $\beta_n$ in the interior of $U^d\cap \mathcal{C}$ such that $\beta_n\to\beta\in \mathcal{C}^E$
 on the boundary of $U^d\cap \mathcal{C}$. We suppose $\rho(\mathcal{C}^E)>0$ throughout.

 \begin{lemma}\label{helpful}
 If $\lambda\beta\in conv(S_{\rho})\cap E$ then for $\beta_n\to\beta$,
 $$\lim_{n\to\infty}\Lambda^*(\lambda\beta_n)=\Lambda^*(\lambda\beta)=\Lambda_E^*(\lambda\beta).$$
 \end{lemma}

 \begin{proof}
$\Lambda^*$ is  convex and differentiable on the interior of $conv(S_{\rho})$ so $\Lambda^*(\lambda\beta_n)\to\Lambda^*(\lambda\beta)<\infty$
 using Lemma \ref{useful}.
 \end{proof}

$g_n(\lambda)=\Lambda^*(\lambda\beta_n)/\lambda$  attains a minimum at $\lambda^*_n$ by Theorem \ref{goanyway}.
 $g_n(\lambda)\to g(\lambda)=\Lambda_E^*(\lambda\beta)/\lambda$ by Lemma \ref{helpful}. Let $\lambda^*_E$ denote the $\lambda$ minimizing $g$.
 \begin{proposition}\label{uniformlybounded}
 $\lim_{n\to\infty}\lambda^*_n=\lambda^*_E$.
 \end{proposition}

 \begin{proof}
 For any $\delta$ define
 $\epsilon=\max\{g(\lambda^*_E-\delta)-g(\lambda^*_E),g(\lambda^*_E+\delta)-g(\lambda^*_E)$.
 Now find $N$ such that for $n\geq N$,
 $$g_n(\lambda^*_E)-g(\lambda^*_E)|<\epsilon/4,g_n(\lambda^*_E-\delta)-g(\lambda^*_E-\delta)|<\epsilon/4\mbox{ and }
 |g_n(\lambda^*_E+\delta)-g(\lambda^*_E+\delta)|<\epsilon/4;$$
 i.e. $g(\lambda^*_E-\delta)>g_n(\lambda^*_E)$ and $g(\lambda^*_E+\delta)>g_n(\lambda^*_E)$ for $n\geq N$.
 Then by convexity $\lambda_n^*\in [\lambda^*_E-\delta,\lambda^*_E+\delta]$ for $n\geq N$.
 But $\delta$ was arbitrarily small so $\lim_{n\to\infty}\lambda_n^*=\lambda_E^*$.

\end{proof}

The sequence  $\beta_n$ in the interior of $\mathcal{C}$ converges to $\beta\in \mathcal{C}^E$. For each $\beta_n$
there are associated $\theta^*_n$ and $\lambda^*_n$ and $\nu^*_n$ such that $ \sum_x  \nu_n^*(x)=1$ and
 $\lambda^*_n\beta_n= \sum_x x \nu_n^*(x)$
where $\nu^*_n(x)=e^{x\cdot\theta_n^*}\rho(x)$ and $\mgf(\theta_n^*)=1$.
Let $\theta_n^*$  denote
 the parameter to twist $\overline{K}$ into $\overline{K}^{\theta^*_n}$ with mean  $\lambda_n^*\beta_n$ while $\theta^*_E$  denotes
 the parameter to twist $\overline{K}$ into  $\overline{K}^{\theta_E^*}$ with mean  $\lambda^*_E\beta$ as defined in Corollary \ref{dotheedge}.
 Note that $\lambda_n^*\beta_n$ is in the interior of $dom(\Lambda^*)$ while $\lambda^*_E\beta$ is in the interior of $dom(\Lambda^*_E)$.

 Let $n_E$ denote the unit vector orthogonal to $\mathcal{C}^E$ at $\lambda^*\beta$ pointed into the interior of the cone $\mathcal{C}$. Then $n_E\cdot\beta=0$.
 $$n_E \cdot (\lambda_n^*\beta_n)=\sum_x  n_E \cdot x e^{x\cdot\theta_n^*}\rho(x)=\sum_{x\in E^c}   n_E \cdot x e^{x\cdot\theta_n^*}\rho(x).$$
 Now $n_E \cdot (\lambda_n^*\beta_n)\to 0$ since $\lambda_n^*$ is uniformly bounded by Proposition \ref{uniformlybounded} and $\beta_n\to\beta$ which is in $E$.
 Next all the terms $ n_E\cdot x$ are nonnegative by convexity so
 $\sum_{x\in E^c}   n_E \cdot x e^{x\cdot\theta_n^*}\rho(x)\to 0$ which means
 $e^{\theta^*_n\cdot x}\to 0$ for $x\in S_{\rho}\setminus S_{\rho}^E$.
Hence $g(x)=0$ if $x\in S_{\rho}\setminus S_{\rho}^E$.
Moreover, by Fatou's Lemma,
$$1\geq \sum_x \liminf_{n_k\to\infty}e^{x\cdot\theta_{n_k}^*}\rho(x)=\sum_x g(x)\rho(x).$$

Since $\sum_x e^{x\cdot\theta_n^*}\rho(x)=1$ it follows that $0\leq e^{x\cdot\theta_n^*}\leq 1/\rho(x)$ uniformly in $\theta_n$.
 Hence we can pick a subsequence indexed by $n_k$ that converges; i.e. $e^{x\cdot\theta_n^*}\to g(x)$ for all $x$; alternatively
 $\nu^*_{n_k}$ converges weakly to
a  measure $\nu^*$.   $\zeta_n(t)=\sum_x e^{t\cdot x}\nu^*_n(x)$ is the moment generating function of $\nu^*_n$
and  $\zeta(t)=\sum_x e^{t\cdot x}\nu^*(x)$ is the moment generating function of $\nu^*$.

 By the separating hyperplane theorem (see Lemma VI.5.4 in \cite{Ellis}) there exists  a unit vector $a$ and a hyperplane $\{x:x\cdot a=0\}$
 supporting the  cone $\mathcal{C}$
 such that $ a\cdot x\leq 0$ for all $x\in \mathcal{C}$.
Since $a\cdot x\leq 0$  it follows that $\zeta_{n_k}(a)<\infty$ and $\zeta_{n_k}(a)\to \zeta(a)$.

\begin{proposition}
$\nu^*$ is a probability.
\end{proposition}
\begin{proof}
Take $t=\epsilon a$ so $e^{t\cdot x}\leq 1$ for all $x\in \mathcal{C}$.
By the convexity of $ e^{t\cdot x}$
 as a function of $x$ using Jensen's inequality we have
 $$\zeta_n(t)\geq \exp(\sum_x t\cdot x\nu^*_n(x))=\exp(\lambda_n^*\beta_n\cdot t).$$
Hence $$\sum_x e^{t\cdot x}\nu^*(x)=\zeta(t)\geq e^{\lambda^*_E\beta\cdot t}.$$ Taking $\epsilon\to 0$ implies
$\nu^*$ is a probability.
\end{proof}

 \begin{proposition} \label{meansconverge}
We assume either the support of $\rho$ is finite or
there is a neighbourhood $\{b:|b-a|\leq \delta\}$ such that $b\cdot x<0$ for all $x\in \mathcal{C}$.
It follows that $\lambda^*\beta= \sum_x x \nu^*(x).$
 \end{proposition}

 \begin{proof}
 Recall $\sum_x x \nu_{n_k}^*(x)=\lambda_{n_k}^*\beta_{n_k}$ and
$$\lambda^*\beta=\lim_{n_k\to\infty}\lambda_{n_k}^*\beta_{n_k}=\lim_{n_k\to\infty}\sum_x x \nu_{n_k}^*(x).$$
But $\nu_{n_k}^*$ converges to $\nu^*$ so, if the support of $\rho$ is finite, $\lim_{n_k\to\infty}\sum_x x \nu_{n_k}^*(x)=\sum_x x \nu^*(x)$.
The  result follows.

Define the neighbourhood $N=\{v:v=\epsilon b, |b-a|\leq \delta\}$ and define $t=\epsilon a$.
We have $\zeta_{n_k}(v)\leq 1$  for each $n_k$ and $\zeta(v)\leq 1$ for $v\in N$. Moreover uniformly on $N$,
$\zeta_{n_k}(v)\to \zeta(v)$ for $v\in N$.
 Let $u_j$ be the unit basis vector in coordinate $j$ and
let $x_j$ be the $j^{th}$ component of $x$.
With $0<|s|\leq 1$,
  \begin{eqnarray*}
 |\frac{(e^{(sx_j)}-1)}{s}-x_j |&= & |\sum_{m=2}^{\infty}\frac{s^{m-1}x_j^m}{m!}|\\
 &\leq& |s|^{1/2}\sum_{m=2}^{\infty}\frac{|s|^{m-3/2}|x_j|^m}{m!}\\
 &\leq& |s|^{1/2}\exp(|s|^{1/4}|x_j|)\mbox{ since } m-3/2\geq m/4\mbox{ for } m\geq 2\\
 &\leq& |s|^{1/2}\left(\exp(|s|^{1/4} x_j)+\exp(-|s|^{1/4} x_j)\right).
  \end{eqnarray*}

Consequently,
 \begin{eqnarray*}
 & &|\frac{1}{s}\left(\zeta_{n_k}(t+s u_j)-\zeta_{n_k}(t)\right)-\frac{\partial\zeta_{n_k}(t)}{\partial t_j}|\\
 &=&|\sum_x \frac{(e^{(sx_j)}-1)}{s}e^{t\cdot x}\nu^*_n(x)-\sum_x x_j e^{t\cdot x}\nu^*_n(x)|\\
 &\leq &|s|^{1/2}\left(\sum_x\exp(|s|^{1/4} x_j)e^{t\cdot x}\nu^*_n(x)+\sum_x\exp(-|s|^{1/4} x_j)e^{t\cdot x}\nu^*_n(x)\right)\\
 &=&|s|^{1/2}\left(\zeta_n(t+|s|^{1/4} u_j)+\zeta_n(t-|s|^{1/4} u_j)\right)\\
 &\leq& 2|s|^{1/2}
  \end{eqnarray*}
  for $|s|^{1/4}<\epsilon$ and this inequality is uniform in $n_k$.
 Similarly,
 \begin{eqnarray*}
 |\frac{1}{s}\left(\zeta(t+s u_j)-\zeta(t)\right)-\frac{\partial\zeta(t)}{\partial t_j}|
 &\leq& 2|s|^{1/2}
  \end{eqnarray*}
   for $|s|^{1/4}<\epsilon$.

  There exists an $M$ such that for $n_k\geq M$
  $$|\left(\zeta_{n_k}(t+s u_j)-\zeta_{n_k}(t)\right)- \left(\zeta(t+s u_j)-\zeta(t)\right)|\leq \delta$$
 for  $n_k\geq M$. Consequently, for $n_k\geq M$,
 \begin{eqnarray*}
 \lefteqn{|\frac{\partial\zeta_{n_k}(t)}{\partial t_j}-\frac{\partial\zeta(t)}{\partial t_j}|
 \leq |\frac{1}{s}\left(\zeta_{n_k}(t+s u_j)-\zeta_{n_k}(t)\right)-\frac{\partial\zeta_{n_k}(t)}{\partial t_j}|}\\
 & &+|\frac{1}{s}\left(\zeta_{n_k}(t+s u_j)-\zeta_{n_k}(t)\right)- \frac{1}{s}\left(\zeta(t+s u_j)-\zeta(t)\right)|
 +|\frac{1}{s}\left(\zeta(t+s u_j)-\zeta(t)\right)-\frac{\partial\zeta(t)}{\partial t_j}|\\
 &\leq&4|s|^{1/2}+\frac{\delta}{s}.
 \end{eqnarray*}
 But $\delta$ is arbitrarily small as $n_k$ tends to infinity so
 $ \lim_{n_k\to\infty}|\frac{\partial\zeta_{n_k}(t)}{\partial t_j}-\frac{\partial\zeta(t)}{\partial t_j}|\leq 4|s|^{1/2}$.
 Moreover $s$ is arbitrarily small   so
 $$ \lim_{n_k\to\infty}\frac{\partial\zeta_{n_k}(t)}{\partial t_j}=\frac{\partial\zeta(t)}{\partial t_j}.$$

 Again, for any $\delta_1>0$, there exists an $M_1$ such that for $n_k\geq M_1$,
 $$|\frac{\partial\zeta_{n_k}(t)}{\partial t_j}-\frac{\partial\zeta(t)}{\partial t_j}|\leq\delta_1.$$
 Consequently, for $n_k\geq M$ and $t=\epsilon a$
 $$|\frac{1}{s}\left(\zeta_{n_k}(t+s u_j)-\zeta_{n_k}(t)\right)-\frac{\partial\zeta(t)}{\partial t_j}|\leq 2|s|^{1/2}+\delta_1.$$
Letting $s\to 0$ above gives
  \begin{eqnarray}\label{itgoes}
  |\frac{\partial\zeta_{n_k}(t)}{\partial t_j}-\frac{\partial\zeta(t)}{\partial t_j}|\leq \delta_1
  \end{eqnarray}
  for $n_k\geq M_1$.

 Hence,
 writing $x_j=x_j^{+}-x_j^{-}$ where $x_j^{+}=\max(0,x_j)$
 \begin{eqnarray*}
 \frac{\partial\zeta_{n_k}(t)}{\partial t_j}&=&\sum_x x_je^{t\cdot x}\nu_{n_k}^*(x)\\
 &=&\sum_x x_j^{+}e^{t\cdot x}\nu_{n_k}^*(x)-\sum_x x_j^{-}e^{t\cdot x}\nu_{n_k}^*(x).
 \end{eqnarray*}
As $\epsilon \downarrow 0$,
\begin{eqnarray*}
\lim_{\epsilon\to 0}\sum_x x_j^{+}e^{t\cdot x}\nu_{n_k}^*(x)
&=& \lim_{n_k\to\infty}\sum_x x_j^{+}\nu_{n_k}^*(x)
\end{eqnarray*}
by monotone convergence.
 The above limit exists because
 $$\lambda_{n_k}^*\beta_{n_k}\cdot u_j=\sum_x x_j^{+}\nu_{n_k}^*(x)-\sum_x x_j^{-}\nu_{n_k}^*(x).$$
Similarly
 \begin{eqnarray*}
\lim_{\epsilon\to 0}\sum_x x_j^{-}e^{t\cdot x}\nu_{n_k}^*(x)
&=& \lim_{_k\to\infty}\sum_x x_j^{-}\nu_{n_k}^*(x).
\end{eqnarray*}

Take the limit of (\ref{itgoes}) as $\epsilon\to 0$.  Together with the above limits  this gives
 \begin{eqnarray}\label{hopeitisright}
|\frac{\partial\zeta_{n_k}(0)}{\partial t_j}-\lim_{\epsilon\to 0}\frac{\partial\zeta(t)}{\partial t_j}|\leq \delta_1;
 \end{eqnarray}
 for $n_k\geq M$.
 Hence, using the fact that $\delta_1$ is arbitrarily small,
 \begin{eqnarray}\label{looksgood}
 \lim_{n_k\to\infty}\frac{\partial\zeta_{n_k}(0)}{\partial t_j}=\lim_{\epsilon\to 0}\frac{\partial\zeta(t)}{\partial t_j}.
 \end{eqnarray}
  But
 $$\lim_{n_k\to\infty}\frac{\partial\zeta_{n_k}(0)}{\partial t_j}=\lim_{n_k\to\infty}\lambda_{n_k}^*\beta_{n_k}\cdot u_j=\lambda^*\beta\cdot u_j.$$
 Hence the limit $\lim_{\epsilon\to 0}\frac{\partial\zeta(t)}{\partial t_j}$ exists and equals $\lambda^*\beta\cdot u_j$.
 Moreover
\begin{eqnarray*}
\lim_{\epsilon\to 0}\frac{\partial\zeta(t)}{\partial t_j}&=&\lim_{\epsilon\to 0}\sum_x x_je^{t\cdot x}\nu^*(x)\\
&=&\lim_{\epsilon\to 0}\sum_x x^{+}_je^{t\cdot x}\nu^*(x)-\lim_{\epsilon\to 0}\sum_x x_j^{-}e^{t\cdot x}\nu^*(x)\\
&=&\sum_x x_j^{+}\nu^*(x)-\sum_x x_j^{-}\nu^*(x)=\sum_x x_j\nu^*(x)\mbox{ by monotone convergence}\\
&=&\frac{\partial\zeta(0)}{\partial t_j}=\sum_x x\cdot u_j \nu^*(x).
\end{eqnarray*}
 Substituting into  (\ref{looksgood}) gives the result for the $j^{th}$ coordinate.
 The full result follows immediately.
 \end{proof}

$\rho$ is discrete having mass $\rho(x_i)$ on a  set of  vectors $\{x_1,x_2,\ldots \}\in \ZZ^d$.
 Consider the cone $\mathcal{C}_E$ with base $o$ generated by $\{x_i :\nu^*(x_i)>0\}$ inside $\mathcal{C}^E$.
The random walk with transition probabilities $K_{E}(0,v_i)=\nu^*(v_i)$ has state space $S_{E}=\ZZ^d\cap \mathcal{C}_E$.
i.e. for $y\in S_{E}$, $y=\sum_{i}m_iv_i\chi\{\nu^*(v_i)>0\}$ where the $m_i$ are integers.
We can define $g(y)=\prod_{i:\nu^*(v_i)>0}\nu^*(v_i)^{m_i}$ on $E$.
This function is well defined. If $y$ is also represented as $y=\sum_{i}p_iv_i\chi\{\nu^*(v_i)>0\}$ then
\begin{eqnarray*}
e^{\theta_{n_k}\cdot y}&=&e^{\theta_{n_k}\cdot \sum_{i}m_iv_i\chi\{\nu^*(v_i)>0\}}\\
&=&\prod_{i:\nu^*(v_i)>0}(e^{\theta_{n_k}\cdot v_i})^{m_i}\\
&\to&\prod_{i:\nu^*(v_i)>0}\nu^*(v_i)^{m_i}.
\end{eqnarray*}
Similarly, $e^{\theta_{n_k}\cdot y}\to \prod_{i:\nu^*(v_i)>0}\nu^*(v_i)^{p_i}$ so $g(y)$ has the same value whatever the representation.
Note that $g$ is harmonic for $K_{E}$ on $S_{E}$. For $y=\sum_{i}p_iv_i\chi\{\nu^*(v_i)>0\}$
\begin{eqnarray*}
K_{E}g(y)&=&\sum_i K_{E}(y,y+v_i)g(y+v_i)=\sum_i K_{E}(0,v_i)g(y)g(v_i)\\
&=&g(y)\left(\sum_i K_{E}(0,v_i)g(v_i)\right)=g(y).
\end{eqnarray*}
If $z=\sum_{i}u_iv_i\chi\{\nu^*(v_i)>0\}$ so $g(z)=\prod_{i:\nu^*(v_i)>0}\nu^*(v_i)^{u_i}$ then
$y+z=\sum_{i}(m_i+u_i)v_i\chi\{\nu^*(v_i)>0\}$ and
$$g(y+z)=\prod_{i:\nu^*(v_i)>0}\nu^*(v_i)^{m_i+u_i}=\prod_{i:\nu^*(v_i)>0}\nu^*(v_i)^{m_i}\prod_{i:\nu^*(v_i)>0}\nu^*(v_i)^{u_i}=g(y)\cdot g(z).$$
Consequently $g(y)=\exp(\theta^*\cdot y)$ if $y$ is of the form $\sum_{i}m_iv_i\chi\{\nu^*(v_i)>0\}$; i.e. when $y\in S_E$. $g(y)=0$ if $y\in E\setminus S_E$.

 \begin{theorem}\label{ontheedge}
 We assume either the support of $\rho$ is finite or
there is a neighbourhood $\{b:|b-a|\leq \delta\}$ such that $b\cdot x<0$ for all $x\in \mathcal{C}$. Then
  $$\nu_{n}^*\Longrightarrow \nu_E^*\mbox{ and }\overline{K}^{\theta^*_{n}}\to \overline{K}^{\theta_E^*}.$$
\end{theorem}

\begin{proof}

By Proposition \ref{meansconverge},
$$\lambda_n^*\beta_n=\sum_x xe^{x\cdot\theta_n^*}\rho(x)\to \sum_x xe^{x\cdot\theta^*}\rho(x).$$
But $\lambda_n^*\beta_n\to \lambda_E^*\beta$
so $\sum_x x\nu^*(x)=\lambda_E^*\beta$; i.e.
$$\sum_x xe^{\theta^*\cdot x}\rho^E(x)=\lambda_E^*\beta.$$
But $\theta^*_E$ is the unique twist such that
$$\sum_{x\in S_{\rho}^E}  x e^{\theta_E^*\cdot x}\rho^E(x)=\lambda^*_E\beta.$$
Hence $\theta^*=\theta^*_E$ and  $\nu^*=\nu_E^*$ and this holds for all possible subsequences $n_k$.
Hence $\nu^*_n\Longrightarrow \nu_E^*$  and equivalently $\overline{K}^{\theta^*_n}\to \overline{K}^{\theta_E^*}$.
\end{proof}

We do need the hypothesis that $\rho(S^E)>0$. Consider $\rho$ on $\ZZ^2$ such that
$$\rho(0,0)=1/2\mbox{ and }\rho(n,1)=\frac{1}{2}e^{-a}\frac{a^n}{n!}\mbox{ for } n\geq 0.$$
Let $\mathcal{C}^E$ be the $x$-axis.
$$\mgf(\theta_1,\theta_2)=1/2+(1/2)\exp(a e^{\theta_1}-a+\theta_2).$$
The egg $\mgf(\theta_1,\theta_2)=1$ is $a e^{\theta_1}-a+\theta_2=0$ so $\theta_2=a(1-e^{\theta_1})$ for $(\theta_1,\theta_2)$ on the egg
 $$\nabla\mgf(\theta_1,\theta_2)=\frac{1}{2}(a e^{\theta_1},a)\mbox{ for }(\theta_1,\theta_2).$$
Hence, as $\theta_1\to\infty$ and $\theta_2=a(1-e^{\theta_1})\to-\infty$ through a sequence indexed by $n$, $\beta_n=\nabla\mgf(\theta_1,\theta_2)$ points more and more in direction $\beta=(1,0)$
but with a length tending to infinity.  There is no limiting measure $e^{\theta^*_E\cdot x}\rho_E(x)$ because the measures
$e^{\theta_n^*\cdot x}\rho_E(x)$ are not tight. There is no contradiction with the above result because $\rho^E=0$.

\begin{remark}
$\overline{G}(x,y)=\overline{G}^{\theta^*_E}(x,d)\exp(-\theta^*_E\cdot (y-x))$ and $\overline{G}(0,y)=\overline{G}^{\theta^*_E}(0,y)\exp(-\theta^*_E\cdot y)$
so
$$k(x,y)=\frac{\overline{G}(x,y)}{\overline{G}(0,y)}=\frac{\overline{G}^{\theta^*_E}(x,y)}{\overline{G}^{\theta^*_E}(0,y)}e^{\theta^*_E\cdot x}
=k^{\theta^*_E}(x,y_n)e^{\theta^*_E\cdot x}.$$
Pick a sequence of $y_{\ell}$ such that $|y_n|\to\infty$ in direction $\beta$. Corollary 1.3 in \cite{Ney-Spitzer} shows
$k^{\theta^*_E}(x,y_{\ell})\to 1$ so $k(x,y_{\ell})\to e^{\theta^*_E\cdot x}$. This holds even when $K$ is not irreducible.
Hence the direction $\beta=\lambda^*\nabla \Lambda(\theta^*_E)$ corresponds to a point on the Martin boundary associated with the
extremal harmonic function $g(x)=e^{\theta^*_E\cdot x}$ for $x\in E$ and $g(x)=0$ for $x\notin E$.
\end{remark}

\section{Application of large deviation theory to our polling model}

Suppose we want to investigate a ray case on the first sheet. In fact we may wish to investigate the probability the queue size reaches $\ell$ by running away  in some direction $\beta$.
We calculate $r^*(t)$ in Theorem 6.15 in \cite{weiss} for this example.
 Extend the chain  in the interior of sheet~1 to the whole plane $\ZZ^2$ giving the free random walk
with increments having log Laplace transform
$$\Lambda(\theta,1)=\log(\mgf_1(\theta,1))=\log(\lambda_1 e^{\theta_1}+\lambda_2 e^{\lambda_2}+\mu e^{-\theta_1})\mbox{ where }\theta=(\theta_1,\theta_2).$$
Define the rate function
$$\Lambda^*(\beta,1)=\sup_{\theta}(\theta\cdot \beta- \Lambda(\theta,1)).$$
The action associated with a smooth path $\overrightarrow{r}(t) $ from $(0,0,1)$ to $(x,y,1)$ where $x+y=\ell$ and $x,y\geq 0$ is given by
$\int_0^T\Lambda^*(\frac{d\overrightarrow{r}(t)}{dt})dt$ where $\overrightarrow{r}(T)=(x,y,1)$.
By the calculus of variation and the fact that the rate function is homogeneous over the plane, the minimal action is given by a straight line $\overrightarrow{r}(t)=t\beta$
where $T(\beta_1+\beta_2)=\ell$. Consequently the minimum action is of the form
$$T\Lambda^*(\beta,1)=\frac{\ell}{(\beta_1+\beta_2)}\Lambda^*(\beta,1).$$

To find the optimal direction  note that $\Lambda(\theta,1)$ is smooth so
\begin{eqnarray}\label{bestdirection}
\beta_1=\frac{\partial \Lambda(\theta,1) }{\partial \theta_1}\mbox{ and }\beta_2=\frac{\partial \Lambda(\theta,1) }{\partial \theta_2}
\end{eqnarray}
provides a differentiable map between $\beta$ and the $\theta$ which gives the supremum of the rate function.
Now note that for this choice of $\beta$ and  $\theta$, the following dual relationships hold:
$$\theta_1=\frac{\partial \Lambda^*(\beta,1) }{\partial \beta_1}\mbox{ and }\theta_2=\frac{\partial \Lambda^*(\beta,1) }{\partial \beta_2}$$

At the optimal $\beta$,
$$\frac{\partial  }{\partial \beta_1} \left(\frac{\Lambda^*(\beta,1)}{\beta_1+\beta_2}\right) =\frac{1}{(\beta_1+\beta_2)^2}(\theta_1(\beta_1+\beta_2)- \Lambda^*(\beta,1))=0$$ and
$$ \frac{\partial }{\partial \beta_2} \left(\frac{\Lambda^*(\beta,1)}{\beta_1+\beta_2}\right) =\frac{1}{(\beta_1+\beta_2)^2}(\theta_2(\beta_1+\beta_2)- \Lambda^*(\beta,1))=0.$$
 Hence
 \begin{eqnarray}\label{opt}
 \theta_1(\beta_1+\beta_2)=\Lambda^*(\beta,1) \mbox{ and }\theta_2(\beta_1+\beta_2)=\Lambda^*_1(\beta,1).
 \end{eqnarray}
Hence   $\theta_1=\theta_2=\theta^{*}$ and
 $$\theta^*(\beta_1+\beta_2)=\Lambda^*(\beta,1)=(\theta^*(\beta_1+\beta_2)-\log(\Lambda((\theta^*,\theta^*),1);$$
i.e. $\log(\Lambda((\theta^*,\theta^*),1)=0$ so $\Lambda((\theta^*,\theta^*),1)=1$.
However  the only nonzero solutions to $\Lambda((\theta^*,\theta^*),1)=\lambda_1e^{\theta^*}+\lambda_2e^{\theta^*}+\mu e^{-\theta^*}=1$ are $1$ and $\\rho^{-1}$.
The velocity along the least action path $\overrightarrow{r}(t)=t\beta$ is given by \ref{bestdirection} but $\theta^*=1$ then yields $\beta_1=\lambda_1-\mu<0$.
This gives a line $\overrightarrow{r}(t)$ with slope $\lambda_2/(\lambda_1-\mu)<-1$ which doesn't hit $x+y=\ell$ for $t\geq 0$. This leave $\theta^*=\rho^{-1}$ in which case
$$\beta_1=\frac{\partial \Lambda(\theta,1) }{\partial \theta_1}|_{\theta=(\rho^{-1},\rho^{-1})}=\lambda_1\rho^{-1}-\mu\rho$$
and
$$\beta_2=\frac{\partial \Lambda(\theta,1) }{\partial \theta_2}|_{\theta=(\rho^{-1},\rho^{-1})}=\lambda_2\rho^{-1}.$$
Hence the least action path to the boundary $x+y=\ell$ is $\overrightarrow{r}^*(t)=(\rho^{-1}t,\rho^{-1}t)$ which is the fluid limit of our twisted path calculated before.

The above large deviation calculation was valid when sheet~1 is extended to the whole plane $\ZZ^2$. This means that even in a ray case when the least action path does lie entirely in sheet~1
we still have to verify that the action calculated above
is smaller than the action along paths reaching $x+y=\ell$ that spiral onto sheet~2 or spiral multiple times onto sheets one and two
or those that immediately start on sheet~2. We also have to consider the action along paths that form bridges
along the $y$-axis on sheet~1 or on the $x$-axis on sheet~2. The action for paths along boundaries are discussed in \cite{weiss} or \cite{FM}. While it's true the action is constant
in any of the above directions it still remains to show that the action along such a sequence of line segments leading to the boundary is greater than  the least action path that does lie entirely in sheet~1.

Contrast this with the proof of Theorem \ref{one}. There all we need to check  is that $\sum_x\rho^{-x}\pi(x,0,1)<\infty$ in order to restrict our calculations to trajectories remaining on sheet~1.
If there is a ray on sheet~1 we showed $\sum_x\rho^{-x}\pi(x,0,1)<\infty$ so  Theorems \ref{one}  holds. On the other hand, in the spiral-spiral case we have seen, at least empirically, that by (\ref{approx1}),
$\pi(\ell,0,1)\sim \rho^{\ell}\frac{\alpha(\ell)}{\ell}$  where
$\alpha(\ell)\sim C_1( 1- \frac{(a+c-b)}{c})$. It follows, at least empirically, that in the spiral-spiral case, $\sum_{\ell}\rho^{-\ell}\pi(\ell,0,1)=\infty$. Consequently the asymptotics in Theorem \ref{one}
can't hold and that's what we see by simulation.

\end{appendix}

\noindent {\bf Acknowledgments}

The authors wish to thank Professor Doug Down from the Department of Computing and Software at McMaster University for many conversations about polling models
when this work was started more than a decade ago. DMcD thanks Fran\c{c}ois Baccelli for the opportunity to present this work in the Inria Dyogene Seminar. The questions and references
helped to improve the paper. In this regard, thanks also to Sergey Foss from the School of Mathematics and Computer Sciences at the Heriot Watt University for his suggestions and references.


\begin{thebibliography}{99}

\baselineskip12pt
\bibitem{AdanI} \textsc{ Adan,I. J.-B. F.},  \textsc{Wessels,J.} and  \textsc{Zijm, W. H. M.} (1990). Analysis of the symmetric shortest queue
problem. Comm. Statist. Stochastic Models 6 691–713.
\bibitem{AdanII} \textsc{ Adan,I. J.-B. F.},  \textsc{Wessels,J.} and  \textsc{Zijm, W. H. M.}  (1993). A compensation approach for two-dimensional
Markov processes. Adv. in Appl. Probab. 25 783–817.

\bibitem{Barndorff}\textsc{Barndorff-Nielsen, O.} (1978) {\textit Information and Exponential Families in Statistical Theory}. Wiley, Chichester.
\bibitem{Boxma} \textsc{Boxma, O.J.},\textsc{Koole, G.M.} and \textsc{Mitrani,I.} (1995).
Polling Models with Threshold Switching.
 Quantative Methods in Parallel Systems, pages 129-140,
Springer-Verlag (Esprit Basic Research Series), A.

\bibitem{Borst}  \textsc{Borst, S.C.} and \textsc{Boxma, O.J.} (1997) ) Polling Models With and Without Switchover Times. Operations Research 45(4):536-543.
https://doi.org/10.1287/opre.45.4.536


\bibitem{Bramson} \textsc{Bramson, M.} (1994). Instability of FIFO queueing networks. Annals of Applied Probability,
4:414–431.

\bibitem{Dai} \textsc{Dai, J. G.} (1995). On positive Harris recurrence of multiclass queueing networks: A unified
approach via fluid limit models. Annals of Applied Probability, 5:49–77.

\bibitem{DelcoigneI} \textsc{Delcoigne, F.} and \textsc{de La Fortelle, A.} (2002). Large Deviations Rate Function for Polling Systems. Queueing Systems 41, 13–44.

\bibitem{DelcoigneII}\textsc{Delcoigne, F.} and \textsc{de La Fortelle, A.} (2000). Large deviations for polling systems, Technical Report 3892, INRIA.





\bibitem{doob}
\textsc{Doob, J. L.},  \textsc{Snell, J. L.} and \textsc{ Williamson, R. E.} (1960). Application of boundary theory to sums of
htdependent random variables, Contributions to Probability and Statistics, Stanford
Univ. Press., Stanford, Calif.,  182-197.
\bibitem{Chang-Down-2002} \textsc{Woojin Chang, W.} and \textsc{Down, D. G.} (2007). Exact Asymptotics for k i -limited Exponential Polling Models, Queueing Systems,
\textbf{42}, 401-419.
\bibitem{Chang-Down-2007} \textsc{Woojin Chang, W.} and \textsc{Down, D. G.} (2007). Polling Models Under limited
service policies: sharp asymptotics, Stochastic models, 23:1, 129-147.
\bibitem{Ellis-Dupuis}
 \textsc{Dupuis P.} and \textsc{Ellis R.S.}(1997).
 {\textit A Weak Convergence Approach to the Theory of Large Deviations,} Wiley, 504 p.p.
\bibitem{Dynkin}
{\sc Dynkin,  E. B.} (1969).
 Boundary Theory of Markov Processes (the Discrete Case). {\em Russian Mathematical Surveys}, Vol. 24,  2,  1-42.
 \bibitem{Asian}\textsc{ Feng,W.}, \textsc{Adachi,K.} and \textsc{Kowada, M.} (2007).
 Large deviations bounds for a polling system with markovian on/off sources and bernoulli service schedule.
 Scientiae Mathematicae Japonicae \textbf{65}, No. 2,  233-252.
 \bibitem{Odd}
 {\sc  Flajolet P. and Odlyzko A.} (1990). Singularity analysis of generating functions. SIAM J. Disc. Math., \textbf{3}, No. 2, 216-240.

 \bibitem{Ellis} {\sc  Ellis R.S.} (2006).
 {\textit Entropy, Large Deviations, and Statistical Mechanics (Classics in Mathematics),
 }
Springer.
\bibitem{Irina}
\textsc{Ignatiouk-Robert I.} (2008).
 Martin boundary of a killed random walk on a half-space. Journal of Theoretical
Probability, \textbf{21}, no. 1, 35-68.
\bibitem{JSQ}
\textsc{Foley, R.} and R. and \textsc{McDonald, D.} (2001). Join the Shortest Queue:
Stability and Exact Asymptotics, {\it Ann. Appl. Probab.} \textbf{11}, 569-607.

\bibitem{FM}
\textsc{Foley, R.} and R. and \textsc{McDonald, D.} ( 2005). Large deviations of a modified Jackson network: Stability and rough asymptotics. Ann. Appl. Probab. \textbf{15}, (1B) 519 - 541.

\bibitem{FM:BridgesExact}
\textsc{Foley, R.} and R. and \textsc{McDonald, D.} (2005). Bridges and networks: Exact asymptotics.
 Ann. Appl. Probab. \textbf{15} (1B), 542 - 586.

 \bibitem{Ivo} \textsc{Adan,Ivo J. B. F.}, \textsc{Foley, R.} and R. and \textsc{McDonald, D.} (2009). Exact asymptotics for the stationary distribution of a Markov chain: a production model. Queueing Syst., 311--344.
\bibitem{Jesse} \textsc{Collingwood, J.},\textsc{Foley, R.} and R. and \textsc{McDonald, D.} (2011). Networks with cascading overloads. QNTA. Proceedings of the 6th international conference on queueing theory and network applications. 33-37.

\bibitem{hennequin}
{\sc Hennequin,  P.L.} (1963). Processus de Markov en cascade.
Ann. Inst. H. Poincaré \textbf{18}(2), 109–196.

\bibitem{Hofri} \textsc{Hofri,M.} and  \textsc{Ross,  K.W.} (1987). On the optimal control of two queues with
server  setup times and its analysis, SIAM Journal on Computing, 16, 399-420.


\bibitem{Jacobsen} \textsc{Jacobsen, M.} (1984). Two operational characterizations of cooptional times. Ann. Probab. \textbf{12}, No. 3 714-725.
\bibitem{Julia} The Julia Programming Language. https://julialang.org/.
\bibitem{kesten} \textsc{Kesten, H.} (1995). A ratio limit theorem for (sub) Markov chains on  $\{1,2, \ldots \}$ with bounded jumps. Adv. Appl. Prob. \textbf{27}, 652-691.
\bibitem{PfH}
\textsc{McDonald, D.} (1999). Asymptotics of first passage times for
random walk in a quadrant, {\it Ann. Appl. Probab}. \textbf{9},
110--145.
\bibitem{MeynTweedie}
\textsc{Meyn, S.P.} and \textsc{Tweedie, R.L.} (1993). {\it Markov Chains and
Stochastic Stability.} Springer Verlag, New York.

\bibitem{Ney-Spitzer}
\textsc{Ney, P., and Spitzer, F.} (1966). The Martin Boundary for Random Walk. Transactions of the American Mathematical Society, \textbf{121}(1), 116–132.

\bibitem{rocky}
\textsc{Rockafellar, R.T.}  (1967).
Conjugates and Legendre transforms of convex functions, Canad. J. Math. 19, 200-205.

\bibitem{rockafellar}
\textsc{Rockafellar, R.T.} (1970). {\textit Convex Analysis.} Princeton
University Press.

\bibitem{Takacs} \textsc{Tak´acs, L.} (1968). Two queues attended by a single server, Operations Research, 16, 639-650.
\bibitem{Takagi} \textsc{Takagi,H.} (1986). {\textit Analysis of polling systems.} The MIT Press, Cambridge, MA, USA.

\bibitem{Raschel}
\textsc{Viet, V.H}, \textsc{Raschel, K.} and  \textsc{Tarrago, P.} (2023). Harmonic functions for singular quadrant walks.
To appear Indagationes Mathematicae.
\bibitem{weiss}
\textsc{Shwartz, A. and Weiss, A.} (1994). {\textit Large deviations for
performance analysis,} Chapman and Hall.



\end{thebibliography}
\end{document}